\documentclass{amsart}

\usepackage{amssymb}
\usepackage{color}
\usepackage{amsfonts}
\usepackage{amsmath}
\usepackage{euscript}
\usepackage{enumerate}
\usepackage{pdfsync}
\newtheorem{theorem}{Theorem}[section]
\newtheorem{lemma}[theorem]{Lemma}
\newtheorem{note}[theorem]{Note}
\newtheorem{prop}[theorem]{Proposition}
\newtheorem{cor}[theorem]{Corollary}

\newtheorem*{Theorem1'}{Theorem 1'}

\theoremstyle{definition}

\theoremstyle{remark}

\newcommand{\groupPresentation}[2]{\langle #1 \mid #2 \rangle}
\newcommand{\gen}[1]{\langle #1 \rangle}

\newcommand \GL{{\mathrm{GL}}}
\newcommand \Z{{\mathbb Z}}

\newcommand \N{{\mathbb N}}

\title[The automorphism group of finite $2$-groups associated to the Macdonald group]{The automorphism group of finite $2$-groups associated to the Macdonald group}

\author{Alexander Montoya Ocampo}
\address{Department of Mathematics and Statistics, University of Regina, Canada}
\email{alexandermontoya1996@gmail.com}

\author{Fernando Szechtman}
\address{Department of Mathematics and Statistics, University of Regina, Canada}
\email{fernando.szechtman@gmail.com}

\thanks{The second author was partially supported by NSERC grant RGPIN-2020-04062}

\subjclass[2020]{20D45, 20D15, 20D20}

\keywords{Automorphism Group, $2$-group, Macdonald Group, Deficiency Zero}

\begin{document}

\begin{abstract} We consider the Macdonald group 
$\langle x,y\,|\, x^{[x,y]}=x^{1+2^m\ell},\, y^{[y,x]}=y^{1+2^m\ell}\rangle$ and its Sylow 2-subgroup 
$J=\langle x,y\,|\, x^{[x,y]}=x^{1+2^m\ell},\, y^{[y,x]}=y^{1+2^m\ell}, x^{2^{3m-1}}=y^{2^{3m-1}}=1\rangle$,
where $m\geq 1$ and $\ell$ is odd. Then $J$ has order $2^{7m-3}$, and nilpotency class 5 if $m>1$ and 3 if $m=1$.
We determine the automorphism group of the 2-groups $J$, $H=J/Z(J)$ and $K=H/Z(H)$, where
$|H|=2^{6m-3}$ and $|K|=2^{5m-3}$. Explicit multiplication, power, and commutator formulas for $J$, $H$, and $K$ are given, and used
in the calculation of $\mathrm{Aut}(J)$, $\mathrm{Aut}(H)$, and $\mathrm{Aut}(K)$.
\end{abstract}

\maketitle

\section{Introduction}

We fix a positive integer $m$ and odd integers $\ell$ and $\alpha$, with $\alpha=1+2^m\ell$, throughout the paper,
and consider the groups
$$
J=J(\alpha)=\langle x,y\,|\, x^{[x,y]}=x^{\alpha},\, y^{[y,x]}=y^{\alpha}, x^{2^{3m-1}}=y^{2^{3m-1}}=1\rangle,
$$
$$
H=J/Z(J)=\langle x,y\,|\, x^{[x,y]}=x^{\alpha},\, y^{[y,x]}=y^{\alpha}, x^{2^{2m-1}}=y^{2^{2m-1}}=1\rangle,
$$
$$
K=H/Z(H)=\langle x,y\,|\, x^{[x,y]}=x^{\alpha},\, y^{[y,x]}=y^{\alpha}, x^{2^{2m-1}}=y^{2^{2m-1}}=[x,y]^{2^{m-1}}=1\rangle.
$$
As shown in \cite{MS}, $J$, $H$, and $K$ have order $2^{7m-3}$, $2^{6m-3}$, and $2^{5m-3}$,
and nilpotency class 5, 4, and 3 if $m>1$, and 3, 2, and 1 if $m=1$. In this paper we determine their automorphism groups.

Given a group $T$ and $i\geq 0$, we let $\langle 1\rangle=Z_0(T),Z_1(T),Z_2(T),\dots$ stand for the terms of the upper central series of~$T$, so that $Z_{i+1}(T)/Z_i(T)$ is the center of $T/Z_i(T)$, and we write $\mathrm{Aut}_i(T)$ for the kernel of the canonical map
$\mathrm{Aut}(T)\to \mathrm{Aut}(T/Z_i(T))$.

In Theorem \ref{teoremafinal}, for $m>1$, we determine the isomorphism type of each factor of the series 
$$
1\subset \mathrm{Aut}_1(J)\subset\mathrm{Aut}_2(J)\subset \mathrm{Aut}_3(J)\subset
\mathrm{Inn}(J)\mathrm{Aut}_3(J)\subset\mathrm{Aut}_4(J)\subset \mathrm{Aut}(J)
$$
of normal subgroups of $\mathrm{Aut}(J)$. The order of $\mathrm{Aut}(J)$ is obtained as an immediate consequence. It turns out
that $|\mathrm{Aut}(J)|=2^{8m}$ if $m>2$ and $|\mathrm{Aut}(J)|=2^{15}$ if $m=2$. We also find 8 explicit generators for $\mathrm{Aut}(J)$
when $m>2$, and 7 when $m=2$. 

The corresponding results for $\mathrm{Aut}(H)$ and $\mathrm{Aut}(K)$ are derived in Sections \ref{autgmodh} and \ref{autgmod2}, respectively.

The structure of $J$ is analyzed in \cite{MS}, in wider generality, where we find that $J$ is the Sylow 2-subgroup of the finite, nilpotent, Macdonald group \cite{M}
$$
G=\langle x,y\,|\, x^{[x,y]}=x^{\alpha},\, y^{[y,x]}=y^{\alpha}\rangle.
$$
The automorphism group of all remaining Sylow subgroups of $G$ are studied in \cite{MS2}, but for an exceptional case.

As a finite group with 2 generators and 2 relations, $G$ is
a finite group of deficiency zero. Such groups have attracted considerable attention
since Mennicke \cite{Me} found the first finite group requiring 3 generators and 3 relations,
namely $M(a,b,c)=\langle x,y,z\,|\, x^y=x^a, y^{z}=y^b, z^x=z^c\rangle$, where $|a|,|b|,|c|\geq 2$. 
His group was investigated in depth in \cite{A,AA,Ja,JR,S,W} but its order is still known
in certain cases only. Following Mennicke's paper, other examples
of finite groups of deficiency zero were encountered in
\cite{Al,AS,AS2,CR,CR2, CRT, J, K, M, P, R, W2, W3, W4, W5}. 

One motivation for our study of $\mathrm{Aut}(J)$ was the detailed and recent studies of the automorphism groups of other families of finite groups of prime power order, as found in \cite{BC, C, C2, Ma, Ma2}, for
instance, all of which are concerned with metacyclic groups. The structure of $J$ is considerably more complicated. 
A second incentive for our investigation of $\mathrm{Aut}(J)$ is to apply it to
the question that arises by replacing $\alpha$ with another integer $\alpha'=1+2^m\ell'$ with $\ell'$ odd: when are 
the finite 2-groups $J(\alpha)$ and $J(\alpha')$ isomorphic? This turns
out to be a nontrivial isomorphism problem, settled in \cite{MS3},
where the automorphism groups of $J$, $H$, and~$K$ play a critical role. In 
analyzing $\mathrm{Aut}(J)$ we found the structure of the factors $\mathrm{Aut}(J)/\mathrm{Aut}_4(J)$ and
$\mathrm{Aut}_4(J)/\mathrm{Inn}(J)\mathrm{Aut}_3(J)$ so challenging that we felt compelled to
derive explicit multiplication, power, and commutator formulas for~$J$, as exhibited in Section \ref{section.formulas.J}.
Our commutator formula from Theorem \ref{thm.commutators.J} plays a key role in the structure of the Wamsley
groups $
W_{\pm}(\beta,\gamma,\delta)=\langle X,Y,Z\,|\, X^Z=X^\beta, Y^{Z^{\pm 1}}=Y^\gamma, Z^\delta=[X,Y]\rangle,
$
as defined in \cite{W2}. See \cite{PS} for details of this application. As very little is known at the moment
about the Wamsley
groups, this application provides a third reason for the existence of this paper.
The proofs of the multiplication, power, and commutator formulas are fairly long, so we placed them in an appendix at 
the end of the paper (Theorem \ref{thm.commutators.J} is proven in more generality than the other formulas
so that it can be applied in \cite{PS}, which deals with all Sylow subgroups of the Wamsley groups).
These formulas as well as the orders of the automorphism groups of $J$, $H$, and $K$
agree with the output produced by GAP and Magma, respectively, for all tested values of $\alpha$.
A final reason
to produce this paper is that our strategy to approach $\mathrm{Aut}(T)$, when $T$ is $J$, $H$, or~$K$,
may be of use
to study the automorphism groups of other finite nilpotent groups $T$. This strategy
essentially consists of considering the series of normal subgroups of $\mathrm{Aut}(T)$
\begin{equation}
\label{norseraca}
\mathrm{Aut}_0(T)=1\subseteq \mathrm{Aut}_1(T)\subseteq\dots\subseteq \mathrm{Aut}_c(T)=\mathrm{Aut}(T),
\end{equation}
where $c$ is the nilpotency class of $T$, and understanding the factors $\mathrm{Aut}_{i+1}(T)/\mathrm{Aut}_i(T)$ of (\ref{norseraca}) 
for all $0\leq i<c$. An imbedding tool, namely Proposition \ref{zi2}, 
allows us to derive information from the foregoing sections of (\ref{norseraca}) arising from $T/Z(T)$ to those arising from $T$.
Thus, we begin our work with $T=K=H/Z(H)$, continue with $T=H=J/Z(J)$, and culminate it with $T=J$. We next describe
how this process is realized and assume for the remaining of this section that $m>1$.

For $T=K$, we have $c=3$, with $\mathrm{Aut}_{1}(K)\cong Z_1(K)^2$ and 
$\mathrm{Aut}_{2}(K)/\mathrm{Aut}_{1}(K)\cong (Z_2(K)/(Z_1(K))^2$ via 
natural maps, as indicated in Propositions \ref{autk} and \ref{autexcel}. Here 
$Z_1(K)\cong (\Z/2^{m-1}\Z)^2$ and $Z_2(K)/Z_1(K)\cong (\Z/2\Z)^2\times\Z/2^{m-1}\Z$, as described in Section \ref{autgmod2}.
Regarding $\mathrm{Aut}(K)/\mathrm{Aut}_2(K)$, we can view $K/Z_2(K)$ as a free module of rank 2 over $\Z/2^{m-1}\Z$.
As $\mathrm{Aut}_2(K)$ is the kernel of the natural map $\mathrm{Aut}(K)\to \mathrm{Aut}(K/Z_2(K))$,
we can see $\mathrm{Aut}(K)/\mathrm{Aut}_2(K)$ as a subgroup of $\GL_2(\Z/2^{m-1}\Z)$. This allows us
to show in Theorem \ref{autk6} that $\mathrm{Aut}(K)/\mathrm{Aut}_2(K)\cong ((\Z/2^{m-1}\Z)^\times \times (\Z/2\Z)^2)\rtimes \Z/2\Z$ if $m>2$, with $\Z/2\Z$
acting by inversion on $(\Z/2^{m-1}\Z)^\times$ and by switching factors on $(\Z/2\Z)^2$, and $\mathrm{Aut}(K)/\mathrm{Aut}_2(K)\cong
\GL_2(\Z/2\Z)$ if $m=2$. 

For $T=H$, we have $c=4$, and $\mathrm{Aut}_{2}(H)\cong Z_2(H)\times Z_2(H)$
via a natural map, as indicated in Proposition \ref{auth}. Here $Z_2(H)\cong (\Z/2^{m-1}\Z)^2\times \Z/2^{m}\Z$, as given in Section 
\ref{autgmodh}. The factor $\mathrm{Aut}_{3}(H)/\mathrm{Aut}_{2}(H)$ is difficult to compute directly,
so we determine $\mathrm{Inn}(H)\mathrm{Aut}_2(H)/\mathrm{Aut}_2(H)$ and $\mathrm{Aut}_3(H)/\mathrm{Inn}(H)\mathrm{Aut}_2(H)$
instead. General principles easily yield $\mathrm{Inn}(H)\mathrm{Aut}_2(H)/\mathrm{Aut}_2(H)\cong (\Z/2^{m-1}\Z)^2$,
and we have $\mathrm{Aut}_3(H)/\mathrm{Inn}(H)\mathrm{Aut}_2(H)\hookrightarrow \mathrm{Aut}_2(K)/\mathrm{Inn}(K)\mathrm{Aut}_1(K)$
by the imbedding tool of Proposition \ref{zi2}.
The latter group is isomorphic to $(\Z/2\Z)^4$ by Proposition \ref{outow}. This is used in Theorem \ref{tamanio}
to prove that $\mathrm{Aut}_3(H)/\mathrm{Inn}(H)\mathrm{Aut}_2(H)$ is the Klein 4-group. The last factor, namely
$\mathrm{Aut}(H)/\mathrm{Aut}_3(H)$, imbeds into $\mathrm{Aut}(K)/\mathrm{Aut}_2(K)$, again by Proposition \ref{zi2},
where $\mathrm{Aut}(K)/\mathrm{Aut}_2(K)$ was described above. This description allows us to prove in Theorem \ref{gpexth}  that
$\mathrm{Aut}(H)/\mathrm{Aut}_3(H)$ is isomorphic to the Klein 4-group if $m>2$ and to $\Z/2\Z$ if $m=2$.

For $T=J$, we have $c=5$, with $\mathrm{Aut}_1(J)\cong (\Z/2^m\Z)^2$ and
$\mathrm{Aut}_2(J)/\mathrm{Aut}_1(J)\cong (\Z/2^m\Z)^2$ via natural maps, as given in Propositions \ref{autj1} and \ref{autjo}.
Applying Proposition \ref{zi2} twice and appealing to the foregoing description of $\mathrm{Aut}_1(K)$, we have
$$
\mathrm{Aut}_3(J)/\mathrm{Aut}_2(J)\hookrightarrow \mathrm{Aut}_2(H)/\mathrm{Aut}_1(H)\hookrightarrow
\mathrm{Aut}_1(K)\cong Z_1(K)\times Z_1(K)\cong (\Z/2^{m-1}\Z)^4.
$$
This allows us in Theorem \ref{autj3} to compute $\mathrm{Aut}_3(J)/\mathrm{Aut}_2(J)\cong(\Z/2^{m-1}\Z)^2\times \Z/2\Z$
as well as $\mathrm{Inn}(J)\mathrm{Aut}_3(J)/\mathrm{Aut}_3(J)\cong (\Z/2^{m-1}\Z)^2$. We have
$\mathrm{Aut}(J)/\mathrm{Aut}_4(J)\hookrightarrow \mathrm{Aut}(H)/\mathrm{Aut}_3(H)$  and
$\mathrm{Aut}_4(J)/\mathrm{Inn}(J)\mathrm{Aut}_3(J)\hookrightarrow \mathrm{Aut}_3(H)/\mathrm{Inn}(H)\mathrm{Aut}_2(H)$
by Proposition \ref{zi2}.
Our prior work on $\mathrm{Aut}(H)$ allows us in Theorem \ref{teoremafinal} to show that
$\mathrm{Aut}_4(J)/\mathrm{Inn}(J)\mathrm{Aut}_3(J)\cong \Z/2\Z$ as well as the fact that $\mathrm{Aut}(J)/\mathrm{Aut}_4(J)$
is the Klein 4-group if $m>2$ and $\Z/2\Z$ if $m=2$.

Regarding notation, given a group $T$, 
we write
$[a,b]=a^{-1}b^{-1}ab,\; b^a=a^{-1}ba,\; {}^a b=aba^{-1}$, for $a,b\in T$, recalling that
\begin{equation}\label{comfor} [a,bc]=[a,c][a,b]^c,\; [bc,a]=[b,a]^c\; [c,a].
\end{equation}
The order of a torsion element $a$ of $T$ will be denoted by $o(a)$. If $S$ is a normal subgroup of $T$,
we sometimes write $\overline{T}$ for $T/S$ and $\overline{t}$ for $tS\in \overline{T}$. We let $\delta:T\to\mathrm{Aut}(T)$ stand for the canonical map $a\mapsto a\delta$, where $a\delta$ is conjugation by $a$,
namely the map $b\mapsto b^a$. Function composition proceeds from left to right.
The automorphisms of $T$ belonging to $\mathrm{Aut}_i(T)$ will be said to be $i$-central,
while those in $\mathrm{Aut}_1(T)$ will simply be said to be central. Note that central
and inner automorphisms commute with each other. Observe also that for $a\in T$, $a\delta\in \mathrm{Aut}_i(T)$ if and only if
$a\in Z_{i+1}(T)$, so that $\mathrm{Inn}(T)\cap \mathrm{Aut}_i(T)=Z_{i+1}(T)\delta\cong Z_{i+1}(T)/Z(T)$. If $n\in\N$, we let $T^n$ stand for the direct product of $n$ copies of $T$.

\section{Background on the Sylow 2-subgroup of the Macdonald group}\label{backJ}

We set
\begin{equation}\label{presn}
    J = \groupPresentation{A,B}{A^{[A,B]}=A^\alpha,\; B^{[B,A]}=B^\alpha,\; A^{2^{3m-1}}=1,\; B^{2^{3m-1}}=1},\quad
    C = [A,B].
\end{equation}
Note the existence of an automorphism $\theta$ of $J$ satisfying $A\leftrightarrow B$, $C\leftrightarrow C^{-1}$.
The following facts were proven in \cite{MS}, as indicated below.

The group $J$ has order $2^{7m-3}$ and every element of $J$ can be written uniquely in the form \cite[Theorem 7.1]{MS}
\begin{equation}\label{gol}
    A^i B^j C^k,\quad 0\leq i<2^{3m-1},\;
    0\leq j<2^{2m-1},\; 0\leq k<2^{2m-1}.
\end{equation}

Moreover $A$, $B$,  $C$ have orders \cite[Proposition 7.2]{MS}
\begin{equation}\label{gol2}
    o(A)=2^{3m-1},\;
    o(B)=2^{3m-1},\;
    o(C)=2^{2m},
\end{equation}
and the following relations hold \cite[Section 6]{MS}:
\begin{equation}\label{put}
    A^{2^{2m-1}}B^{2^{2m-1}}=1,\;
    A^{2^{3m-2}}=B^{2^{3m-2}}=C^{2^{2m-1}}.
\end{equation}

If $m>1$ \cite[Proposition 8.1]{MS}, then $J$ is nilpotent of class 5, where
\begin{equation}\label{put3}
    Z_1(J)=\langle A^{2^{2m-1}}\rangle,\;
    Z_2(J)=\langle  A^{2^{2m-1}},  C^{2^{m-1}}\rangle,
\end{equation}
\begin{equation}\label{put4}
    Z_3(J)=\langle A^{2^{m}},  B^{2^{m}}, C^{2^{m-1}}\rangle,\;
    Z_4(J)=\langle  A^{2^{m-1}},  B^{2^{m-1}}, C\rangle,
\end{equation}
{$Z_3(J)$ is an abelian group of order $2^{4m-2}$, and $J$ has exponent $2^{3m-1}$ \cite[Proposition 9.2]{MS}.

If $m=1$ \cite[Proposition 8.2]{MS}, then $J$ is nilpotent of class 3, and isomorphic to the generalized quaternion group
$Q= \groupPresentation{u,v}{u^4=v^2,\; u^v=u^{-1}}$ of order 16. We thus have
$\mathrm{Aut}(Q)=\mathrm{Hol}(\Z/8\Z)$, the holomorph of $\Z/8\Z$;
$Q/Z_1(Q)\cong D_8$, the dihedral group of order 8, with $\mathrm{Aut}(D_8)=D_8$; 
$Q/Z_2(Q)\cong (\Z/2\Z)^2$, with $\mathrm{Aut}((\Z/2\Z)^2)=\GL_2(\Z/2\Z)$.

\section{Structural formulas for $J$}\label{section.formulas.J}

We fix the following integers for the remainder of the paper: $s = 2^{m-1}$, $u = s^2$, as well as $r = s/2$ provided $m>1$.
In this notation, we have $o(A) = 4us = o(B)$, $o(C)= 4u$, $A^{2u} B^{2u} = 1$, $A^{2us} = C^{2u} = B^{2us}$, and $Z_1(J) = \langle A^{2u}\rangle$.

For all $n\in\Z$, define
\[
\phi(n) = \frac{(n-1)n}{2},
\qquad\varphi(n) = \frac{n(n-1)(n-2)}{6}.
\]

All statements below are demonstrated in the Appendix.

\begin{theorem}\label{thm.commutators.J}
    For all $n,t\in\Z$ the following identities hold in $J$:
    \begin{align}
        &[C^n,A^t]
        = A^{-2s\ell nt - 4u\ell^2\phi(n)t},\label{eq.comj.2}\\
        &[C^n,B^t]
        = B^{2s\ell nt - 4u\ell^2\phi(n+1)t},\label{eq.comj.3}\\
        &[A^n,B^t]
        = A^{\textup{exp} A} B^{\textup{exp} B} C^{\textup{exp} C} A^{\xi_1},\label{eq.comj.4}
    \end{align}
    where $\textup{exp} A = -2s\ell\phi(n)t$, $\textup{exp} B = 2s\ell n\phi(t)$, $\textup{exp}\,C = nt - 2s\ell\phi(n)\phi(t)$ and
    \[
    \xi_1(n,t)
    = 2u\ell^2\{2\varphi(n+1)t + (2n-7)\phi(n)\phi(t) - 2n\phi(t) - (3n+1)n\varphi(t)\}.
    \]
\end{theorem}

\begin{theorem}\label{thm.product.J}
    Let $i,j,k,a,b,c\in\Z$. Then
    $
    (A^iB^jC^k)(A^aB^bC^c)
    = A^{\textup{exp} A} B^{\textup{exp} B} C^{\textup{exp} C} A^{\xi_2},
    $
    where
    \begin{align*}
        \textup{exp} A
        &= i + a + 2s\ell\{j\phi(a) - ka\},\\
        \textup{exp} B
        &= j + b + 2s\ell\{kb - jab - \phi(j)a\},\\
        \textup{exp}\, C
        &= k + c - ja + 2s\ell\{jka - \phi(j+1)\phi(a)\},\\
        \xi_2(j,k,a,b)
        = 2u\ell^2 \{\phi(j)\phi(a)(-2j + &2a - 2b + 5)
        - 2\phi(j)a(j + k + b - 1)
        - j\phi(a)(4k - 2a + 1)\\
        + 2\phi(k)(a - b)
        &- ka(j - 2)
        - 2j(ab + \varphi(a+1))
        + (3a + 1)a\varphi(j)\}.
    \end{align*}
\end{theorem}

\begin{cor}\label{cor.inverse.J}
    Let $a,b,c\in\Z$. Then $(A^aB^bC^c)^{-1} = A^{\textup{exp} A} B^{\textup{exp} B} C^{\textup{exp} C} A^{\xi_3}$, where
    \begin{align*}
        \textup{exp} A
        &= -a - 2s\ell(\phi(a+1)b + ac),\\
        \textup{exp} B
        &= -b + 2s\ell(a\phi(b+1) + bc),\\
        \textup{exp}\, C
        &= -c - ab - 2s\ell\phi(a+1)\phi(b),\\
        \xi_3
        = 4u\ell^2\{a\phi(b) - \phi(a)b\}c + \xi_2&(-b,0,-a,0)
        + \xi_2(0,-c,-a-2s\ell\phi(a+1)b,-b+2s\ell a\phi(b+1)).
    \end{align*}
\end{cor}

\begin{prop}\label{prop.bigCommutator.J}
    Let $i,j,k,a,b,c\in\Z$. Then
    $
    [A^iB^jC^k,A^aB^bC^c]
    \equiv A^{\textup{exp} A} B^{\textup{exp} B}C^{\textup{exp} C}\mod Z_1(J),
    $
    where
    \begin{align*}
        \textup{exp} A
        &= 2s\ell\{j\phi(a) - \phi(i)b + ic - ka\},\\
        \textup{exp} B
        &= 2s\ell\{i\phi(b) - \phi(j)a + kb + j(ib - ab - c)\},\\
        \textup{exp}\, C
        &= ib - ja + 2s\ell\{\phi(a)(\phi(j) + jb) - \phi(i)(\phi(b) + jb) + ijc - kab\}.
    \end{align*}
\end{prop}

\begin{theorem}\label{thm.power.J}
    Let $a,b,c,n,k,t\in\Z$. Then
    $
    (A^aB^bC^c)^n
    = A^{\textup{exp} A} B^{\textup{exp} B} C^{\textup{exp} C} A^{\xi_4},
    $
    where
    \begin{align*}
        \textup{exp} A
        &= na + 2s\ell\{a^2b\varphi(n) + (\phi(a)b - ac)\phi(n)\},\\
        \textup{exp} B
        &= nb + 2s\ell\{(b(c - ab) - a\phi(b))\phi(n) - 2ab^2\varphi(n)\},\\
        \textup{exp}\, C
        &= nc - ab\phi(n) + 2s\ell\{a^2\phi(b)\varphi(n) + \phi(a)\phi(b)\phi(n) - a^2b^2\sigma_2(1,n) - (\phi(a)b-ac)b\varphi(n+1)\},
    \end{align*}
    \begin{align*}
        \xi_4
        &= 2u\ell^2\{a(\varphi(b)+2\phi(b)-2\phi(c))\phi(n)
        + a^2(3\varphi(b)-2b\phi(b))(2\varphi(n)+\phi(n))\\
        &\qquad\qquad + 2a\phi(b)c\varphi(n) - 2a^2b\phi(b)(\sigma_1(1,n)-\varphi(n))\\
        &\qquad\qquad + (2bc+7\phi(b))(a^2\varphi(n) + \phi(a)\phi(n))\\
        &\qquad\qquad - 2\phi(b)(a+b)(\sigma_1(a,n) - a\varphi(n)) - 2b\sigma_2(a,n)\\
        &\qquad\qquad + (2\phi(b)-2c-b)(a^2b\sigma_2(1,n) + (\phi(a)b-ac)\varphi(n+1))\\
        &\qquad\qquad + 2b(a(\phi(a)b-ac)+b^2(\phi(a)-a^2))(\sigma_1(1,n) + \varphi(n+1))\\
        &\qquad\qquad + 2a^2b^2(a-b)(\sigma_3(n) + \sigma_1(1,n) - \varphi(n))\\
        &\qquad\qquad - 2c((c-ab)b\varphi(n+1) - 2ab^2\sigma_2(1,n))\\
        &\qquad\qquad + 2b(\sigma_1(ab,n) + (\phi(c+1) - abc)\phi(n) - ab(2c+1)\varphi(n))\},
    \end{align*}
    \[
    \sigma_1(k,t) = \frac{k\phi(t)(k\phi(t)-1)}{2},\quad
    \sigma_2(k,t) = \frac{k\phi(t)(k^2\phi(t)-1)}{6},
    \]
    \[
    \sigma_3(t) = \frac{1}{6}\left(\frac{(t-1)t(6(t-1)^3+9(t-1)^2+t-2)}{30} - 3\phi(t)^2 + 4\varphi(t) + 2\phi(t)\right).
    \]
\end{theorem}

\section{Order of the terms of the upper central series of $J$}\label{sizeJ}

Unless otherwise stated we assume from now on that $m>1$. It follows from \eqref{gol2} and \eqref{put3} that $|Z_1(J)|=2^m$.
By \eqref{presn} and \eqref{put3}, we have $Z_2(J)=\langle  A^{2^{2m-1}}\rangle \langle  C^{2^{m-1}}\rangle$,
so \eqref{gol}, \eqref{gol2}, and \eqref{put} yield $|Z_2(J)|=2^{2m}$. By~\eqref{presn}, there is an epimorphism
$f:J\to (\Z/{2^{m-1}}\Z)^2$, whose kernel contains $Z_4(J)$ by \eqref{put4}.
As $|J|=2^{7m-3}$, we infer $|\ker(f)|=2^{5m-1}$. But \eqref{gol} ensures $|Z_4(J)|\geq 2^{5m-1}$,
so $Z_4(J)=\ker(f)$ has order $2^{5m-1}$. The order of $Z_3(J)$ was indicated in Section \ref{backJ}, so we have
\begin{equation}\label{put5}
    \begin{split}
        |Z_1(J)|=2^{m},\;
        |Z_2(J)|=2^{2m},\;&
        |Z_3(J)|=2^{4m-2},\;
        |Z_4(J)|=2^{5m-1},\\
        |J|=2^{7m-3},\;&
        |\mathrm{Inn}(J)|=2^{6m-3}.
    \end{split}
\end{equation}

\section{The automorphism group of $J/Z_2(J)$}\label{autgmod2}

Set $K =J/Z_2(J)$. It follows from \eqref{presn}, \eqref{put3}, and \cite[Proposition 5.1]{MS} that $K$ has presentation
\begin{equation}\label{put6}
    K = \groupPresentation{a,b}{a^{[a,b]}=a^\alpha,\; b^{[b,a]}=b^\alpha,\; a^{2^{2m-1}}=1,\; b^{2^{2m-1}}=1,\; [a,b]^{2^{m-1}}=1},
\end{equation}
and we set $c = [a,b]$. We have an automomorphism $a\leftrightarrow b$, $c\leftrightarrow c^{-1}$, say $\mu$, of $K$.

We deduce from \eqref{put3} and \eqref{put4} that
\begin{equation}\label{put7}
    Z_1(K)=\langle a^{2^{m}}, b^{2^{m}}\rangle,\;
    Z_2(K)=\langle a^{2^{m-1}}, b^{2^{m-1}}, c\rangle.
\end{equation}
The indicated orders from \eqref{put5} yield
$$
|Z_1(K)|=2^{2(m-1)},\;
|Z_2(K)|=2^{3m-1},\;
|K|=2^{5m-3},\;
|\mathrm{Inn}(K)|=2^{3m-1}.
$$

We may now deduce from (\ref{gol}) that every element of $K$ can be written uniquely in the form
$a^i b^j c^k$, where $0\leq i,j<2^{2m-1}$ and $0\leq k <2^{m-1}$, and that $o(a)=o(b)=2^{2m-1}$
and $o(c)=2^{m-1}$.

\begin{prop}\label{autk}
    For every $x,y\in Z_1(K)$ the assignment $a\mapsto ax,\; b\mapsto by$ extends to a central automorphism $\Omega_{(x,y)}$ of $K$ that fixes $Z_2(K)$ pointwise. Moreover, the corresponding map, say
    $\Omega:Z_1(K)\times Z_1(K)\to \mathrm{Aut}(K)$, is a group monomorphism whose image is $\mathrm{Aut}_1(K)$. In particular,
    $|\mathrm{Aut}_1(K)|=2^{4(m-1)}$.
\end{prop}

\begin{proof}
    Let $x,y\in Z_1(K)$. Then $[ax,by]=[a,b]=c$ by \eqref{comfor}. As $Z_1(K)$ has exponent $2^{m-1}$,
    the defining relations of $K$ are preserved. This yields
    an endomorphism $\Omega_{(x,y)}$ of $K$. Note that $\Omega_{(x,y)}$ fixes $Z_2(K)$ pointwise, because $c\mapsto [ax,by]=c$, and
    $$
    a^{2^{m}}\mapsto (ax)^{2^{m}} = a^{2^{m}}x^{2^{m}} = a^{2^{m}},\;
    b^{2^{m}}\mapsto (by)^{2^{m}} = b^{2^{m}}y^{2^{m}} = b^{2^{m}}.
    $$

    It is now easy to see that $\Omega_{(x,y)}\circ \Omega_{(x^\prime,y^\prime)}=\Omega_{(x,y)(x^\prime,y^\prime)}$. As $\Omega_{(1,1)}=1_K$,
    each $\Omega_{(x,y)}$ is an automorphism of $K$. It is clear that $\Omega_{(x,y)}$ is trivial only
    when $(x,y)$ is trivial. Finally, by definition, every central automorphism of $K$ must be of the form $\Omega_{(x,y)}$
    for some $x,y\in Z_1(K)$.
\end{proof}

\begin{prop}\label{autk2}
    We have
    $\mathrm{Inn}(K)\cap\mathrm{Aut}_1(K)=\langle c\delta, a^{2^{m-1}}\delta, b^{2^{m-1}}\delta\rangle$,
    $|\mathrm{Inn}(K)\cap\mathrm{Aut}_1(K)|=2^{m+1}$,
    $$
    |\mathrm{Inn}(K)\mathrm{Aut}_1(K)|=2^{6m-6},\; \mathrm{Inn}(K)\mathrm{Aut}_1(K)\subset \mathrm{Aut}_2(K).
    $$
\end{prop}

\begin{proof}
    Given that $Z_2(K)/Z_1(K)$ is generated by the cosets of $a^{2^{m-1}}$, $b^{2^{m-1}}$, $c$ and has order $2^{m+1}$,
    we deduce that $Z_2(K)/Z_1(K)\cong\mathrm{Inn}(K)\cap\mathrm{Aut}_1(K)=
    \langle a^{2^{m-1}}\delta, b^{2^{m-1}}\delta, c\delta\rangle$ has order $2^{m+1}$.

    We infer from Proposition \ref{autk} that $|\mathrm{Inn}(K)\mathrm{Aut}_1(K)|=2^{6m-6}$.
    On the other hand, it is obvious  that $\mathrm{Aut}_1(K)\subset\mathrm{Aut}_2(K)$.
    As $b^a=bc^{-1}$, $a^b=ac$, with $c\in Z_2(K)$, we deduce that $\mathrm{Inn}(K)$ is included in $\mathrm{Aut}_2(K)$ as well.
\end{proof}

Recalling the meaning of $s$ and $u$, as given in Section \ref{section.formulas.J}, we have $a^{2u} = 1 = b^{2u}$, $c^s = 1$, $Z_1(K) = \gen{a^{2s},b^{2s}}$, and $Z_2(K) = \gen{a^s,b^s,c}$.

The following three corollaries are consequences of the main results of Section \ref{section.formulas.J}.

\begin{cor}\label{cor.commutator.K}
    For all $n,t\in\Z$ the following identities hold in $K$:
    \begin{align*}
        [c^n,a^t]
        &= a^{-2s\ell nt},\\
        [c^n,b^t]
        &= b^{2s\ell nt},\\
        [a^n,b^t]
        &= a^{-2s\ell\phi(n)t} b^{2s\ell n\phi(t)} c^{nt - 2s\ell\phi(n)\phi(t)}.
    \end{align*}
\end{cor}

\begin{cor}\label{cor.bigCommutator.K}
    Let $i,j,k,n,t,q\in\Z$. Then $[a^ib^jc^k,a^nb^tc^q]\equiv c^{it - jn}\mod Z_1(K)$.
\end{cor}

\begin{note}\label{note.K.1}
    Given $x,y\in K$, there exist $t\in\Z$ and $z\in Z_1(K)$ such that $[x,y]^s = (c^t z)^s$, so $[x,y]^s = 1$.
\end{note}

\begin{cor}\label{cor.power.K}
    Let $i,j,k,n\in\Z$. Then $(a^ib^jc^k)^n = a^{\textup{exp} a} b^{\textup{exp} b} c^{\textup{exp} c}$, where
    \begin{align*}
        \textup{exp}\,a
        &= ni + 2s\ell\{i^2j\varphi(n) + (\phi(i)j - ik)\phi(n)\},\\
        \textup{exp}\,b
        &= nj + 2s\ell\{(j(k - ij) - i\phi(j))\phi(n) - 2ij^2\varphi(n)\},\\
        \textup{exp}\,c
        &= nk - ij\phi(n).
    \end{align*}
\end{cor}

\begin{note}\label{note.K.2}
    If $s\mid\phi(n)$, then $s\mid\varphi(n)$ and $(a^ib^jc^k)^n = a^{ni} b^{nj} c^{nk}$. Thus, the exponent of $K$ is $2u$.
\end{note}

It follows from Notes \ref{note.K.1} and \ref{note.K.2} that for any $x,y\in K$, the assignment
$a\mapsto x, b\mapsto y$ extends to an endomorphism of $K$ if and only if it preserves the first and second defining relations of $K$.

\begin{cor}\label{cor.product.K}
    Let $i,j,k,n,t,q\in\Z$. Then $(a^ib^jc^k)(a^nb^tc^q) = a^{\textup{exp} a} b^{\textup{exp} b} c^{\textup{exp} c}$, where
    \begin{align*}
        \textup{exp}\,a
        &= i + n + 2s\ell\{j\phi(n) - kn\},\\
        \textup{exp}\,b
        &= j + t + 2s\ell\{kt - jnt - \phi(j)n\},\\
        \textup{exp}\,c
        &= k + q - jn.
    \end{align*}
\end{cor}

\begin{prop}\label{autexcel}
    For any $x,y\in Z_2(K)$, the assignment $a\mapsto ax,\; b\mapsto by$ extends to a 2-central automorphism $\Gamma_{(x,y)}$ of $K$ that fixes $Z_1(K)$ and $Z_2(K)/Z_1(K)$ pointwise. Thus, $|\mathrm{Aut}_2(K)|=2^{6m-2}$ and the map $g:Z_2(K)\times Z_2(K)\to \mathrm{Aut}_2(K)/\mathrm{Aut}_1(K)$, given by $(x,y)\mapsto \Gamma_{(x,y)}\mathrm{Aut}_1(K)$ is a group epimorphism with kernel $Z_1(K)^2$, so that $\mathrm{Aut}_2(K)/\mathrm{Aut}_1(K)\cong (Z_2(K)/Z_1(K))^2$.
\end{prop}

\begin{proof}
    Let $x,y\in Z_2(K)$. Then $x = a^{si} b^{sj} c^k$ and $y = a^{sa} b^{sb} c^c$ for some $i,j,k,n,t,q\in\Z$.

    By \eqref{comfor}, we have
    \begin{equation}\label{eq.K.1}
        [ax,by]\equiv[a,b]\mod Z_1(K).
    \end{equation}

    Since $Z_2(K)$ is abelian, $[x,c] = 1$. Then
    \[
    (ax)^{[ax,by]}
    = (ax)^c
    = a^c a^{si} b^{sj} c^k
    = a^{\alpha + si} b^{sj} c^k.
    \]

    As $s\mid\phi(\alpha)$, Note \ref{note.K.1} gives
    \[
    (ax)^\alpha
    = (a^{1+si} b^{sj} c^k)^\alpha
    = a^{\alpha(1+si)} b^{\alpha sj} c^{\alpha k}
    = a^{\alpha + si} b^{sj} c^k.
    \]

    Thus $(ax)^{[ax,by]} = (ax)^\alpha$. By the automorphism $a\leftrightarrow b$, $(by)^{[by,ax]} = (by)^\alpha$. Thus all defining relations of $K$ are preserved and the given assignment extends to an endomorphism $\Gamma_{(x,y)}$ of $K$.

    Since $s\mid\phi(2sa)$, Note \ref{note.K.2} gives
    \[
    a^{2s}\Gamma_{(x,y)}
    = (a^{1+si} b^{sj} c^k)^{2s}
    = a^{2s(1+si)} b^{2uj} c^{2sk}
    = a^{2s},
    \]
	and likewise for $b$, so $\Gamma_{(x,y)}$ fixes $Z_1(K)$ and the nilpotency of $K$ ensures $\Gamma_{(x,y)}$ is an automorphism.

    By Corollary \ref{cor.power.K} there exists $z_1\in Z_1(K)$ such that
    \[
    a^{s}\Gamma _{(x,y)}
    = (a^{1+si} b^{sj} c^k)^s
    = a^{s(1+si)} b^{uj} c^{sk - (1+si)sj\phi(s)} z_1
    = a^s (a^{ui} b^{uj} z_1).
    \]

    Applying the automorphism $a\leftrightarrow b$ to the above equality yields $(by)^s = b^s (b^{un} a^{ut} z_2)$ for some $z_2\in Z_1(K)$, so $b^s\Gamma_{(x,y)} = b^s (b^{un} a^{ut} z_2)$. Also, from \eqref{eq.K.1}, there exists $z\in Z_1(K)$ such that $c\Gamma_{(x,y)} = cz$. Thus $\Gamma_{(x,y)}$ fixes $Z_2(K)/Z_1(K)$ pointwise.
\end{proof}

\begin{prop}\label{autk3}
    For $n,t\in\Z$ such that $nt\equiv 1\mod 2^{2m-1}$, the assignment
    \begin{equation}\label{defrs}
        a\mapsto a^n,\; b\mapsto b^t
    \end{equation}
    extends to an automorphism $f_n$ of $K$. The corresponding map
    $f:(\Z/2^{2m-1}\Z)^\times\to\mathrm{Aut}(K)$ is a group monomorphism, whose image will be denoted by $S$.
\end{prop}

\begin{proof}
    Since $[a,b]\in Z_2(K)$, \eqref{comfor} yields
    \begin{equation}\label{rse}
        [a^n,b^t]\equiv [a,b]^{nt}\equiv [a,b]\mod Z_1(K).
    \end{equation}

    This implies that the first two defining relations of $K$ are preserved, which produces an endomorphism $f_n$ of $K$. As $f$ is clearly a homomorphism, $f_n$ has inverse $f_t$ and is then an automorphism. Evidently, $f$ is a monomorphism.
\end{proof}

Thus $|S|=${\boldmath$\varphi$}$(2^{2m-1})=2^{2m-2}=2^{m}${\boldmath$\varphi$}$(2^{m-1})$, where {\boldmath$\varphi$} is Euler's totient function.

\begin{prop}\label{autk4}
    We have $f_n\in \mathrm{Aut}_2(K)$ if and only if $n\equiv 1\mod 2^{m-1}$. Moreover,
    \[
    \mathrm{Aut}_2(K)\cap S=\langle f_{1+2^{m-1}}\rangle,\; |\mathrm{Aut}_2(K)\cap S|=2^{m},\;
    S/(\mathrm{Aut}_2(K)\cap S)\cong (\Z/2^{m-1}\Z)^\times.
    \]
\end{prop}

\begin{proof}
    Note that $f_n\in \mathrm{Aut}_2(K)$ if and only if
    $a^{n-1},b^{t-1}\in Z_2(K) = \langle a^{2^{m-1}}, b^{2^{m-1}}, c\rangle$. The normal form of the elements of $K$ makes the last condition equivalent to $n\equiv 1\mod 2^{m-1}$.

    Let $T$ be the kernel of the canonical group epimorphism $(\Z/2^{2m-1}\Z)^\times\to (\Z/2^{m-1}\Z)^\times$.
    Thus, $T$ corresponds to $\mathrm{Aut}_2(K)\cap S$ under $f$, whence
    \[
    \mathrm{Aut}_2(K) S/\mathrm{Aut}_2(K)\cong
    S/(\mathrm{Aut}_2(K)\cap S)\cong (\Z/2^{2m-1}\Z)^\times/T\cong (\Z/2^{m-1}\Z)^\times.
    \]

    Moreover, as $|T|=${\boldmath$\varphi$}$(2^{2m-1})/${\boldmath$\varphi$}$(2^{m-1})=2^{m}$, it follows that
		$|\mathrm{Aut}_2(K)\cap S|=2^{m}$.
    But $f_{1+2^{m-1}}$ is in $\mathrm{Aut}_2(K)\cap S$ and has order $2^{m}$, so it generates $\mathrm{Aut}_2(K)\cap S$.
\end{proof}


\begin{prop}\label{newauto}
    The assignment $a\mapsto ab^r,\; b\mapsto b$ extends to an automorphism, say $\Phi$, of~$K$.
\end{prop}

\begin{proof}
    By Corollary \ref{cor.bigCommutator.K}, $[ab^r,b]\equiv c\mod Z_1(K)$ and, by Corollary \ref{cor.commutator.K}, $[b^r,c] = b^u$, so
    \[
    (ab^r)^{[ab^r,b]}
    = (ab^r)^c
    = a^\alpha b^{r+u}.
    \]

    On the other hand, Note \ref{note.K.2} ensures that
    \[
    (ab^r)^\alpha
    = a^\alpha b^{\alpha r}
    = a^\alpha b^{r+u}.
    \]

    Thus $(ab^r)^{[ab^r,b]} = (ab^r)^\alpha$ and the given assignment extends to an endomorphism $\Phi$ of $K$. Since $ab^r$ and $b$ generate $K$, then $\Phi$ is surjective and hence an automorphism.
\end{proof}

In agreement with our convention on function composition,
if $V$ is a module over a commutative ring $R$ with identity and $V$ admits a finite basis $\{v_1,\dots,v_n\}$,
in order to make the correspondence between $\mathrm{Aut}(V)$ and $\GL_n(R)$ an isomorphism, we will construct
the matrix of a given automorphism of $V$ row by row instead of column by column.


\begin{theorem}\label{autk6}
    The canonical map $P:\mathrm{Aut}(K)\to \mathrm{Aut}(K/Z_2(K))$
    is a group homomorphism with kernel $\mathrm{Aut}_2(K)$ and image $(S\langle\mu,\Phi\rangle)^P$, so that
    $\mathrm{Aut}(K)=\mathrm{Aut}_2(K)S\langle \mu, \Phi \rangle$. Moreover, if $m>2$ then
    $\mathrm{Aut}(K)/\mathrm{Aut}_2(K)\cong ((\Z/2^{m-1}\Z)^\times \times (\Z/2\Z)^2)\rtimes \Z/2\Z$, with $\Z/2\Z$
    acting by inversion on $(\Z/2^{m-1}\Z)^\times$ and by switching factors on $(\Z/2\Z)^2$, and
    if $m=2$ then $\mathrm{Aut}(K)/\mathrm{Aut}_2(K)\cong\GL_2(\Z/2\Z)$. In particular,
    $\mathrm{Aut}(K)$ has order $2^{7m-1}$ if $m>2$,
    and $2^{11}\times 3$ if $m=2$, in which case $P$ is surjective. Furthermore,
    if $m>2$ (resp. $m=2$) then every element of $\mathrm{Aut}(K)$ can be written uniquely
    in the form $gf_n\Phi^i(\Phi^\mu)^j\mu^k$ (resp. $g(\mu\Phi)^i\mu^k$), where $g\in \mathrm{Aut}_2(K)$,
    $1\leq n<2^{m-1}$ is odd, and $0\leq i,j,k\leq 1$ (resp. $0\leq i\leq 2$ and $0\leq k\leq 1$).
\end{theorem}

\begin{proof}
    Note that $K/Z_2(K)$ is a free module over $\Z/2^{m-1}\Z$ of rank 2, so $P$ gives rise to a homomorphism
    $D:\mathrm{Aut}(K)\to (\Z/2^{m-1}\Z)^\times$, $\Psi\mapsto D_\Psi$, the determinant of $\Psi^P$.

    Let $\Psi\in\mathrm{Aut}(K)$. Then
    $$
    a^\Psi=a^i b^j z,\; b^\Psi=a^e b^f w\quad i,j,e,f\in\N, z,w\in Z_2(K).
    $$

    Let $d$ be any positive integer that maps into $D_\Psi$ under $\Z\to \Z/2^{m-1}\Z$. Since $c\in Z_2(K)$,
    we have
    \begin{equation}\label{det}
        c^\Psi\equiv [a^i b^j z,a^e b^f w]\equiv c^{if-je}\equiv c^{d}\mod Z_1(K).
    \end{equation}

    Taking $z = a^{sg} b^{sh} c^k$ with $g,h,k\in\Z$ and applying Corollary \ref{cor.product.K}, yields $a^i b^j z = a^{i + sg + ujg} b^{j + sh} c^k$. Thus, Note \ref{note.K.2} gives
    \[
    (a^i b^j z)^\alpha
    = a^{\alpha(i + sg + ujg)} b^{\alpha(j + sh)} c^{\alpha k}
    = a^{i + sg + ujg + 2s\ell i} b^{j + sh + 2s\ell j} c^k.
    \]
    Since $\alpha^d\equiv 1 + 2s\ell d\mod 2u$, $(1 - 2s\ell)^d\equiv 1 - 2s\ell d\mod 2u$, and $b^c = b^{1 - 2s\ell}$, we have
    \[
    (a^i b^j z)^{c^d}
    = a^{\alpha^d(i + sg + ujg)} b^{(1 - 2s\ell)^d(j + sh)} c^k
    = a^{i + sg + ujg + 2s\ell id} b^{j + sh - 2s\ell jd} c^k.
    \]

    As $\Psi$ preserves the relations $a^c=a^\alpha$ and $b^c=b^\alpha$ it follows from \eqref{det} that $(a^i b^j z)^{c^d}=(a^i b^j z)^\alpha$ and $(a^e b^f w)^{c^d}=(a^e b^f w)^\alpha$. Therefore, from above, $a^{2s\ell i(d-1)} b^{-2s\ell j(d+1)} = 1$, and similarly, $a^{2s\ell e(d-1)} b^{-2s\ell f(d+1)} = 1$. From this we infer 
    \begin{equation}\label{coli14}
        i(d-1)\equiv 0,\, j(d+1)\equiv 0,\, e(d-1)\equiv 0,\, f(d+1)\equiv 0\quad \mod 2^{m-1}.
    \end{equation}

    Since $d$ is invertible modulo $2^{m-1}$, at least one of $i,j$ must be invertible modulo $2^{m-1}$ by \eqref{det}.

    Assume until further notice that $m>2$. 
		If $2\nmid i$ then \eqref{coli14} yields $d\equiv 1\mod 2^{m-1}$, whence $2j\equiv 0\mod 2^{m-1}$, that
    is, $j\equiv 0\mod 2^{m-2}$. Since $m>2$, we infer $2\nmid f$. From $d\equiv 1\mod 2^{m-1}$ we also obtain
    $2e\equiv 0\mod 2^{m-1}$, that is, $e\equiv 0\mod 2^{m-2}$. If
    $2\nmid j$ then replacing $\Psi$ by $\mu\Psi$ the previous case yields $D_\Psi=-1$, $2\nmid e$,
    and $i,f\equiv 0\mod 2^{m-2}$.

    We have shown that $D_\Psi=\pm 1$ as well as the following: if $D_\Psi=1$ then
    $if\equiv 1\mod 2^{m-1}$ and $j,e\equiv 0\mod 2^{m-2}$, while
    if $D_\Psi=-1$ then $je\equiv -1\mod 2^{m-1}$ and $i,f\equiv 0\mod 2^{m-2}$.

    Going back to the general case $m>1$, we fix the $\Z/2^{m-1}\Z$-basis $\{a Z_2(K),bZ_2(K)\}$ of $K/Z_2(K)$, and 
		identify $\mathrm{Aut}(K/Z_2(K))$
    with $\GL_2(\Z/2^{m-1}\Z)$. We then have the following matrices:
    $$
    M_n=f_n^P=
    \left(\begin{array}{cc}
    [n] & 0\\ 0 & [t]
    \end{array}\right),\,
    Q=\mu^P=
    \left(\begin{array}{cc}
    0 & 1\\ 1 & 0
    \end{array}\right),\,
    R=\Phi^P=
    \left(\begin{array}{cc}
    1 & w\\ 0 & 1
    \end{array}\right),
    $$
    where $z\mapsto [z]$ is the canonical projection $\Z\to\Z/2^{m-1}\Z$,
    $n\in\Z$ is odd with inverse $t$ modulo $2^{m-1}$, and $w=[2^{m-2}]$.
    Set $U=\{M_n\,|\, n\in\Z, n\text{ odd}\}$, a group of order {\boldmath$\varphi$}$(2^{m-1})=2^{m-2}$,
    and $T=R^Q$.

    Suppose first $m>2$. Then $V=U\langle R\rangle\langle T\rangle$ is an abelian
    group of order $2^m$ isomorphic to $(\Z/2^{m-1}\Z)^\times \times (\Z/2\Z)^2$. Moreover, every element of $V$ has
    determinant 1, so $Q\notin V$, but $Q$ normalizes $V$,
    conjugating every $M_r$ into its inverse, and $R$ and $T$ into each other. Therefore $W=V\rtimes\langle Q\rangle$
    is a group of order $2^{m+1}$. If $D_\Psi=1$ then $if\equiv 1\mod 2^{m-1}$ and $j,e\equiv 0\mod 2^{m-2}$, which makes it obvious that
    $\Psi^P\in V$. Thus, if $D_\Psi=-1$ then $\Psi^P\in W$. This proves
    $$\mathrm{Aut}(K)^P=W=U\langle R\rangle\langle T\rangle\langle Q\rangle=(S \langle \Phi\rangle
    \langle \Phi^\mu\rangle\langle \mu\rangle)^P.
    $$
    As $\ker(P)=\mathrm{Aut}_2(K)$, we see that
    $\mathrm{Aut}(K)=\mathrm{Aut}_2(K)S\langle \Phi \rangle\langle \Phi^\mu \rangle\langle \mu \rangle$ and
    $\mathrm{Aut}(K)/\mathrm{Aut}_2(K)\cong W$.
    Since $|\mathrm{Aut}_2(K)|=2^{6m-2}$
    by Proposition \ref{autexcel}, and $|W|=2^{m+1}$ by above, we deduce $|\mathrm{Aut}(K)|=2^{7m-1}$ with uniqueness
    of expression.

    Suppose next $m=2$. Then $T,Q$ generate $\GL_2(\Z/2\Z)$ and $P$ is surjective,
    so by Proposition \ref{autexcel}
    $$
    |\mathrm{Aut}(K)|=|\mathrm{Aut}_2(K)|\times|\GL_2(\Z/2\Z)|=2^{10}\times 6=2^{11}\times 3,
    $$
    with uniqueness of expression and $\mathrm{Aut}(K)/\mathrm{Aut}_2(K)\cong \GL_2(\Z/2\Z)$.
\end{proof}

\begin{cor}
    Every automorphism of $K$ has determinant $\pm 1$ when acting on $K/Z_2(K)$. Moreover, if $m>2$ then
    $\mathrm{Aut}_2(K)S\langle \Phi,\Phi^\mu \rangle$ is the kernel of the corresponding determinant map, as well as the
    pointwise stabilizer of $c Z_1(K)$ when $\mathrm{Aut}(K)$ acts on $K/Z_1(K)$.
\end{cor}

We next obtain a description of $\mathrm{Aut}(K)$ in terms of group extensions.

\begin{theorem}\label{gpextk}
    We have the following series of normal subgroups of $\mathrm{Aut}(K)$:
    \begin{equation}\label{norser}
        1\subset\mathrm{Aut}_1(K)\subset\mathrm{Aut}_2(K)\subset\mathrm{Aut}(K),
    \end{equation}
    where $\mathrm{Aut}_1(K)\cong (\Z/2^{m-1}\Z)^4$, $\mathrm{Aut}_2(K)/\mathrm{Aut}_1(K)\cong (\Z/2^{m-1}\Z)^2\times (\Z/2\Z)^4$,
    and the factor $\mathrm{Aut}(K)/\mathrm{Aut}_2(K)\cong ((\Z/2^{m-1}\Z)^\times \times (\Z/2\Z)^2)\rtimes \Z/2\Z$, with $\Z/2\Z$
    acting by inversion on $(\Z/2^{m-1}\Z)^\times$ and by switching factors on $(\Z/2\Z)^2$ if $m>2$, while
    $\mathrm{Aut}(K)/\mathrm{Aut}_2(K)\cong\GL_2(\Z/2\Z)$ if $m=2$.
\end{theorem}

\begin{proof}
    Immediate consequence of Propositions \ref{autk} and \ref{autexcel}, and Theorem \ref{autk6}.
\end{proof}

Our next result is designed for later use, but it can also be used to insert an additional normal
subgroup of $\mathrm{Aut}(K)$ to the series \eqref{norser}, namely $\mathrm{Inn}(K)\mathrm{Aut}_1(K)$.

\begin{prop}\label{outow}
    We have
    \begin{equation}\label{iso3}
        \mathrm{Aut}_2(K)/\mathrm{Inn}(K)\mathrm{Aut}_1(K)\cong Z_2(K)^2/(Z_1(K)\times \langle c\rangle)^2\cong (\Z/2\Z)^4.
    \end{equation}
    Moreover, $\mathrm{Aut}_2(K)/\mathrm{Inn}(K)\mathrm{Aut}_1(K)$ is generated by the cosets of
    $\psi_1,\psi_2,\psi_3,\psi_4\in \mathrm{Aut}_2(K)$, where these automorphisms are respectively given by
	$$
    a\mapsto a^{1+s},b\mapsto b;\;
    a\mapsto a, b\mapsto b^{1+s};\;
    a\mapsto ab^{s},b\mapsto b;\;
    a\mapsto a,b\mapsto ba^{s},
	$$
	and whose existence is ensured by Proposition \ref{autexcel}.
\end{prop}

\begin{proof}
    Let $g:Z_2(K)^2\to \mathrm{Aut}_2(K)/\mathrm{Aut}_1(K)$ be the epimorphism defined in Proposition \ref{autexcel},
    and let $\pi:\mathrm{Aut}_2(K)/\mathrm{Aut}_1(K)\to \mathrm{Aut}_2(K)/\mathrm{Inn}(K)\mathrm{Aut}_1(K)$ be the canonical
    projection.

    From $\Gamma_{(1,c^{-1})}=a\delta$ and $\Gamma_{(1,c)}=b\delta$ we deduce $\langle c\rangle^2 g=\mathrm{Inn}(K)\mathrm{Aut}_1(K)/\mathrm{Aut}_1(K)=\ker\pi$. Since $\ker g=Z_1(K)^2$, we deduce $\ker g\pi=
    (Z_1(K)\times \langle c\rangle)^2$. The first isomorphism theorem now yields~(\ref{iso3}).
    Moreover, $\mathrm{Aut}_2(K)/\mathrm{Inn}(K)\mathrm{Aut}_1(K)$ is generated by the images of
    $(a^s,1), (b^s,1), (1,a^s), (1,b^s)$
    under the epimorphism $Z_2(K)^2\to \mathrm{Aut}_2(K)/\mathrm{Inn}(K)\mathrm{Aut}_1(K)$, namely the cosets of $\psi_1,\psi_2,\psi_3,\psi_4$.
\end{proof}

All factors of the new series
$1\subset\mathrm{Aut}_1(K)\subset\mathrm{Inn}(K)\mathrm{Aut}_1(K)\subset\mathrm{Aut}_2(K)\subset\mathrm{Aut}(K)$
have already been computed, except for the second, which can be determined as follows:
$$
\mathrm{Inn}(K)\mathrm{Aut}_1(K)/\mathrm{Aut}_1(K)\cong
\mathrm{Inn}(K)/\mathrm{Aut}_1(K)\cap \mathrm{Inn}(K)\cong K\delta/Z_2(K)\delta\cong K/Z_2(K)\cong (\Z/2^{m-1}\Z)^2.
$$
This is compatible with Proposition \ref{outow} and the structure of $\mathrm{Aut}_2(K)/\mathrm{Aut}_1(K)$.

\section{The automorphism group of $J/Z_1(J)$}\label{autgmodh}

Set $H = J/Z_1(J)$. It follows from \eqref{presn}, \eqref{put3}, and \cite[Proposition 5.1]{MS} that $H$ has presentation
$$
H = \groupPresentation{A,B}{A^{[A,B]}=A^\alpha,\; B^{[B,A]}=B^\alpha,\; A^{2^{2m-1}}=1,\; B^{2^{2m-1}}=1},
$$
and we set $C=[A,B]$, noting that $C^{2^{2m-1}}=1$ by \eqref{put}. Observe the automorphism
$A\leftrightarrow B$, $C\leftrightarrow C^{-1}$, say $\nu$, of $H$. It follows
from \eqref{put3} and \eqref{put4} that
$$
Z_1(H)=\langle C^{2^{m-1}}\rangle,\;
Z_2(H)=\langle A^{2^{m}}, B^{2^{m}}, C^{2^{m-1}}\rangle,\;
Z_3(H)=\langle A^{2^{m-1}}, B^{2^{m-1}}, C\rangle.
$$

As $Z_3(J)$ is abelian, so is $Z_2(H)\cong Z_3(J)/Z_1(J)$, and \eqref{put5} yields
$$
|Z_1(H)|=2^{m},\;
|Z_2(H)|=2^{3m-2},\;
|Z_3(H)|=2^{4m-1},\;
|H|=2^{6m-3},\;
|\mathrm{Inn}(H)|=2^{5m-3}.
$$

We may now deduce from \eqref{gol} that every element of $H$ can be written uniquely in the form
$A^i B^j C^k$, where $0\leq i,j,k<2^{2m-1}$, and that $o(A)=o(B)=o(C)=2^{2m-1}$.

Recalling the notation of Section \ref{section.formulas.J}, we have $A^{2u} = B^{2u}=C^{2u} = 1$, $Z_1(H) = \gen{C^s}$, $Z_2(H) = \gen{A^{2s},B^{2s},C^{s}}$, and $Z_3(H) = \gen{A^s,B^s,C}$.

The following four corollaries follow from the main results in Section \ref{section.formulas.J}.

\begin{cor}\label{cor.commutator.H}
    For all $n,t\in\Z$ the following identities hold in $H$:
    \begin{align*}
        &[C^n,A^t]
        = A^{-2s\ell nt},\\
        &[C^n,B^t]
        = B^{2s\ell nt},\\
        &[A^n,B^t]
        = A^{-2s\ell\phi(n)t} B^{2s\ell n\phi(t)} C^{nt - 2s\ell\phi(n)\phi(t)}.
    \end{align*}
\end{cor}

\begin{cor}\label{cor.product.H}
    Let $i,j,k,a,b,c\in\Z$. Then $(A^iB^jC^k)(A^aB^bC^c) = A^{\textup{exp} A} B^{\textup{exp} B} C^{\textup{exp} C}$, where
    \begin{align*}
        \textup{exp} A
        &= i + a + 2s\ell\{j\phi(a) - ka\},\\
        \textup{exp} B
        &= j + b + 2s\ell\{kb - jab - \phi(j)a\},\\
        \textup{exp}\, C
        &= k + c - ja + 2s\ell\{jka - \phi(j+1)\phi(a)\}.
    \end{align*}
\end{cor}

\begin{cor}\label{cor.bigCommutator.H}
    Let $i,j,k,a,b,c\in\Z$. Then
    $[A^iB^jC^k,A^aB^bC^c] = A^{\textup{exp} A} B^{\textup{exp} B}C^{\textup{exp} C}$, where
    \begin{align*}
        \textup{exp} A
        &= 2s\ell\{j\phi(a) - \phi(i)b + ic - ka\},\\
        \textup{exp} B
        &= 2s\ell\{i\phi(b) - \phi(j)a + kb + j(ib - ab - c)\},\\
        \textup{exp}\, C
        &= ib - ja + 2s\ell\{\phi(a)(\phi(j) + jb) - \phi(i)(\phi(b) + jb) + ijc - kab\}.
    \end{align*}
\end{cor}

\begin{cor}\label{cor.power.H}
    Let $a,b,c,n,k,t\in\Z$. Then
    $(A^aB^bC^c)^n = A^{\textup{exp} A} B^{\textup{exp} B} C^{\textup{exp} C}$, where
    \begin{align*}
        \textup{exp} A
        &= na + 2s\ell\{a^2b\varphi(n) + (\phi(a)b - ac)\phi(n)\},\\
        \textup{exp} B
        &= nb + 2s\ell\{(b(c - ab) - a\phi(b))\phi(n) - 2ab^2\varphi(n)\},\\
        \textup{exp}\, C
        &= nc - ab\phi(n) + 2s\ell\{a^2\phi(b)\varphi(n) + \phi(a)\phi(b)\phi(n) - a^2b^2\sigma_2(1,n) - (\phi(a)b-ac)b\varphi(n+1)\}.
    \end{align*}
\end{cor}

\begin{note}\label{note.H.1}
    If $s\mid\phi(n)$, then $s\mid\varphi(n)$ and $s\mid\varphi(n+1)$, so
	$$(A^aB^bC^c)^n = A^{na} B^{nb} C^{nc - ab\phi(n) - 2s\ell a^2b^2\sigma_2(1,n)}.$$ In particular, $(A^aB^bC^c)^{2u} = C^{uab}$
    and the exponent of $H$ is $2^{2m}$.
\end{note}

\begin{prop}\label{auth}
    For every $x,y\in Z_2(H)$ the assignment $A\mapsto Ax,\; B\mapsto By$ extends to a 2-central automorphism $\Pi_{(x,y)}$ of $H$ that fixes $Z_2(H)$ pointwise. Moreover, the corresponding map
    $\Pi:Z_2(H)\times Z_2(H)\to \mathrm{Aut}(H)$ is a group monomorphism whose image is $\mathrm{Aut}_2(H)$. In particular,
    $|\mathrm{Aut}_2(H)|=2^{6m-4}$.
\end{prop}

\begin{proof}
    By Corollary \ref{cor.commutator.H}, $[B,A^{2sa}] = C^{-2sa}$ and $[C^s,B^{2s}] = 1$.
	Let $x=A^{2si}B^{2sj}C^{sk}$ as well as $y=A^{2sa} B^{2sb} C^{sc}$. Then
    \[
    By
    = B A^{2sa} B^{2sb} C^{sc}
    = A^{2sa} B C^{-2sa} B^{2sb} C^{sc}
    = A^{2sa} B^{1+2sb} C^{s(c-2a)}.
    \]

    Applying Corollary \ref{cor.bigCommutator.H} to $[Ax,By] = [A^{1+2si} B^{2sj} C^{sk},A^{2sa} B^{1+2sb} C^{s(c-2a)}] = A^{\textup{exp} A} B^{\textup{exp} B} C^{\textup{exp} C}$ gives $\textup{exp} A\equiv 0\mod 2u$, $\textup{exp} B\equiv 0\mod 2u$, $\textup{exp}\, C\equiv 1 + 2s(i+b)\mod 2u$, so
    \begin{equation}\label{eq.H.1}
        [Ax,By]
        = C^{1 + 2s(i+b)}.
    \end{equation}

    As $C^{2s(i+b)}\in Z_1(H)$ and, by Corollary \ref{cor.commutator.H}, $[A^{1+2si},C] = A^{2s\ell}$ and $[B^{2sj},C] = 1$, then
    \[
    (Ax)^{[Ax,By]} = (A^{1+2si} B^{2sj} C^{sk})^C = A^{\alpha+2si} B^{2sj} C^{sk}.
    \]

    On the other hand, Note \ref{note.H.1} yields
    \[
    (Ax)^\alpha
    = (A^{1+2si} B^{2sj} C^{sk})^\alpha
    = A^{\alpha(1+2si)} B^{\alpha 2sj} C^{\alpha sk}
    = A^{\alpha+2si} B^{2sj} C^{sk}.
    \]

    Thus $(Ax)^{[Ax,By]} = (Ax)^\alpha$. By the automorphism $A\leftrightarrow B$, $(By)^{[By,Ax]} = (By)^\alpha$.

    By Note \ref{note.H.1}, $(Ax)^{2u} = (A^{1+2si} B^{2sj} C^{sk})^{2u} = C^{u(1+2si)(2sj)} = 1$. By the automorphism $A\leftrightarrow B$, $(By)^{2u} = 1$. Thus the given assignment extends to an endomorphism $\Pi$ of $H$.

    Now, by Note \ref{note.H.1}, $(A^{2s})\Pi_{(x,y)} = (Ax)^{2s} = (A^{1+2si} B^{2sj} C^{sk})^{2s} = A^{2s}$ and, by the automorphism $A\leftrightarrow B$, $(By)^{2s} = B^{2s}$, so $(B^{2s})\Pi_{(x,y)} = B^{2s}$. Also, from \eqref{eq.H.1}, $(C^s)\Pi_{(x,y)} = [Ax,By]^s = C^s$. Thus $\Pi_{(x,y)}$ fixes $Z_2(H)$ pointwise and the nilpotency of $H$ ensures $\Pi_{(x,y)}$ is an automorphism. We may now continue as in the proof of Proposition \ref{autk}.
\end{proof}

\begin{cor}\label{corauth}
    We have
    $\mathrm{Inn}(H)\cap\mathrm{Aut}_2(H) = \langle C\delta,A^{2^{m-1}}\delta, B^{2^{m-1}}\delta\rangle$, $|\mathrm{Inn}(H)\cap\mathrm{Aut}_2(H)|=2^{3m-1}$,
    $$
    |\mathrm{Inn}(H)\mathrm{Aut}_2(H)|=2^{8m-6},\; \mathrm{Inn}(H)\mathrm{Aut}_2(H)\subset \mathrm{Aut}_3(H).
    $$
\end{cor}

\begin{proof}
    As $Z_3(H)/Z_1(H)$ has order $2^{3m-1}$ and is generated by the cosets of $C,A^{2^{m-1}}, B^{2^{m-1}}$,
    it follows that
    $\mathrm{Inn}(H)\cap\mathrm{Aut}_2(H)=\langle C\delta,A^{2^{m-1}}\delta, B^{2^{m-1}}\delta\rangle$ has order $2^{3m-1}$.

    We deduce from Proposition \ref{auth} that $|\mathrm{Inn}(H)\mathrm{Aut}_2(H)|=2^{8m-6}$.
    On the other hand, it is obvious  that $\mathrm{Aut}_2(K)$ is included in $\mathrm{Aut}_3(K)$.
    As $B^A=BC^{-1}$, $A^B=AC$, with $C\in Z_3(H)$, it follows that $\mathrm{Inn}(K)$ is also included in $\mathrm{Aut}_3(H)$.
\end{proof}



\begin{prop}\label{casomdos2}
    The assignment $A\mapsto A^{1+s},\; B\mapsto BA^s$ extends to an automorphism, say $\Gamma$, of $H$ that belongs to $\mathrm{Aut}_3(H)$.
\end{prop}

\begin{proof}
    By Corollary \ref{cor.commutator.H}, $[B,A^s] = (A^u B^0 C^s)^{-1} = A^u C^{-s}$ and $[B,A^u] = (A^0 B^0 C^u)^{-1} = C^u$, so
    \[
    BA^s
    = A^sB^{A^s}
    = A^sBA^uC^{-s}
    = A^{s+u}BC^{-s+u}.
    \]

    By Corollary \ref{cor.bigCommutator.H}, $[A^{1+s},BA^s] = [A^{1+s},A^{s+u}BC^{-s+u}] = A^uB^0C^{1+s} = A^uC^{1+s}$, and by Corollary \ref{cor.commutator.H}, $[A^{1+s},C^{1+s}] = A^{2s\ell}$. Then
    \[
    (A^{1+s})^{[A^{1+s},BA^s]}
    = (A^{1+s})^{A^uC^{1+s}}
    = A^{\alpha+s}.
    \]

    On the other hand, $(A^{1+s})^\alpha = A^{\alpha+s}$. Thus $(A^{1+s})^{[A^{1+s},BA^s]} = (A^{1+s})^\alpha$.

    As for the second relation, Corollary \ref{cor.commutator.H} and \eqref{comfor} give, $[BA^s,A^{1+s}] = A^uC^{-1-s}$, $[B,A^u] = C^u$, and
	$[A^{u+s},C^{-1}] = 1$, so
    \[
    (BA^s)^{[BA^s,A^{1+s}]}
    = (A^{s+u}BC^{-s+u})^{A^uC^{-1}}
    = (A^{s+u}B)^{C^{-1}}C^{-s}
    = A^{s+u}B^\alpha C^{-s}.
    \]

    On the other hand, Note \ref{note.H.1} gives
    \[
    (BA^s)^\alpha
    = (A^{s+u}BC^{-s+u})^\alpha
    = (A^{s+u}B)^\alpha C^{-s+u}
    = A^{\alpha(s+u)}B^\alpha C^u C^{-s+u}
    = A^{s+u}B^\alpha C^{-s}.
    \]
    Also, by Note \ref{note.H.1}, $(A^{1+s})^{2u} = 1$ and $(BA^s)^{2u} = (A^{s+u}BC^{-s+u})^{2u} = C^{u(s+u)} = 1$. Then the given assignment extends to an endomorphism $\Gamma$ of $H$.

    Since $[C,A^u] = 1$, then $(C^s)\Gamma = [A^{1+s},BA^s]^s = (A^uC^{1+s})^s = A^{us}C^{(1+s)s} = C^{(1+s)s}$, so the restriction of $\Gamma$ to $Z_1(H)$ is an automorphism, whence
    $\Gamma$ is an automorphism, clearly in $\mathrm{Aut}_3(H)$.
\end{proof}

\begin{prop}\label{zi2}
    Let $T$ be a group, and set $Y=T/Z_1(T)$. Let $\lambda:T\to Y$ be the canonical projection,
    and consider the associated map $\Lambda:\mathrm{Aut}(T)\to \mathrm{Aut}(Y)$.
    Then for any $i\geq 0$, $\Lambda$ maps $\mathrm{Aut}_{i+2}(T)$ into $\mathrm{Aut}_{i+1}(Y)$, and the kernels of the
    induced maps $\mathrm{Aut}_{i+2}(T)\to\mathrm{Aut}_{i+1}(Y)/\mathrm{Aut}_{i}(Y)$ and
    $\mathrm{Inn}(T)\mathrm{Aut}_{i+2}(T)\to\mathrm{Inn}(Y)\mathrm{Aut}_{i+1}(Y)/\mathrm{Inn}(Y)\mathrm{Aut}_{i}(Y)$ are
    $\mathrm{Aut}_{i+1}(T)$ and $\mathrm{Inn}(T)\mathrm{Aut}_{i+1}(T)$, respectively.
    Thus, $\mathrm{Aut}_{i+2}(T)/\mathrm{Aut}_{i+1}(T)$ is isomorphic to a subgroup of
    $\mathrm{Aut}_{i+1}(Y)/\mathrm{Aut}_{i}(Y)$,
    and $\mathrm{Inn}(T)\mathrm{Aut}_{i+2}(T)/\mathrm{Inn}(T)\mathrm{Aut}_{i+1}(T)$ imbeds into
    $\mathrm{Inn}(Y)\mathrm{Aut}_{i+1}(Y)/\mathrm{Inn}(Y)\mathrm{Aut}_{i}(Y)$.
\end{prop}

\begin{proof}
    We first show by induction that $\lambda$ sends $Z_{i+1}(T)$ onto $Z_{i}(Y)$ for any $i\geq 0$.
    The base case $i=0$ holds by the very definition of $\lambda$. Suppose that $\lambda$ maps $Z_{i+1}(T)$ onto $Z_i(Y)$
    for some $i\geq 0$. Let $t\in Z_{i+2}(T)$. Then $[t,\sigma]\in Z_{i+1}(T)$, so
    $[t^\lambda,\sigma^\lambda]=[t,\sigma]^\lambda\in Z_i(Y)$ for every $\sigma\in T$,
    whence $t^\lambda\in Z_{i+1}(Y)$. Conversely, if $y\in Z_{i+1}(Y)$, then $y=t^\lambda$ for some $t\in T$. Let $\sigma\in T$.
    Then $[t,\sigma]^\lambda=[t^\lambda,\sigma^\lambda]\in Z_i(Y)$, so $[t,\sigma]^\lambda=w^\lambda$ for some $w\in Z_{i+1}(T)$,
    hence $[t,\sigma]w^{-1}\in\ker(\lambda)=Z_1(T)$, and therefore $[t,\sigma]\in Z_1(T)Z_{i+1}(T)=Z_{i+1}(T)$, which implies $t\in Z_{i+2}(T)$.
    This shows that $\lambda$ maps $Z_{i+2}(T)$ onto $Z_i(Y)$.

    For $t\in T$, we set $\overline{t}=t^\lambda$, and if $g\in \mathrm{Aut}(T)$, then $\overline{g}=g^\Lambda$ is defined
    by $\overline{t}^{\overline{g}}=\overline{t^g}$. This is well defined, as $Z_1(T)$ is a characteristic subgroup of $T$.

    Let $i\geq 0$. We claim that $\Lambda$ sends $\mathrm{Aut}_{i+2}(T)$ into $\mathrm{Aut}_{i+1}(Y)$.
    Indeed, let $g\in \mathrm{Aut}_{i+2}(T)$ and
    $t\in T$. Then $t^g t^{-1}\in Z_{i+2}(T)$, so $\overline{t^g t^{-1}}\in Z_{i+1}(Y)$ by the above, which means
    $\overline{t}^{\overline{g}} \overline{t}^{-1} Z_{i+1}(Y)$, so $g^\Lambda\in \mathrm{Aut}_{i+1}(Y)$,  as claimed.

    By the above, we have a group homomorphism $\eta:\mathrm{Aut}_{i+2}(T)\to \mathrm{Aut}_{i+1}(Y)/\mathrm{Aut}_{i}(Y)$ with
    $\mathrm{Aut}_{i+1}(T)$ contained in $\ker(\eta)$, and we claim equality holds.
    Let $g\in\ker(\eta)$. Then $g^\Lambda\in \mathrm{Aut}_{i}(Y)$. Let $t\in T$.
    Then $\overline{t}^{\overline{g}} \overline{t}^{-1}\in Z_i(Y)$, so $\overline{t^g t^{-1}}\in Z_{i}(Y)$. As $\lambda$
    maps $Z_{i+1}(T)$ onto $Z_i(Y)$, there is some $\sigma\in Z_{i+1}(T)$ such that $\overline{t^g t^{-1}}=\overline{\sigma}$,
    whence $t^g t^{-1}\sigma^{-1}\in Z_1(T)$, so $t^g t^{-1}\in Z_1(T)Z_{i+1}(T)=Z_{i+1}(T)$, and therefore $g\in \mathrm{Aut}_{i+1}(T)$.
    Thus $\ker(\eta)=\mathrm{Aut}_{i+1}(T)$, as claimed.

    Now $\Lambda$ maps $\mathrm{Inn}(T)\mathrm{Aut}_{i+2}(T)$ into
    $\mathrm{Inn}(Y)\mathrm{Aut}_{i+1}(Y)$, and $\mathrm{Inn}(T)\mathrm{Aut}_{i+1}(T)$ into
    $\mathrm{Inn}(Y)\mathrm{Aut}_{i}(Y)$, thus producing a map $\phi:\mathrm{Inn}(T)\mathrm{Aut}_{i+2}(T)\to
    \mathrm{Inn}(Y)\mathrm{Aut}_{i+1}(Y)/\mathrm{Inn}(Y)\mathrm{Aut}_{i}(Y)$ whose kernel contains $\mathrm{Inn}(T)\mathrm{Aut}_{i+1}(T)$,
    and we claim that $\ker(\phi)=\mathrm{Inn}(T)\mathrm{Aut}_{i+1}(T)$. Indeed, if
    $g\in \ker(\phi)$, then $g^\Lambda\in \mathrm{Inn}(Y)\mathrm{Aut}_{i}(Y)$, so that $g^\Lambda=hk$,
    with $h\in \mathrm{Inn}(Y)$ and $k\in\mathrm{Aut}_{i}(Y)$. But $\Lambda$ maps $\mathrm{Inn}(T)$ onto $\mathrm{Inn}(Y)$,
    so $h=w^\Lambda$ for some $w\in \mathrm{Inn}(T)$, whence $k=(w^{-1}g)^\Lambda$, so by the above $w^{-1}g\in \mathrm{Aut}_{i+1}(T)$
    and therefore $g\in \mathrm{Inn}(T)\mathrm{Aut}_{i+1}(T)$, as claimed.
\end{proof}

We apply Proposition \ref{zi2} to the case $T=H$ and $Y=K=H/Z_1(H)$, which has presentation
$$
K = \groupPresentation{\rho,\eta}{\rho^{[\rho,\eta]}=\rho^\alpha,\; \eta^{[\eta,\rho]}=\eta^\alpha,\; \rho^{2^{2m-1}}=1,\; \eta^{2^{2m-1}}=1,\; [\rho,\eta]^{2^{m-1}}=1}.
$$

\begin{theorem}\label{tamanio}
    Let $\lambda:H\to K$ be the projection $A\mapsto \rho$, $B\mapsto\eta$,
    $\Lambda:\mathrm{Aut}(H)\to\mathrm{Aut}(K)$ the
    corresponding homomorphism, and
    $$
    \widehat{\Lambda}:\mathrm{Aut}_3(H)/(\mathrm{Inn}(H)\mathrm{Aut}_2(H))\hookrightarrow\mathrm{Aut}_2(K)/(\mathrm{Inn}(K)\mathrm{Aut}_1(K))
    $$
    the imbedding associated to $\Lambda$, as indicated in Proposition \ref{zi2}. Then
    $\mathrm{Im}(\widehat{\Lambda})=\langle \overline{\Gamma},\overline{\Gamma^\nu}\rangle^{\widehat{\Lambda}}$ is isomorphic to the Klein 4-group.
    Therefore, $|\mathrm{Aut}_3(H)|=2^{8m-4}$ and
    $\mathrm{Aut}_3(H)=\mathrm{Inn}(H)\mathrm{Aut}_2(H)\langle \Gamma,\Gamma^\nu\rangle$,
    where $\Gamma$ is as defined in Proposition \ref{casomdos2}, and $\nu$ is the automorphism $A\leftrightarrow B$ of $H$.
\end{theorem}

\begin{proof}
    As $\mathrm{Inn}(H)\subset \mathrm{Aut}_3(H)$ and $\mathrm{Inn}(K)\subset \mathrm{Aut}_2(K)$,
    Proposition \ref{zi2} yields the imbedding $\widehat{\Lambda}$.

    We claim that
    $\mathrm{Im}(\widehat{\Lambda})=\langle \overline{\Gamma},\overline{\Gamma^\nu}\rangle^{\widehat{\Lambda}}$ is isomorphic to
    the Klein 4-group. It will then follow that $\mathrm{Aut}_3(H)=\mathrm{Inn}(H)\mathrm{Aut}_2(H)\langle \Gamma,\Gamma^\nu\rangle$,
    with $|\mathrm{Aut}_3(H)|=2^{8m-4}$, as $|\mathrm{Inn}(H)\mathrm{Aut}_2(H)|=2^{8m-6}$ by
    Corollary \ref{corauth}.


	We know from Proposition \ref{outow} that $\mathrm{Aut}_2(K)/\mathrm{Inn}(K)\mathrm{Aut}_1(K)$ is generated
	by $\overline{\psi_1},\overline{\psi_2},\overline{\psi_3},\overline{\psi_4}$. Setting
	$T=\langle \overline{\psi_1}\overline{\psi_4}, \overline{\psi_2}\overline{\psi_3}\rangle$, we proceed
    to show that $\mathrm{Im}(\widehat{\Lambda})=T$. By Proposition \ref{casomdos2},
    we have $\Gamma^\Lambda=\psi_1\psi_4$, and therefore $(\nu^{-1}\Gamma\nu)^\Lambda=\psi_2\psi_3$. Thus
    $T\subseteq \mathrm{Im}(\widehat{\Lambda})$. We readily see that
	$1,\overline{\psi_1},\overline{\psi_2},\overline{\psi_1}\overline{\psi_2}$ is a system
    of representatives for the cosets of $T$ in $\overline{\mathrm{Aut}_2(K)}$. Thus, it suffices to show that none of
	$\overline{\psi_1},\overline{\psi_2},\overline{\psi_1}\overline{\psi_2}$ are in
    $\mathrm{Im}(\widehat{\Lambda})$. As $\mathrm{Inn}(H)\mathrm{Aut}_2(H)$ maps onto $\mathrm{Inn}(K)\mathrm{Aut}_1(K)$
	by Proposition~\ref{auth}, this translates as follows: none of $\psi_1,\psi_2,\psi_1\psi_2$
	lift to an automorphism of $H$. Note that we may replace
    any of $\psi_1,\psi_2,\psi_1\psi_2$ by itself multiplied by any element of $\mathrm{Inn}(K)\mathrm{Aut}_1(K)$
    in this statement.

    Suppose, if possible, that $\psi_1$ lifts to an automorphism, say $\tau$, of $H$. Then there are $x,y\in Z_1(H)$ such that
    $$
    A^\tau=A^{1+s}x,\;
    B^\tau=By.
    $$

    Now by Corollary \ref{cor.commutator.H} and \eqref{comfor}, we have $[A^{1+s}x,By] = [A^{1+s},B] = A^u C^{1+s}$,
    and $[A^u,B] = C^u$, so
    $$
    (By)^{[By,A^{1+s}x]}
    = B^{C^{-s}C^{-1}A^u} y
    = (B^\alpha)^{A^u}y
    = (BC^u)^\alpha y
    = B^\alpha C^u y.
    $$

    Since $(By)^\alpha=B^\alpha y$, the second defining relation of $H$ is not preserved by $\tau$. This proves
    that $\psi_1$ does not lift to an automorphism of $H$. Since $\psi_2^\mu=\psi_1$, it follows that
    $\psi_2$ does not lift to an automorphism of $H$ either.

    Let $\psi_5\in\mathrm{Aut}_2(K)$ be given by $a\mapsto a^{1+s}b^s$, $b\mapsto b$. Then	$\overline{\psi_1\psi_2}\equiv 	
		\overline{\psi_1\psi_3}\mod T$ and $\psi_1\psi_3\equiv\psi_5\mod\mathrm{Inn}(K)\mathrm{Aut}_1(K)$. Suppose, if possible, that
    $\psi_5$ lifts to an automorphism, say $\sigma$, of $H$.
    Then there are $w,z\in Z_1(H)$ such that
    $$
    A^\sigma=A^{1+s}B^s w,\;
    B^\sigma=Bz.
    $$
    Since $[A^{1+s}B^s w,Bz] = [A^{1+s},B]^{B^s} = A^u C^{1+s}$,
    we see, as in the previous case, that the second defining relation of $H$ is not preserved by $\sigma$.
\end{proof}

We set $\overline{r} = r/2$ whenever $m>2$ for the remainder of this section.

\begin{prop}\label{muchos}
    Suppose that $m>2$. Let $x = rd$, $y = re$, with $d,e\in\{0,1\}$, and further let $a,b\in\Z$, with $ab\equiv 1\mod s$, as well as $z,z^\prime\in Z_1(H)$. Then the assignment
    \begin{equation}\label{eq.15}
        A\mapsto A^aB^xz,\quad
        B\mapsto B^bA^yz^\prime,
    \end{equation}
    extends to an automorphism of $H$ if and only if $x=y=0$ and $a\equiv b\equiv 1\mod 2s$, or $x=y=r$ and $a\equiv 1+r\equiv b\mod 2s$.
\end{prop}

\begin{proof}
    By Corollary \ref{cor.product.H},
    \[B^bA^y = (A^0B^bC^0)(A^yB^0C^0) = A^{y + 2s\ell b\phi(y)} B^{b - 2s\ell\phi(b)y} C^{-by - 2s\ell\phi(b+1)\phi(y)}.
    \]

    Since $z,z^\prime\in Z_1(H) = \gen{C^s}$, then
    \begin{equation}\label{eq.H.2}
        [A^aB^xz,B^bA^yz^\prime]
        = [A^aB^x,B^bA^y]
        = [A^aB^x,A^{y + 2s\ell b\phi(y)} B^{b - 2s\ell\phi(b)y} C^{-by}].
    \end{equation}

    Applying Corollary \ref{cor.bigCommutator.H} to \eqref{eq.H.2}, using the fact that $\phi(n+2st)\equiv\phi(n)\mod s$ for $n,t\in\Z$, and recalling that $ab\equiv 1\mod s$ produces
    \begin{align*}
        \textup{exp} A & \equiv 2s\ell\{x\phi(y) - \phi(a)b - aby\}\equiv -s\ell(xy + 2y + a - 1)\mod 2u,        \\
        \textup{exp} B & \equiv 2s\ell\{a\phi(b) - \phi(x)y + xab\}\equiv s\ell(xy + 2x + b - 1)\mod 2u,         \\
        \textup{exp}\, C & \equiv ab - xy + 2s\ell\{\phi(y)\phi(x) - \phi(a)(\phi(b) + xb)-ay\phi(b)\}\mod 2u,
    \end{align*}
    where $\textup{exp}\, C\equiv 1\mod s$, so
    \begin{equation}\label{eq.H.3}
        \begin{split}
             &[A^aB^x,B^bA^y]\\
             &\quad = A^{2s\ell\{x\phi(y) - \phi(a)b - aby\}} B^{2s\ell\{a\phi(b) - \phi(x)y + xab\}} C^{ab - xy + 2s\ell\{\phi(y)\phi(x) - \phi(a)(\phi(b) + xb) -ay\phi(b)\}}\\
             &\quad\equiv A^{-s\ell(xy + 2y + a - 1)} B^{s\ell(xy + 2x + b - 1)} C\mod Z_1(H).
        \end{split}
    \end{equation}

    By Corollary \ref{cor.commutator.H}, $[B^x,A^{-s\ell(xy + 2y + a - 1)}] = C^{sx(a - 1)}$, $[A^a,B^{s\ell(xy + 2x + b - 1)}] = C^{sxa(y + 2) + sa\ell(b - 1)}$, and $[B^x,C] = B^{-2s\ell x} = B^{2s\ell x} $, so
    \begin{align*}
        (A^aB^x)^{[A^aB^xz,B^bA^yz^\prime]}
        & = (A^aB^x)^{A^{-s\ell(xy + 2y + a - 1)} B^{s\ell(xy + 2x + b - 1)} C} \\
        & = (A^aB^x)^{B^{s\ell(xy + 2x + b - 1)} C} C^{sx(a - 1)}               \\
        & = (A^aB^x)^C C^{sx(3a + ya - 1) + s\ell a(b - 1)}                     \\
        & = A^{\alpha a} B^{\alpha x} C^{sx(3a + ya - 1) + s\ell a(b - 1)}.
    \end{align*}

    On the other hand, by Note \ref{note.H.1}, $(A^aB^x)^\alpha = A^{\alpha a} B^{\alpha x} C^{-sxa\ell}$ and $(A^aB^x)^{2u} = 1$.

    Applying the automorphism $A\leftrightarrow B$ we get,
    \begin{align*}
        (B^bA^y)^{[B^bA^yz^\prime,A^aB^xz]}
        & = B^{\alpha b} A^{\alpha y} C^{-sy(3b + xb - 1) + s\ell b(a - 1)}, \\
        (B^bA^y)^\alpha
        & = B^{\alpha b} A^{\alpha y} C^{syb\ell},                          \\
        (B^bA^y)^{2u}
        & = 1.
    \end{align*}

    Thus, the assignment (\ref{eq.15}) extends to an endomorphism of $H$ if and only if $(A^aB^x)^{[A^aB^x,B^bA^y]} = (A^aB^x)^\alpha$ and
    $(B^bA^y)^{[B^bA^y,A^aB^x]} = (B^bA^y)^\alpha$ if and only if
    $C^{sx(3a + ya - 1) + s\ell a(b - 1)} = C^{-sxa\ell}$ and
    $C^{-sy(3b + xb + b\ell - 1) + s\ell b(a - 1)} = C^{syb\ell}$ if and only if
    \begin{align}
        x(3a + ya + a\ell - 1) + \ell a(b - 1) & \equiv 0\mod 2s,\label{eq.19}  \\
        y(3b + xb + b\ell - 1) + \ell b(a - 1) & \equiv 0\mod 2s.\label{eq.20}
    \end{align}

    Assume that \eqref{eq.19} and \eqref{eq.20} hold. We claim that $d=0=e$ and
    $a\equiv 1\equiv b\mod 2s$, or $d=1=e$ and $a\equiv 1+r\equiv b\mod 2s$.

    Suppose first that $d=0$. From \eqref{eq.19},  we obtain $\ell a(b - 1)\equiv 0\mod 2s$. As $\ell$ and $a$ are odd, we deduce
    $b-1\equiv 0\mod 2s$. Since $ab\equiv 1\mod s$, we infer $a\equiv 1\mod s$. Now, suppose, if possible, that $e=1$. From \eqref{eq.20}, we obtain $s + \ell(r + a - 1)\equiv 0\mod 2s$, and from this, $s\mid \ell(r + a - 1)$. Since $\ell$ is odd and $a\equiv 1\mod s$, this yields $s\mid r$ which is a contradiction. This proves that $e=0$. But then \eqref{eq.20} gives
    $\ell b(a - 1)\equiv 0\mod 2s$. Since $\ell$ and $b$ are odd, we derive $a-1\equiv 0\mod 2s$.

    Suppose next that $e=0$. Then, as indicated above, \eqref{eq.20} gives $a\equiv 1\mod 2s$, and as $ab\equiv 1\mod s$,
    we deduce $b\equiv 1\mod s$. If $d=1$ then \eqref{eq.19} yields the contradiction $s\mid r$. This forces
    $d=0$, which then implies $b\equiv 1\mod 2s$, as shown above.

    Thus $d=0$ if and only if $e=0$, in which case $a\equiv 1\equiv b\mod 2s$.

    Suppose next that $d=1=e$. Then \eqref{eq.19} and \eqref{eq.20} imply $a\equiv 1\equiv b\mod r$. Let $t_1,t_2\in\Z$ be such that $a = 1 + rt_1$, $b= 1 + rt_2$. Replacing in \eqref{eq.19} and \eqref{eq.20} and dividing by $r$ produces
    \begin{align}
        r(3t_1 + t_1\ell + rt_1 + t_2t_1\ell + 1) + \ell(t_2 + 1) + 2 & \equiv 0\mod 4,\label{eq.5.H} \\
        r(3t_2 + t_2\ell + rt_2 + t_1t_2\ell + 1) + \ell(t_1 + 1) + 2 & \equiv 0\mod 4,\label{eq.6.H}
    \end{align}
    which imply $t_1 + 1\equiv 0\equiv t_2 + 1\mod 2$. Let $q_1,q_2\in\Z$ be such that $t_1 = 2q_1 + 1$ and $t_2 = 2q_2 + 1$. Replacing this in \eqref{eq.5.H} and \eqref{eq.6.H} and dividing by 2 gives
    \begin{align*}
        \overline{r}(2q_1(3 + r + 2\ell) + 2q_2\ell + 4q_1q_2\ell + r + 2\ell + 4) + \ell(q_2 + 1) + 1 & \equiv 0\mod 2, \\
        \overline{r}(2q_2(3 + r + 2\ell) + 2q_1\ell + 4q_1q_2\ell + r + 2\ell + 4) + \ell(q_1 + 1) + 1 & \equiv 0\mod 2.
    \end{align*}

    From this we get $\ell(q_2 + 1) + 1\equiv 0\equiv\ell(q_1 + 1) + 1\mod 2$, which imply $q_1\equiv 0\equiv q_2\mod 2$. Thus $a\equiv 1 + r\equiv b\mod 2s$. This proves the claim.

    Suppose, conversely, that $d=0=e$ and
    $a\equiv 1\equiv b\mod 2s$, or $d=1=e$ and $a\equiv 1+r\equiv b\mod 2s$. We claim
    that \eqref{eq.19} and \eqref{eq.20} are true. Indeed, if $d=0=e$ and
    $a\equiv 1\equiv b\mod 2s$, we see directly that \eqref{eq.19} and \eqref{eq.20} hold.
    If $d=1=e$ and $a\equiv 1+r\equiv b\mod 2s$, then $a=1+r+sq_1$ and $b=1+r+sq_2$ for some even $q_1,q_2\in\Z$. After replacing $a$ and $b$ we readily see that \eqref{eq.19} and \eqref{eq.20} hold, as claimed.

    We finally claim that if $d=e=0$ and
    $a\equiv b\equiv 1\mod 2s$, or $d=e=1$ and $a\equiv 1+r\equiv b\mod 2s$, then the endomorphism of $H$, say $T$,
    extending \eqref{eq.15} is actually an automorphism. As $H$ is a finite nilpotent group, it suffices to show that
    the restriction of $T$ to $Z_1(H)=\langle C^s\rangle$ is a monomorphism. Now, from \eqref{eq.H.2}, \eqref{eq.H.3}, and Corollary \ref{cor.power.H}, we have
    \[
    (C^s)T
    = [A^aB^xz,B^bA^yz^\prime]^s
    = [A^aB^x,B^bA^y]^s
    = C^{s(ab-xy)}.
    \]
    As $ab-xy$ is odd, the restriction of $T$ to $Z_1(H)$ is a monomorphism, which completes the proof.
\end{proof}

\begin{cor}\label{ulti}
    Suppose $m>2$. Then the assignment
    $A\mapsto A^{1+r} B^r$,
    $B\mapsto B^{1+r} A^r$ extends to an automorphism, say $\Sigma$, of $H$, which clearly commutes with $\nu$.
\end{cor}

\begin{lemma}\label{mmlin}
    Suppose that $m=2$. Then $H$ is a group of order 512, with presentation
    $$
    \langle A,B,C\,|\, C=[A,B], A^C=A^5, B^C=B^5, A^8=B^8=C^8=1\rangle,
    $$
    where $Z_1(H) = \gen{C^2}$, $[A^2,C]=1=[B^2,C]$, and $[A^2,B^2]=1$. Moreover, if $z,w\in Z_1(H)$,  then
    the assignment $A\mapsto A B z$, $B\mapsto B w$ does not
    extend to an automorphism of $H$.
\end{lemma}

\begin{proof}
    The first statement follows by specializing general facts about $H$ to the case $m=2$
    (e.g. the relation $[A^2,B^2]=1$ follows from Corollary \ref{cor.commutator.H}). The given presentation
    shows, in particular, that the isomorphism type of $J$ is independent of $\ell$, and we take $\ell=1$ and $\alpha=5$.

    As for the second statement, we claim that $(ABz)^{[ABz,Bw]}\neq (ABz)^5$. Now $(ABz)^{[ABz,Bw]}=(AB)^{[AB,B]} z$
    and $(ABz)^5=(AB)^5 z$, so it suffices to show that $(AB)^{[AB,B]}\neq (AB)^5$.

    Using \eqref{comfor} and Corollary \ref{cor.commutator.H}, we see that $[AB,B] = C^B = B^4 C$, so
    $$
    (AB)^{[AB,B]}=(AB)^{B^4 C}=(A^{B^4} B)^C=(AC^4 B)^C=A^5 B^5 C^4.
    $$

    On the other hand, using Note \ref{note.H.1} gives
    \[
    (AB)^5
    = A^5 B^5 C^{-\phi(5) + 4\sigma_2(1,5)}
    = A^5 B^5 C^2.
    \]

    Thus, the normal form of the
    elements of $H$ yields that $(AB)^{[AB,B]}\neq (AB)^5$ (with equality modulo $Z_1(H)$, in agreement with Proposition \ref{newauto}).
\end{proof}

Set $\Sigma_1\in \mathrm{Aut}(H)$ to be $\Sigma$ if $m>2$ and $1_H$ if $m=2$.

\begin{theorem}\label{main3}
    The canonical map $\zeta:\mathrm{Aut}(H)\to\mathrm{Aut}(H/Z_3(H))$ has kernel $\mathrm{Aut}_3(H)$
    and image $\langle \Sigma_1,\nu\rangle^\zeta$, so that
    $\mathrm{Aut}(H)=\mathrm{Aut}_3(H)\langle \Sigma_1,\nu\rangle=\mathrm{Aut}_2(H)\mathrm{Inn}(H)\langle \Gamma,\Sigma_1,\nu\rangle$.
    Moreover, we have $\mathrm{Aut}(H)/\mathrm{Aut}_3(H)\cong (\Z/2\Z)^2$ if $m>2$ and
    $\mathrm{Aut}(H)/\mathrm{Aut}_3(H)\cong \Z/2\Z$ if $m=2$. Thus
    the order of $\mathrm{Aut}(H)$ is $2^{8m-2}$ if $m>2$ and $2^{13}$ if $m=2$.
\end{theorem}

\begin{proof}
    As $J$ is nilpotent of class 5, the nilpotency classes of $H$ and $K$ are 4 and 3, respectively.
    Thus, Proposition \ref{zi2} gives an imbedding
    $\widetilde{\Lambda}:\mathrm{Aut}(H)/\mathrm{Aut}_3(H)\to \mathrm{Aut}(K)/\mathrm{Aut}_2(K)$.
    We appeal to Theorem \ref{autk6}, its proof, and notation therein.

    We have an
    imbedding $\eta:\mathrm{Aut}(K)/\mathrm{Aut}_2(K)\to\GL_2(\Z/2^{m-1}\Z)$, which
    yields an imbedding $\widetilde{\Lambda}\eta:\mathrm{Aut}(H)/\mathrm{Aut}_3(H)\to \GL_2(\Z/2^{m-1}\Z)$. Set
    $\chi=\zeta\widetilde{\Lambda}\eta$ and $I=\mathrm{Im}(\chi)$.
    As $\widetilde{\Lambda}\eta$ is injective, $\mathrm{Im}(\zeta)=\langle \Sigma_1,\nu\rangle^\zeta$
    if and only if $I=\langle \Sigma_1,\nu\rangle^\chi$. Note also that $\nu^{\chi}=Q$.

    Suppose first $m>2$ and set $Z=M_{1+2^{m-2}} R T$. Then
    $\Sigma^{\chi}=Z$, and we must show that $I=\langle Q,Z\rangle$.
    Now $W=V\langle Q\rangle$, with $Q\in I$, so
    it suffices to show that $I\cap V=\langle Z\rangle$. Here
    $V$ is an abelian group of order $2^m$. By Proposition \ref{muchos}, $I\cap\langle U\rangle$ is trivial.
    It remains to analyze the cosets $\langle U\rangle R$, $\langle U\rangle T$, and $\langle U\rangle RT$, inside of $V$.

    We claim that none of the elements in $\langle U\rangle R$ is in $I$, which translates as follows: none of the assignments
    $A\mapsto A^q B^{2^{m-2}}z$, $B\mapsto B^t w$, where $z,w\in Z_1(H)$, $qt\equiv 1\mod 2^{m-1}$, $1\leq q,t\leq 2^{m-1}$,
    extends to an automorphism of $H$. This follows from Proposition \ref{muchos}.

    Since $R=T^Q$ and $U=U^Q$, it follows that none of the elements in $\langle U\rangle T$ is in $I$ either.

    We next claim that if $M_q RT\in I$ for any odd $q$ such that $1\leq q\leq 2^{m-1}$, then $q=1+2^{m-2}$.
    This translates as follows: if the assignment
    $A\mapsto A^q B^{2^{m-2}}z$, $B\mapsto B^t A^{2^{m-2}} w$, where $z,w\in Z_1(H)$, $qt\equiv 1\mod 2^{m-1}$, and $1\leq q,t\leq 2^{m-1}$,
    extends to an automorphism of $H$ then $q=1+2^{m-2}$. This follows from Proposition \ref{muchos}.

    Clearly, $\langle Z,Q\rangle$ is isomorphic to the Klein 4-group,
    so $|\mathrm{Aut}(H)|=|\mathrm{Aut}_3(H)|\times|\langle Z,Q\rangle|=2^{8m-2}$ by Theorem \ref{tamanio}.

    Suppose next $m=2$. Since $\GL_2(\Z/2\Z)$ is the disjoint union of
    $\langle Q\rangle$, $\langle Q\rangle R$, and $\langle Q\rangle T$,
    we must show that no element of the last 2 cosets
     is in $I$. Since $Q$ conjugates $R$ and $T$ into each other, it suffices to show that $R$ is not in $I$,
    which means that the assignment $A\mapsto A B z$, $B\mapsto B w$, where $z,w\in Z_1(H)$, does not
    extend to an automorphism of $H$. This follows from Lemma \ref{mmlin}.

    Thus Theorem \ref{tamanio} gives $|\mathrm{Aut}(H)|=2|\mathrm{Aut}_3(H)|=2^{8m-3}=2^{13}$.
\end{proof}

We next obtain a description of $\mathrm{Aut}(H)$ in terms of group extensions.

\begin{theorem}\label{gpexth}
    We have the following series of normal subgroups of $\mathrm{Aut}(H)$:
    $$
    1\subset \mathrm{Aut}_2(H)\subset \mathrm{Inn}(H)\mathrm{Aut}_2(H)\subset\mathrm{Aut}_3(H)\subset \mathrm{Aut}(H),
    $$
    where the factors
    $\mathrm{Aut}_2(H)\cong (\Z/2^{m-1}\Z)^4\times (\Z/2^{m}\Z)^2$,  $\mathrm{Inn}(H)\mathrm{Aut}_2(H)/\mathrm{Aut}_2(H)\cong
    (\Z/2^{m-1}\Z)^2$, $\mathrm{Aut}_3(H)/\mathrm{Inn}(H)\mathrm{Aut}_2(H)\cong (\Z/2\Z)^2$, and
    $\mathrm{Aut}(H)/\mathrm{Aut}_3(H)$ is isomorphic to $(\Z/2\Z)^2$ if $m>2$ and $\Z/2\Z$ if $m=2$.
\end{theorem}

\begin{proof}
    We know that $Z_2(H)$ is abelian.
    Moreover, by Proposition \ref{auth}, $\mathrm{Aut}_2(H)\cong Z_2(H)\times Z_2(H)$, where
    $Z_2(H)=\langle A^{2^{m}}, B^{2^{m}}, C^{2^{m-1}}\rangle\cong
    (\Z/2^{m-1}\Z)^2\times \Z/2^{m}\Z$ by the normal form of the elements of $H$ and the orders of $A,B,C$. Thus
    $\mathrm{Aut}_2(H)\cong (\Z/2^{m-1}\Z)^4\times (\Z/2^{m}\Z)^2$. Next, we have
    $$
    \mathrm{Inn}(H)\mathrm{Aut}_2(H)/\mathrm{Aut}_2(H)\cong\mathrm{Inn}(H)/\mathrm{Aut}_2(H)\cap \mathrm{Inn}(H)\cong H\delta/Z_3(H)\delta\cong H/Z_3(H)\cong (\Z/2^{m-1}\Z)^2.
    $$

    That $\mathrm{Aut}_3(H)/\mathrm{Inn}(H)\mathrm{Aut}_2(H)\cong (\Z/2\Z)^2$ was shown in Theorem \ref{tamanio}.
    Finally, Theorem \ref{main3} gives the structure of $\mathrm{Aut}(H)/\mathrm{Aut}_3(H)$.
\end{proof}






\section{The automorphism group of $J$}\label{autojota}

Throughout this section we maintain the notation introduced in Sections \ref{backJ} and \ref{section.formulas.J} and recall the numerical
information given in Section \ref{sizeJ}. Also, recall that the exponent of $J$ is $4us$, which implies that
if $x,y\in  J$, then the assignment $A\mapsto x$, $B\mapsto y$ extends to an endomorphism of $J$ if and only if it preserves the first and second defining relations of $J$.

\begin{prop}\label{autj1}
    For every $x,y\in Z_1(J)$ the assignment $A\mapsto Ax,\; B\mapsto By$ extends to a central automorphism $\Omega_{(x,y)}$ of $J$ that fixes $Z_2(J)$ pointwise. Moreover, the corresponding map
    $\Omega:Z_1(J)\times Z_1(J)\to\mathrm{Aut}(J)$ is a group monomorphism whose image is $\mathrm{Aut}_1(J)$. In particular,
    $|\mathrm{Aut}_1(J)|=2^{2m}$.
\end{prop}

\begin{proof}
    This follows as in the proof of Proposition \ref{autk}.
\end{proof}

\begin{prop}\label{autjo}
    For every $x,y\in Z_2(J)$ the assignment
    $A\mapsto Ax,\; B\mapsto By$ extends to a 2-central automorphism $\Psi_{(x,y)}$ of $J$ that fixes $Z_1(J)$ and $Z_2(J)/Z_1(J)$ pointwise.  Thus, $|\mathrm{Aut}_2(J)|=2^{4m}$ and the map
    $g:Z_2(J)\times Z_2(J)\to \mathrm{Aut}_2(J)/\mathrm{Aut}_1(J)$, given by
    $(x,y)\mapsto \Psi_{(x,y)}\mathrm{Aut}_1(J)$ is a group epimorphism with kernel $Z_1(J)^2$, so that
    $\mathrm{Aut}_2(J)/\mathrm{Aut}_1(J)\cong (Z_2(J)/Z_1(J))^2$.
\end{prop}

\begin{proof}
    Let $x\in Z_2(J)$. Then $x = A^{2ui} C^{sk}$ for some $i,k\in\Z$. Let us first show that
    \begin{equation}\label{7.eq.task3}
        (Ax)^{[A,B]} = (Ax)^\alpha,\;
        (Bx)^{[B,A]} = (Bx)^\alpha.
    \end{equation}

    Applying Theorem \ref{thm.power.J} to $(Ax)^\alpha = (A^{1+2ui} C^{sk})^\alpha = A^{\textup{exp} A + \xi} B^{\textup{exp} B} C^{\textup{exp} C}$ yields $\textup{exp} A\equiv\alpha + 2ui + 2usk\mod 4us$, $\textup{exp} B\equiv 0\mod 4us$, $\textup{exp}\, C\equiv sk + 2uk\mod 4u$, $\xi\equiv 0\mod 4us$, so
    \[
    (Ax)^\alpha
    = A^{\alpha + 2ui + 2usk} C^{sk + 2uk}
    = A^{\alpha + 2ui} C^{sk}.
    \]

    On the other hand,
    \[
    (Ax)^C
    = (A^{1+2ui} C^{sk})^C
    = A^{\alpha(1+2ui)} C^{sk}
    = A^{\alpha + 2ui} C^{sk}.
    \]

    Thus $(Ax)^{[A,B]} = (Ax)^\alpha$ and, by the automorphism $A\leftrightarrow B$, $(Bx)^{[B,A]} = (Bx)^\alpha$.

    Now, taking $x,y\in Z_2(J)$ and working modulo $Z_1(J)$, we obtain
    \[
    [Ax,By] = [A,B]z
    \]
    for a unique $z\in Z_1(J)$. Then
    \begin{equation}
        \label{8.eq.task3}
        (Ax)^{[Ax,By]} = (Ax)^{[A,B]z} = (Ax)^{[A,B]},\;
        (By)^{[By,Ax]} = (By)^{[B,A]z} = (By)^{[B,A]}.
    \end{equation}

    By \eqref{7.eq.task3} and \eqref{8.eq.task3}, it follows that the defining relations of $J$ are preserved. Thus the above
	assignment extends to an endomorphism $\Psi_{(x,y)}$ of $J$.

    Applying Theorem \ref{thm.power.J} to $(A^{1 + 2ui} C^{sk})^{2u} = A^{\textup{exp} A + \xi} B^{\textup{exp} B} C^{\textup{exp} C}$ produces, $\textup{exp} A\equiv 2u\mod 4us$, 
		$\textup{exp} B\equiv 0\mod 4us$, $\textup{exp}\, C\equiv 0\mod 4u$, $\xi\equiv 0\mod 4us$, so
    \[
    (A^{2u})\Psi_{(x,y)}
    = (Ax)^{2u}
    = (A^{1 + 2ui} C^{sk})^{2u}
    = A^{2u}
    \]
    and $\Psi_{(x,y)}$ restricts to the identity map on $Z_1(J)$. As $J$ is a finite nilpotent group, it follows that $\Psi_{(x,y)}$
	is an automorphism of $J$ that fixes $Z_1(J)$ pointwise. Moreover, since $\Psi_{(x,y)}$ sends $C$ to $Cz$, with $z\in Z_1(J)$,
	we see that $\Psi_{(x,y)}$ fixes $C^{2^{m-1}}$ modulo  $Z_1(J)$, whence $\Psi_{(x,y)}$ fixes $Z_2(J)/Z_1(J)$ pointwise.
	This implies that $g$ is a group homomorphism, whose kernel clearly contains $Z_1(J)^2$.

    On the other hand, by definition, every element of $\mathrm{Aut}_2(J)$
	has the form $\Psi_{(x,y)}$ for a unique $(x,y)\in Z_2(J)\times Z_2(J)$, so $g$ is an epimorphism,
	$|\mathrm{Aut}_2(J)|=2^{4m}$, and
	$|\mathrm{Aut}_2(J)/\mathrm{Aut}_1(J)| = |Z_2(J)/Z_1(J)|^2$, which implies that $\ker(g)=Z_1(J)\times Z_1(J)$.
\end{proof}

We have the following series of normal subgroups of $\mathrm{Aut}(J)$:
\begin{equation}\label{seriedejota}
1\subset \mathrm{Aut}_1(J)\subset\mathrm{Aut}_2(J)\subset \mathrm{Aut}_3(J)\subset
\mathrm{Inn}(J)\mathrm{Aut}_3(J)\subset\mathrm{Aut}_4(J)\subset \mathrm{Aut}(J).
\end{equation}

Here $\mathrm{Aut}_1(J)\cong (\Z/2^m\Z)^2$ and $\mathrm{Aut}_2(J)/\mathrm{Aut}_1(J)
\cong (\Z/2^m\Z)^2$ by Propositions \ref{autj1} and \ref{autjo}.
Regarding $\mathrm{Aut}_3(J)/\mathrm{Aut}_2(J)$, by Proposition \ref{zi2}, we have imbeddings
$$
\mathrm{Aut}_3(J)/\mathrm{Aut}_2(J)\hookrightarrow \mathrm{Aut}_2(H)/\mathrm{Aut}_1(H)\hookrightarrow
\mathrm{Aut}_1(K)\cong Z_1(K)\times Z_1(K).
$$
We proceed to find the image of this composite imbedding and the structure of $\mathrm{Inn}(J)\mathrm{Aut}_3(J)/\mathrm{Aut}_3(J)$.

\begin{theorem}\label{autj3}
    The assignments
    $A\mapsto AB^{2s},\; B\mapsto B^{1 - 4s} A^{2s}$ and
    $A\mapsto A,\; B\mapsto BA^u$ extend to automorphisms, say $\Delta_1$ and $\Delta_2$, of $J$. Moreover,
    $$
    \mathrm{Inn}(J)\mathrm{Aut}_3(J)=\langle \Delta_1,\Delta_2\rangle \mathrm{Inn}(J)\mathrm{Aut}_2(J),\;
    \mathrm{Aut}_3(J)=\langle \Delta_1,\Delta_2, C\delta\rangle\mathrm{Aut}_2(J),
    $$
    $$
    \mathrm{Aut}_3(J)/\mathrm{Aut}_2(J)\cong(\Z/2^{m-1}\Z)^2\times \Z/2\Z,\;
		\mathrm{Inn}(J)\mathrm{Aut}_3(J)/\mathrm{Aut}_3(J)\cong (\Z/2^{m-1}\Z)^2,
    $$
    $$
    |\mathrm{Aut}_3(J)|=2^{6m-1},\; |\mathrm{Inn}(J)\mathrm{Aut}_3(J)|=2^{8m-3}.
    $$
\end{theorem}

\begin{proof}
    Applying Proposition \ref{zi2} twice, we obtain
    $$
    \mathrm{Aut}_3(J)/\mathrm{Aut}_2(J)\hookrightarrow
    \mathrm{Aut}_2(H)/\mathrm{Aut}_1(H)\hookrightarrow
    \mathrm{Aut}_1(K),
    $$
    where $\mathrm{Aut}_1(K)\cong Z_1(K)\times Z_1(K)$ by Proposition \ref{autk}. In this way we obtain a homomorphism
    $\psi:\mathrm{Aut}_3(J)\to Z_1(K)^2$ such that $\ker(\psi)=\mathrm{Aut}_2(J)$. We proceed to show that
    \begin{equation}\label{impsi}
        \mathrm{Im}(\psi)
        = \langle (a^{2s},b^{-2s}), (b^{2s},a^{2s} b^{-4s}), (1,a^u)\rangle.
    \end{equation}

    Let $i,j,g,h\in\Z$ and $w,z\in Z_2(J)$. We need to determine under what conditions the assignment
    $$
    A\mapsto AA^{2si}B^{2sj}w,\; B\mapsto B A^{2sg} B^{2sh}z
    $$
    extends to an automorphism of $J$. Skillfully using Propositions \ref{autj1} and \ref{autjo}, we see that this is the case if and only if
    \begin{equation}\label{pupicat}
        A\mapsto A^{1+2si}B^{2sj},\;
        B\mapsto B A^{2sg} B^{2sh}
    \end{equation}
    extends to an automorphism of $J$. We will implicitly use in what follows that $Z_3(J)$ is abelian, as indicated in Section \ref{backJ}.

    By Theorem \ref{thm.commutators.J}, $[B,A^{2sg}] = (A^{2u\ell g} C^{2sg})^{-1} = A^{2u\ell g} C^{-2sg}$. Then
    \[
    B A^{2sg} B^{2sh}
    = A^{2sg} B A^{2u\ell g} C^{-2sg} B^{2sh}
    = A^{2sg} B^{1+2sh} C^{-2sg} A^{2u\ell g}
    \]
    and
    \[
    [A^{1+2si}B^{2sj},B A^{2sg} B^{2sh}]
    = [A^{1+2si}B^{2sj},A^{2sg} B^{1+2sh} C^{-2sg}].
    \]

    Applying Proposition \ref{prop.bigCommutator.J} to the above commutator produces, $\textup{exp} A\equiv 0\mod 2u$, $\textup{exp} B\equiv 0\mod 2u$, $\textup{exp}\, C\equiv 1 + 2s(i+h)\mod 2u$, so
    \[
    [A^{1+2si}B^{2sj},B A^{2sg} B^{2sh}]
    \equiv A^{\textup{exp} A} B^{\textup{exp} B} C^{\textup{exp} C}
    \equiv C^{1 + 2s(i+h)}\mod Z_1(J).
    \]

    Thus
    \[
    (A^{1+2si}B^{2sj})^{[A^{1+2si}B^{2sj},B A^{2sg} B^{2sh}]}
    = (A^{1+2si}B^{2sj})^{C^{1 + 2s(i+h)}}.
    \]

    As $[A^{1+2si},C^{1+2s(i+h)}] = A^{2s\ell + 4u\ell h + 8u\ell i}$ and $[B^{2sj},C^{1+2s(i+h)}] = B^{-4u\ell j}$, then
    \[
    (A^{1+2si}B^{2sj})^{[A^{1+2si}B^{2sj},BA^{2sg}B^{2sh}]}
    = A^{1+2s(i+\ell) + 4u\ell h + 8u\ell i}B^{2sj - 4u\ell j}
    = A^{1+2s(i+\ell) + 4u\ell (h + j + 2i)}B^{2sj}.
    \]

    On the other hand, applying Theorem \ref{thm.power.J} to $(A^{1+2si} B^{2sj})^\alpha = A^{\textup{exp} A + \xi} B^{\textup{exp} } 
		C^{\textup{exp} C}$ gives
    $\textup{exp} A\equiv 1 + 2s(i+\ell) + 4u\ell i\mod 4us$,
    $\textup{exp} B\equiv 2sj + 4u\ell j + 2usj\mod 4us$,
    $\textup{exp}\, C\equiv 2uj\mod 4u$, and
    $\xi\equiv 0 \mod 4us$, so
    \[
    (A^{1+2si} B^{2sj})^\alpha
    = A^{1 + 2s(i+\ell) + 4u\ell i} B^{2sj + 4u\ell j + 2usj} C^{2uj}
    = A^{1 + 2s(i+\ell) + 4u\ell(i-j)} B^{2sj}.
    \]

    Thus $(A^{1+2si}B^{2sj})^{[A^{1+2si}B^{2sj},BA^{2sg}B^{2sh}]} = (A^{1+2si}B^{2sj})^\alpha$ is equivalent to
    \begin{equation}
        \label{eq.J.1}
        A^{4u\ell(2j + h + i)} = 1.
    \end{equation}

    By the automorphism $A\leftrightarrow B$ we have that $(BA^{2sg}B^{2sh})^{[BA^{2sg}B^{2sh},A^{1+2si}B^{2sj}]} = (BA^{2sg}B^{2sh})^\alpha$ is equivalent to
    \begin{equation}
        \label{eq.J.2}
        B^{4u\ell(2g + h + i)} = 1.
    \end{equation}

    Note that \eqref{eq.J.1} and \eqref{eq.J.2} hold if and only if
    \begin{equation}
        \label{roj}
        2j + h + i\equiv 0\mod s,\quad
        2g + h + i\equiv 0\mod s.
    \end{equation}

    Thus, \eqref{pupicat} extends to an endomorphism, say $\Upsilon$, of $J$ if and only if \eqref{roj} holds,
    in which case $\Upsilon$ is an automorphism.
    Indeed, applying Theorem \ref{thm.power.J} to $(A^{1+2si}B^{2sj})^{2u} = A^{\textup{exp} A + \xi} B^{\textup{exp} B} C^{\textup{exp} C}$ yields,
    $\textup{exp} A\equiv 2u\mod 4us$,
    $\textup{exp} B\equiv 0\mod 4us$,
    $\textup{exp}\, C\equiv 0\mod 4u$,
    $\xi\equiv 0\mod 4us$, so
    \[
    (A^{2u})\Upsilon
    = (A^{1+2si}B^{2sj})^{2u}
    = A^{2u},
    \]
    which ensures that the restriction of $\Upsilon$ to $Z_1(J)$ is the identity map and, since $J$ is nilpotent, this guarantees that $\ker(\Upsilon)$ is trivial.

    Now \eqref{roj} implies $g\equiv j\mod 2^{m-2}$ and $h\equiv -i-2j\mod 2^{m-1}$, so that
    $g = j + t2^{m-2}$ and $h = - i - 2j + q2^{m-1}$ with $t,q\in\Z$. Thus, the image of $\psi$ consists of all
    $$
    (a^{2^m},b^{-2^m})^i (b^{2^m}, a^{2^m} b^{-2^{m+1}})^j (1,a^{2m-2})^t,\quad i,j,t.
    $$

    Here $(a^{2^m},b^{-2^m}),(b^{2^m}, a^{2^m} b^{-2^{m+1}}), (1,a^{2m-2})$ generate a subgroup
    of $Z_1(K)^2$ isomorphic to $(\Z/2^{m-1}\Z)^2\times \Z/2\Z$. Thus
    the image of $\psi$, and hence $\mathrm{Aut}_3(J)/\mathrm{Aut}_2(J)$, is as required.
    We also see that $\Delta_1$ and $\Delta_2$ are automorphisms, and that $(C\delta)^\psi$
    and $(a^{2^m},b^{-2^m})$ generate the same subgroup of $Z_1(K)^2$. We conclude that
    $|\mathrm{Aut}_3(J)|=|\mathrm{Aut}_2(J)|\times 2^{2m-1}=2^{6m-1}$, by Proposition~\ref{autjo}. Finally,
    $$
    \mathrm{Inn}(J)\mathrm{Aut}_3(J)/\mathrm{Aut}_3(J)\cong
    \mathrm{Inn}(J)/\mathrm{Aut}_3(J)\cap \mathrm{Inn}(J)\cong J\delta/Z_4(J)\delta\cong J/Z_4(J)\cong (\Z/{2^{m-1}}\Z)^2,
    $$
    so
    $$
    |\mathrm{Inn}(J)\mathrm{Aut}_3(J)|=2^{6m-1}\times 2^{2m-2}=2^{8m-3}.\qedhere
    $$
\end{proof}

We next determine the last two factors of \eqref{seriedejota}. It turns out that
$\mathrm{Aut}_4(J)/\mathrm{Inn}(J)\mathrm{Aut}_3(J)$ is isomorphic to $\Z/2\Z$ and
$\mathrm{Aut}(J)/\mathrm{Aut}_4(J)$ is isomorphic to $(\Z/2\Z)^2$ if $m>2$ and to $\Z/2\Z$ if $m=2$.
This will take considerable effort.

\begin{prop}\label{gammasnosonauto} The assignment $A\mapsto A^{1+s}u_1,\, B\mapsto A^s Bu_2$, where 
$u_1,u_2\in Z_3(J)$, does not extend to an endomorphism of $J$.

\end{prop}

\begin{proof} Suppose, if possible, that the given assignment extends to an endomorphism, say $\Gamma$, of~$J$.
We first consider the case when $u_1 =A^{2si}B^{2sj}$, $u_2 = A^{2sa}B^{2sb}$.
Applying Theorem \ref{thm.product.J} to $(A^0BC^0)(A^{2sa}B^{2sb}C^0) = A^{\textup{exp} A + \xi} B^{\textup{exp} B} C^{\textup{exp} C}$, gives
    $\textup{exp} A\equiv 2sa\mod 2u$,
    $\textup{exp} B\equiv 1+2sb\mod 2u$,
    $\textup{exp}\, C\equiv -2sa\mod 2u$, and
    $\xi\equiv 0\mod 2u$, so
    \[
    B\Gamma\equiv A^{s+2sa}B^{1+2sb}C^{-2sa}\mod Z_1(J),
    \]
    \[
    [A\Gamma,B\Gamma]
    = [A^{1+s+2si}B^{2sj},A^{s+2sa}B^{1+2sb}C^{-2sa}].
    \]

    By Proposition \ref{prop.bigCommutator.J} applied to $[A\Gamma,B\Gamma] = A^{\textup{exp} A} B^{\textup{exp} B} C^{\textup{exp} C}$, we get
    $\textup{exp} A\equiv u\mod 2u$,
    $\textup{exp} B\equiv 0\mod 2u$,
    $\textup{exp}\, C\equiv 1 + s + 2s(i + b)\mod 2u$, so
    \[
    [A\Gamma,B\Gamma]\equiv A^{u} C^{1 + s + 2s(i + b)}\mod Z_1(J).
    \]
    Thus, using that $Z_3(J)$ is abelian, we obtain
    \[
    (A\Gamma)^{[A\Gamma,B\Gamma]}
    = (A^{1+s+2si}B^{2sj})^{A^{u} C^{1 + s + 2s(i + b)}}=(A^{1+s+2si}B^{2sj})^{C^{1 + s + 2s(i + b)}}.
    \]
    A careful application of
    Theorem \ref{thm.commutators.J} gives
    \[
    (A\Gamma)^{[A\Gamma,B\Gamma]}
    = A^{1 + s + 2si + 2s\ell + 4u\ell(2i + b + 1)}B^{2sj - 4u\ell j} = A^{\alpha + s + 2si + 4u\ell(2i + b + j+1)}B^{2sj}.
    \]

    On the other hand, applying Theorem \ref{thm.power.J} to $(A^{1+s+2si}B^{2sj})^\alpha = A^{\textup{exp} A + \xi} B^{\textup{exp} B} C^{\textup{exp} C}$ gives
    $\textup{exp} A\equiv 1 + 2s\ell + s + 2si + 2u\ell + 4u\ell i\mod 4us$,
    $\textup{exp} B\equiv 2sj + 4u\ell j + 2usj\mod 4us$,
    $\textup{exp}\, C\equiv 2uj\mod 4u$,
    $\xi\equiv 0\mod 4us$, so
    \[
    (A\Gamma)^\alpha
    = (A^{1+s+2si}B^{2sj})^\alpha
    = A^{\alpha + s + 2si + 2u\ell + 4u\ell i} B^{2sj + 4u\ell j + 2usj} C^{2uj}
    = A^{\alpha + s + 2si + 2u\ell + 4u\ell(i-j)} B^{2sj}
    \]
    Since $2u\ell\not\equiv 0\mod 4u$, $(A\Gamma)^{[A\Gamma,B\Gamma]}\neq (A\Gamma)^\alpha$, so $\Gamma$ does not extend to an 
		endomorphism of $J$ in this case.

In general, we have $u_1=v_1z_1$ and $u_2=v_2z_2$, where $v_1=A^{2se_1}B^{2sf_1}$, $v_2=A^{2se_2}B^{2sf_2}$, and $z_1,z_2\in Z_2(J)$.
For any odd $t$, the map $x\mapsto x^t$ defines an automorphism of $Z_2(J)$ and $Z_1(J)$. Thus,
    we can find $x_1,x_2\in Z_2(J)$ such that $x_1^{1+s+2sh_1+2sk_1}=z_1^{-1}$, $x_2^{1+s+2sh_2+2sk_2}=z_2^{-1}$. It follows from
    Proposition \ref{autjo} that $\Gamma\Psi_{(x_1,x_2)}$ sends $A$ to $A^{1+s}v_1 y_1$ and $B$ to $A^{s}B v_2 y_2$, where
    $y_1,y_2\in Z_1(J)$. Next we find $q_1,q_2\in Z_1(J)$ such that $q_1^{1+s+2sh_1+2sk_1}=y_1^{-1}$ and
    $q_2^{1+s+2sh_2+2sk_2}=y_2^{-1}$. We deduce from
    Proposition \ref{autj1} that $\Gamma\Psi_{(x_1,x_2)}\Omega_{(q_1,q_2)}$ sends $A$ to $A^{1+s}v_1$ and $B$ to $A^{s} B v_2$,
		which contradicts the previous case.
\end{proof}

\begin{prop}\label{delta3isauto}
    The assignment $A\mapsto A^{1+s} B^s,\; B\mapsto B^{1 + s(2\delta_{2,m} + s - 3)} A^s$,
    where $\delta_{i,j}$ is the Kronecker delta function, extends to an automorphism, say $\Delta_3$, of $J$.
\end{prop}

\begin{proof} 
    Applying Theorem \ref{thm.product.J} to
    $B\Delta_3 = (A^0 B^{1 + s(2\delta_{2,m} + s - 3)} C^0) (A^s B^0 C^0)$, we obtain
		\begin{equation}\label{bd3}
		B\Delta_3\equiv A^{s+ u} B^{1 - 3s + 2s\delta_{2,m} + u} C^{-s + ur}\mod Z_1(J),
		\end{equation}
		where the exponents of $A,B,C$ were reduced modulo $2u$. Thus
    \[
    [A\Delta_3,B\Delta_3]
    = [A^{1+s} B^s, A^{s + u} B^{1 - 3s + 2s\delta_{2,m} + u} C^{-s + ur}].
    \]

    By Proposition \ref{prop.bigCommutator.J} applied to $[A\Delta_3,B\Delta_3]$, we get
    \[
    [A\Delta_3,B\Delta_3]\equiv A^u B^u C^{1 + 2s(\delta_{2,m}-1) + u}\mod Z_1(J).
    \]
		
		By Theorem \ref{thm.commutators.J},  $[B^s,A^u] =A^{u^2}$ and $[B,A^u] =A^{-us\ell}C^{-u}$, so
		$[B^{1+s},A^u]=A^{-us\ell+u^2}C^{-u}$ by (\ref{comfor}), whence $[A^{1+s},B^u]=B^{-us\ell+u^2}C^{u}=A^{us\ell+u^2}C^{u}$
		using  $\theta$, $B^{2u}A^{2u}=1$ and $A^{4us}=1$. Theorem \ref{thm.commutators.J} also gives $[B^s,C^{2s}]=1$,
which implies $[B^s,C^{u}]=1$ and $[B^s,C^{1+2s(\delta_{2,m}-1)+u}] = [B^s,C]=B^{-2u\ell}=A^{2u\ell}$.
Another application of Theorem \ref{thm.commutators.J} yields 
$[A^{1+s},C^{1+2s(\delta_{2,m}-1)+u}] = A^{2s\ell + 2u\ell + 4u\ell(\delta_{2,m}-1) + 2us}$. Using these commutators yields
    \begin{align*}
        (A\Delta_3)^{[A\Delta_3,B\Delta_3]}
        &= (A^{1+s} B^s)^{A^u B^u C^{1 + 2s(\delta_{2,m}-1) + u}}\\
        &= (A^{1+s} B^s)^{B^u C^{1 + 2s(\delta_{2,m}-1) + u}}A^{u^2}\\
        &= (A^{1+s} C^u B^s)^{C^{1 + 2s(\delta_{2,m}-1) + u}} A^{us\ell}\\
        &= A^{\alpha + s} B^s C^u A^{4u\ell\delta_{2,m} + us\ell + 2us}.
    \end{align*}

	On the other hand, applying Theorem \ref{thm.power.J} to 
	$(A\Delta_3)^\alpha = (A^{1+s} B^s)^\alpha = A^{\textup{exp} A + \xi} B^{\textup{exp} B} C^{\textup{exp} C}$ gives
    $\textup{exp} A\equiv\alpha + s + 2u\ell + 2us + u^2\mod 4us$,
    $\textup{exp} B\equiv s + 2u\ell + us\ell^2\mod 4us$, 
    $\textup{exp}\, C\equiv -u\ell + us\mod 4u$, and
    $\xi\equiv u^2\mod 4us$, so
    \[
    (A\Delta_3)^\alpha
    = A^{\alpha + s + 2u\ell + 2us + u^2} B^{s + 2u\ell + us\ell^2} C^{-u\ell + us} A^{u^2}
    = A^{\alpha + s} B^{s} C^{-u\ell} A^{-us\ell^2 + 2us + u^2}.
    \]

    As $-u\ell\equiv 4u - u\ell\equiv u + u(3-\ell)\mod 4u$, then $(A\Delta_3)^\alpha = A^{\alpha + s} B^{s} C^{u} A^{us(3 - \ell^2 - \ell) + 2us + u^2}$. Therefore $(A\Delta_3)^{[A\Delta_3,B\Delta_3]} = (A\Delta_3)^\alpha$ if and only if $A^{4u\ell\delta_{2,m} + us(\ell^2 - 1) + u^2} = 1$ if and only if $4s\mid 4\ell\delta_{2,m} + s(\ell^2 - 1) + u$. If $m=2$, then $s=2$ and $4s$ is a factor of 
		$2(\ell+1)^2 = 4\ell\delta_{2,m} + s(\ell^2 - 1) + u$. If $m\geq 3$, then $s\geq 4$ and $4s$ is a factor of $s\{(\ell-1)(\ell+1) + s\} = 4\ell\delta_{2,m} + s(\ell^2 - 1) + u$. Thus $(A\Delta_3)^{[A\Delta_3,B\Delta_3]} = (A\Delta_3)^\alpha$.

    As for the second relation, we have
    \begin{equation}\label{bd4}
    [B\Delta_3,A\Delta_3]
    \equiv (A^u B^u C^{1 + 2s(\delta_{2,m}-1) + u})^{-1}
    \equiv C^{-1 - 2s(\delta_{2,m}-1) - u} B^{-u} A^{-u}\mod Z_1(J).
    \end{equation} 
    
		To conjugate $B\Delta_3$ by $[B\Delta_3,A\Delta_3]$ we need a sharpening of (\ref{bd3})
		and information on how each factor of (\ref{bd4}) conjugates $B\Delta_3$. Applying Theorem \ref{thm.product.J} to
    $B\Delta_3 = (A^0 B^{1 + s(2\delta_{2,m} + s - 3)} C^0) (A^s B^0 C^0)$, we obtain
		\begin{align*}
        B\Delta_3
        &= A^{s - u\ell + 2us\delta_{2,m}} B^{1 - 3s + 2s\delta_{2,m} + u + 3us\ell + 2us\delta_{2,m}} C^{-s - 3ur\ell + u(\ell+1) + 2u(\delta_{2,m}+1) + us(\delta_{2,m}+r+1)}\\
        &\qquad A^{us\ell^2(r+1) + 2us + u^2(\delta_{2,m}+r+1)}.\\
				&= A^{s - u\ell} B^{1 - 3s + 2s\delta_{2,m} + u} C^{-s + ur\ell} A^{us(r\ell^2 + \ell^2 + 1) + 2us(\delta_{2,m} + 1)},
				\end{align*}
       where the exponents of $A,B,C$ were reduced modulo $4us,4us,4u$ to produce the first equality, and we used
			$A^{2u}\in Z_1(J)$, $A^{2u}B^{2u}=1$, and $A^{2us}=C^{2u}$ to obtain the second equality.

     By (\ref{comfor}) and Theorem \ref{thm.commutators.J}, we have
    $[A^{s - u\ell},C^{-1 - 2s(\delta_{2,m}-1) - u}] = [A^{s - u\ell},C^{-1}]= A^{-2u\ell + 2us}$,
		$[B^{1 - 3s + 2s\delta_{2,m} + u},C^{-1 - 2s(\delta_{2,m}-1) - u}] = B^{2s\ell - 2u\ell(5-4\delta_{2,m})}$, 
		$[C^{-s + ur\ell},B^{-u}] = 1$, $[C^{-s + ur\ell},A^{-u}] = 1$,   $[A^{s - u\ell},B^{-u}]=[A^s,B^{-u}]=A^{u^2}$, and
		$[B^{\alpha - 3s + 2s\delta_{2,m} + u},A^{-u}]=[B^{1+s},A^{-u}]=A^{us\ell+u^2} C^{u}$.

    Using the above commutators to compute $(B\Delta_3)^{[B\Delta_3,A\Delta_3]}$ yields
    \begin{align*}
        (B\Delta_3)^{[B\Delta_3,A\Delta_3]}
        &= (A^{s - u\ell} B^{1 - 3s + 2s\delta_{2,m} + u} C^{-s + ur\ell})^{C^{-1 - 2s(\delta_{2,m}-1) - u} B^{-u} A^{-u}} A^{us(r\ell^2 + \ell^2 + 1) + 2us(\delta_{2,m} + 1)}\\
        &= (A^{s - u\ell} B^{\alpha - 3s + 2s\delta_{2,m} + u} C^{-s + ur\ell})^{B^{-u} A^{-u}} A^{8u\ell(1-\delta_{2,m}) + us(r\ell^2 + \ell^2 + 1) + 2us\delta_{2,m}}\\
        &= (A^{s - u\ell} B^{\alpha - 3s + 2s\delta_{2,m} + u} C^{-s + ur\ell})^{A^{-u}} 
				A^{8u\ell(1-\delta_{2,m}) + us(r\ell^2 + \ell^2 + 1) + 2us\delta_{2,m}+  u^2}\\
        &= A^{s - u\ell} B^{\alpha - 3s + 2s\delta_{2,m} + u} C^{-s  + u + ur\ell} A^{8u\ell(1-\delta_{2,m}) +   
				us\ell(r\ell - 1) + 2us\delta_{2,m}}.
    \end{align*}

    On the other hand, applying Theorem \ref{thm.power.J} to
    \[
    (A^{s - u\ell} B^{1 - 3s + 2s\delta_{2,m} + u} C^{-s + ur\ell})^\alpha
    = A^{\textup{exp} A + \xi} B^{\textup{exp} B} C^{\textup{exp} C}
    \]
    yields,
    $\textup{exp} A\equiv s - u\ell + 2u\ell - us\ell^2 + 2us + u^2\mod 4us$,
    $\textup{exp} B\equiv\alpha - 3s + 2s\delta_{2,m} + u - 6u\ell + 4u\ell\delta_{2,m} + 2us + u^2\mod 4us$,
    $\textup{exp}\, C\equiv -s - u + ur\ell + u(3 - \ell) + usr\mod 4u$, and
    $\xi\equiv u^2(r+1)\mod 4us$, so
    \begin{align*}
        &(A^{s - u\ell} B^{1 - 3s + 2s\delta_{2,m} + u} C^{-s + ur\ell})^\alpha\\
        &= A^{s - u\ell + 2u\ell - us\ell^2 + 2us + u^2} B^{\alpha - 3s + 2s\delta_{2,m} + u - 6u\ell + 4u\ell\delta_{2,m} + 2us + u^2} C^{-s - u + ur\ell + u(3 - \ell) + usr} A^{u^2(r+1)}\\
        &= A^{s - u\ell} B^{\alpha - 3s + 2s\delta_{2,m} + u} C^{-s - u + ur\ell} A^{4u\ell(2 - \delta_{2,m}) + us\ell + u^2}.
    \end{align*}
    This and $C^{2u}=A^{2us}$ yield
    \begin{align*}
        (B\Delta_3)^\alpha
        &= (A^{s - u\ell} B^{1 - 3s + 2s\delta_{2,m} + u} C^{-s + ur\ell})^\alpha A^{us(r\ell^2 + \ell^2 + 1) + 2us(\delta_{2,m} + 1)}\\
        &= A^{s - u\ell} B^{\alpha - 3s + 2s\delta_{2,m} + u}  C^{-s + u + ur\ell} A^{4u\ell(2 - \delta_{2,m}) + 
				us\ell(r\ell - 1) + 2us\delta_{2,m} + u^2}.
    \end{align*}

    Then $(B\Delta_3)^{[B\Delta_3,A\Delta_3]} = (B\Delta_3)^\alpha$ if and only if $A^{4u\ell\delta_{2,m} + u^2} = 1$ if and only if $s\mid\ell\delta_{2,m} + r^2$. If $m=2$, then $s=2$ is a factor of $\ell + 1 = \ell\delta_{2,m} + r^2$. If $m>2$, then $s\geq 4$ is
a factor of $s\frac{r}{2} = \ell\delta_{2,m} + r^2$. Thus $(B\Delta_3)^{[B\Delta_3,A\Delta_3]} = (B\Delta_3)^\alpha$ and $\Delta_3$ extends to an endomorphism of $J$.

    Let us see that $\Delta_3$ fixes $Z_1(J)$. By Theorem \ref{thm.power.J} applied to $(A^{1+s} B^s)^{2u} = A^{\textup{exp} A + \xi} 
		B^{\textup{exp} B} C^{\textup{exp} C}$, we get
    $\textup{exp} A\equiv 2u + 2us\mod 4us$,
    $\textup{exp} B\equiv 2us + u^2\mod 4us$,
    $\textup{exp}\, C\equiv us\mod 4u$, and
    $\xi\equiv 0\mod 4us$, so
    \[
    (A^{2u})\Delta_3 = (A\Delta_3)^{2u} = (A^{1+s} B^s)^{2u} = A^{2u + 2us} B^{2us + u^2} C^{us} = A^{2u}.
    \]

    Then $\Delta_3$ is an automorphism of $J$.
\end{proof}

Once again, we set $\overline{r} = r/2$ whenever $m>2$.

\begin{prop}\label{sigmaisauto}
    Suppose that $m>2$. Then the assignment
    \[
    A\mapsto A^{1+r}B^rA^{s(sk + r\ell - 2)},\; B\mapsto B^{1+r}A^r,
    \]
    where $k$ is even if $m>3$ and odd if $m=3$, extends to an automorphism, say $\Sigma$, of $J$.
\end{prop}

\begin{proof}
    Let $i = sk + r\ell - 2$. Applying Theorem \ref{thm.product.J} to
    \[
    A\Sigma = (A^{1+r}B^rC^0)(A^{si}B^0C^0) = A^{\textup{exp} A + \xi} B^{\textup{exp} B} C^{\textup{exp} C}
    \]
    yields
    $\textup{exp} A\equiv 1 + r + si - uri\ell\mod 4us$,
    $\textup{exp} B\equiv r + uri\ell + usi\overline{r}\mod 4us$,
    $\textup{exp}\, C\equiv - sri + ui\overline{r}\mod 4u$, and
    $\xi\equiv usi\overline{r}\mod 4us$, so
    \begin{align*}
        A\Sigma
        &= A^{1 + r + si - uri\ell} B^{r + uri\ell + usi\overline{r}} C^{- sri + ui\overline{r}} A^{usi\overline{r}}\\
        &= A^{1 + r + si} B^r C^{- sri} A^{usi(\overline{r}+1)}.
    \end{align*}

    Also, applying the same theorem to $B\Sigma = (A^0B^{1+r}C^0)(A^rB^0C^0) = A^{\textup{exp} A + \xi} B^{\textup{exp} B} C^{\textup{exp} C}$ yields
    $\textup{exp} A\equiv r - sr\ell + ur\overline{r}\ell\mod 4us$,
    $\textup{exp} B\equiv 1 + r - u\overline{r}\ell - ur\overline{r}\ell\mod 4us$,
    $\textup{exp}\, C\equiv - r - s\overline{r} + sr\ell(\overline{r}+1)\mod 4u$, and
    $\xi\equiv ur\ell^2(1-\overline{r}) + 4u\ell^2(\overline{r}-1)\varphi(r+1) - us\overline{r} + 2us(\overline{r}-1)\varphi(r+1) + usr\equiv ur\ell^2(1-\overline{r}) + 4u(\overline{r}-1)\varphi(r+1) - us\overline{r} + usr\overline{r}\mod 4us$, so
    \begin{align*}
        B\Sigma
        &= A^{r - sr\ell + ur\overline{r}\ell} B^{1 + r - u\overline{r}\ell - ur\overline{r}\ell} C^{- r - s\overline{r} + sr\ell(\overline{r}+1)}\\
        &\qquad
        A^{ur\ell^2(1-\overline{r}) + 4u\ell^2(\overline{r}-1)\varphi(r+1) - us\overline{r} + 2us(\overline{r}-1)\varphi(r+1) + usr}\\
        &= A^{r - sr\ell} B^{1 + r - u\overline{r}\ell} C^{- r - s\overline{r} + sr\ell(\overline{r}+1)} A^{ur\ell^2(1-\overline{r}) + 4u(\overline{r}-1)\varphi(r+1) + us\overline{r}(\ell - 1) + usr\overline{r}}.
    \end{align*}

    Then $[A\Sigma,B\Sigma] = [A^{1 + r + si} B^r C^{- sri}, A^{r - sr\ell} B^{1 + r - u\overline{r}\ell} C^{- r - s\overline{r} + sr\ell(\overline{r}+1)}]$.

    Applying Proposition \ref{prop.bigCommutator.J} to $[A\Sigma,B\Sigma]$ produces
    $\textup{exp} A\equiv - sr\ell + u(\overline{r} + 1)\mod 2u$,
    $\textup{exp} B\equiv sr\ell + u(\overline{r}+1)\mod 2u$, and
    $\textup{exp}\, C\equiv 1 + s + si + sri\mod 2u$, so
    \[
    [A\Sigma,B\Sigma]
    \equiv A^{\textup{exp} A} B^{\textup{exp} B} C^{\textup{exp} C}
    \equiv A^{- sr\ell + u(\overline{r} + 1)} B^{sr\ell + u(\overline{r}+1)} C^{1 + s + si + sri}\mod Z_1(J).
    \]

    Thus
    \[
    (A\Sigma)^{[A\Sigma,B\Sigma]}
    = (A^{1 + r + si} B^r C^{- sri})^{A^{- sr\ell + u(\overline{r} + 1)} B^{sr\ell + u(\overline{r}+1)} C^{1 + s + si + sri}} A^{usi(\overline{r}+1)}.
    \]

    By Theorem \ref{thm.commutators.J} applied to $[B^r,A^{- sr\ell + u(\overline{r} + 1)}] = (A^{\textup{exp} A + \xi} B^{\textup{exp} B} C^{\textup{exp} C})^{-1}$ we get
    $\textup{exp} A\equiv - us\overline{r} + usr(\overline{r} + 1)\mod 4us$,
    $\textup{exp} B\equiv us\overline{r} + usr\mod 4us$,
    $\textup{exp}\, C\equiv - u\overline{r}\ell + ur\mod 4u$, and
    $\xi\equiv usr\overline{r}\mod 4us$, so
    \[
    [B^r,A^{- sr\ell + u(\overline{r} + 1)}]
    = (A^{- us\overline{r} + usr(\overline{r} + 1)} B^{us\overline{r} + usr} C^{- u\overline{r}\ell + ur} A^{usr\overline{r}})^{-1}
    = C^{u\overline{r}\ell}.
    \]

    By Theorem \ref{thm.commutators.J}, $[C^{- sri},A^{- sr\ell + u(\overline{r} + 1)}] = 1$.

    Applying the same theorem to $[A^{1 + r + si},B^{sr\ell + u(\overline{r}+1)}] = A^{\textup{exp} A + \xi} B^{\textup{exp} B} C^{\textup{exp} C}$ gives
    $\textup{exp} A\equiv - us\overline{r} + usr\mod 4us$,
    $\textup{exp} B\equiv - ur\ell^2 - us(\overline{r}\ell + \overline{r} + \ell) + usr(\overline{r} + 1)\mod 4us$, and
    $\textup{exp}\, C\equiv sr\ell + u(\overline{r}\ell + \overline{r} + 1) + ur\mod 4u$,
    $\xi\equiv usr\overline{r}\mod 4us$, so
    \begin{align*}
        [A^{1 + r + si},B^{sr\ell + u(\overline{r}+1)}]
        &= A^{- us\overline{r} + usr} B^{- ur\ell^2 - us(\overline{r}\ell + \overline{r} + \ell) + usr(\overline{r} + 1)} C^{sr\ell + u(\overline{r}\ell + \overline{r} + 1) + ur} A^{usr\overline{r}}\\
        &= A^{ur\ell^2 + us(\overline{r} + \ell)} C^{sr\ell + u}.
    \end{align*}

    Also,
    $[C^{- sri + u\overline{r}\ell},B^{sr\ell + u(\overline{r}+1)}] = 1$,
    $[C^{sr\ell + u},B^r] = B^{usr}$,\\
    $[A^{1 + r + si},C^{1 + s + si + sri}] = A^{2s\ell + 2u\ell + 4ui\ell + u\ell + us\ell + 2us}$, and
    $[B^r,C^{1 + s + si + sri}] = B^{- u\ell - us\ell + usi + 2us}$.

    Gathering the above commutators produces
    \begin{align*}
        (A\Sigma)^{[A\Sigma,B\Sigma]}
        &= (A^{1 + r + si} B^r C^{- sri + u\overline{r}\ell})^{B^{sr\ell + u(\overline{r}+1)} C^{1 + s + si + sri}} A^{usi(\overline{r} + 1)}\\
        &= (A^{1 + r + si} C^{sr\ell + u} B^r C^{- sri + u\overline{r}\ell})^{C^{1 + s + si + sri}} A^{ur\ell^2 + us(\overline{r} + \ell) + usi(\overline{r} + 1)}\\
        &= (A^{1 + r + si} B^r C^{sr\ell + u - sri + u\overline{r}\ell})^{C^{1 + s + si + sri}} A^{ur\ell^2 + us(\overline{r} + \ell) + usi(\overline{r} + 1) + usr}\\
        &= A^{\alpha + r + si + u\ell} B^{r - u\ell} C^{sr\ell + u - sri + u\overline{r}\ell} A^{2u\ell + 4ui\ell + ur\ell^2 + us(\overline{r} + \ell) + usi\overline{r} + usr + 2us}.
    \end{align*}

    On the other hand, applying Theorem \ref{thm.power.J} to
    $(A^{1 + r + si} B^r C^{- sri})^\alpha = A^{\textup{exp} A + \xi} B^{\textup{exp} B} C^{\textup{exp} C}$ gives
    $\textup{exp} A\equiv \alpha + r + si + u\ell + 2ui\ell + u\ell\varphi(\alpha) + us\overline{r} + usr\overline{r}\mod 4us$,
    $\textup{exp} B\equiv r + u\ell + ur\ell^2\mod 4us$,
    $\textup{exp}\, C\equiv - sr\ell - sri - u\overline{r}\ell + sr\varphi(\alpha) + ur\overline{r}\equiv - sr\ell - sri - u\overline{r}\ell + ur(\overline{r} + 1)\mod 4u$, and
    $\xi\equiv ur\varphi(\alpha) + usr(\overline{r}+1)\equiv usr\overline{r}\mod 4us$, so
    \begin{align*}
        &(A^{1 + r + si} B^r C^{- sri})^\alpha\\
        &= A^{\alpha + r + si + u\ell + 2ui\ell + u\ell\varphi(\alpha) + us\overline{r} + usr\overline{r}} B^{r + u\ell + ur\ell^2} C^{- sr\ell - sri - u\overline{r}\ell + ur(\overline{r} + 1)} A^{usr\overline{r}}\\
        &= A^{\alpha + r + si + u\ell} B^{r + u\ell} C^{- sr\ell - sri - u\overline{r}\ell} A^{2ui\ell + u\ell\varphi(\alpha) - ur\ell^2 + us\overline{r} + usr(\overline{r}+1)},
    \end{align*}
    and from this,
    \begin{align*}
        (A\Sigma)^\alpha
        &= (A^{1 + r + si} B^r C^{- sri})^\alpha A^{usi(\overline{r}+1)}\\
        &= A^{\alpha + r + si + u\ell} B^{r + u\ell} C^{- sr\ell - sri - u\overline{r}\ell} A^{2ui\ell + u\ell\varphi(\alpha) - ur\ell^2 + us\overline{r} + usi(\overline{r}+1) + usr(\overline{r}+1)}.
    \end{align*}

    Then $(A\Sigma)^{[A\Sigma,B\Sigma]} = (A\Sigma)^\alpha$ if and only if $A^{4u\ell + 2ui\ell - u\ell\varphi(\alpha) + 2us + usi + usr(\overline{r}+1)} = 1$. Replacing $i = sk + r\ell - 2$, the last expression is equivalent to $A^{us - u\ell\varphi(\alpha) + usr\overline{r} + 2usk} = 1$, which is true whenever $1 - \frac{\varphi(\alpha)}{s}\ell + r\overline{r} + 2k\equiv 0\mod 4$. This congruence can be immediately verified.

    As for the second relation, we have
    \begin{align*}
        [B\Sigma,A\Sigma]
        &\equiv (A^{- sr\ell + u(\overline{r} + 1)} B^{sr\ell + u(\overline{r}+1)} C^{1 + s + si + sri})^{-1}\\
        &\equiv C^{- 1 - s - si - sri} B^{- sr\ell - u(\overline{r}+1)} A^{sr\ell - u(\overline{r} + 1)}\mod Z_1(J).
    \end{align*}

    Thus
    \begin{align*}
        (B\Sigma)^{[B\Sigma,A\Sigma]}
        &= (A^{r - sr\ell} B^{1 + r - u\overline{r}\ell} C^{- r - s\overline{r} + sr\ell(\overline{r}+1)})^{C^{- 1 - s - si - sri} B^{- sr\ell - u(\overline{r}+1)} A^{sr\ell - u(\overline{r} + 1)}}\\
        &\qquad A^{ur\ell^2(1-\overline{r}) + 4u(\overline{r}-1)\varphi(r+1) + us\overline{r}(\ell - 1) + usr\overline{r}}.
    \end{align*}

    From Theorem \ref{thm.commutators.J} we obtain
    $[A^{r - sr\ell},C^{- 1 - s - si - sri}] = A^{- u\ell + us(\ell - 1) + 2us + usi}$ and\\
    $[B^{1 + r - u\overline{r}\ell},C^{- 1 - s - si - sri}] = B^{2s\ell + 3u\ell + 2ui\ell + us\ell + 2us + usr}$.

    Applying Theorem \ref{thm.commutators.J} to $[A^{r - sr\ell - u\ell},B^{-sr\ell - u(\overline{r}+1)}] = A^{\textup{exp} A + \xi} B^{\textup{exp} B} C^{\textup{exp} C}$ gives,
    $\textup{exp} A\equiv - us\overline{r} + usr\mod 4us$,
    $\textup{exp} B\equiv us\overline{r} + usr(\overline{r}+1)\mod 4us$,
    $\textup{exp}\, C\equiv - u\overline{r}\ell + ur\mod 4u$, and
    $\xi\equiv usr\overline{r}\mod 4us$, so
    \[
    [A^{r - sr\ell - u\ell},B^{-sr\ell - u(\overline{r}+1)}]
    = A^{- us\overline{r} + usr} B^{us\overline{r} + usr(\overline{r}+1)} C^{- u\overline{r}\ell + ur} A^{usr\overline{r}}
    = C^{-u\overline{r}\ell}.
    \]
    By Theorem \ref{thm.commutators.J},
    $[C^{- r - s\overline{r} + sr\ell(\overline{r}-1)},B^{- sr\ell - u(\overline{r}+1)}] = B^{usr}$,
    $[C^{-u\overline{r}\ell},B^{\alpha + r + u\ell(1 - \overline{r})}] = B^{usr}$, and
    $[C^{- r - s\overline{r} - sr\ell(\overline{r}-1) - u\overline{r}\ell},A^{sr\ell - u(\overline{r} + 1)}] = A^{usr}$. Also, applying the same theorem to\\
    $[A^{sr\ell - u(\overline{r} + 1)},B^{\alpha + r + u\ell(1 - \overline{r})}] = A^{\textup{exp} A + \xi} B^{\textup{exp} B} C^{\textup{exp} C}$ gives,
    $\textup{exp} A\equiv ur\ell^2 + us(\overline{r} - \overline{r}\ell - \ell) + usr(\overline{r} + 1)\mod 4us$,
    $\textup{exp} B\equiv us\overline{r} + usr\mod 4us$,
    $\textup{exp}\, C\equiv sr\ell - u + u\overline{r}(\ell - 1) + ur\mod 4u$, and
    $\xi\equiv usr\overline{r}\mod 4us$, so
    \begin{align*}
        &[B^{\alpha + r + u\ell(1 - \overline{r})},A^{sr\ell - u(\overline{r} + 1)}]\\
        &= (A^{ur\ell^2 + us(\overline{r} - \overline{r}\ell - \ell) + usr(\overline{r} + 1)} B^{us\overline{r} + usr} C^{sr\ell - u + u\overline{r}(\ell - 1) + ur} A^{usr\overline{r}})^{-1}\\
        &= A^{- ur\ell^2 + us(\overline{r} + \ell) + usr} C^{- sr\ell + u}.
    \end{align*}

    Gathering all the above commutators produces
    \begin{align*}
        (B\Sigma)^{[B\Sigma,A\Sigma]}
        &= (A^{r - sr\ell - u\ell} B^{\alpha + r + u\ell(1 - \overline{r})} C^{- r - s\overline{r} + sr\ell(\overline{r}+1)})^{B^{- sr\ell - u(\overline{r}+1)} A^{sr\ell - u(\overline{r} + 1)}}\\
        &\qquad A^{ur\ell^2(1-\overline{r}) + 4u(\overline{r}-1)\varphi(r+1) - 2u\ell - us - 2ui\ell + us\overline{r}(\ell - 1) + usr(\overline{r} + 1) + usi}\\
        &= (A^{r - sr\ell - u\ell} C^{- u\overline{r}\ell} B^{\alpha + r + u\ell(1 - \overline{r})} C^{- r - s\overline{r} + sr\ell(\overline{r}+1)})^{A^{sr\ell - u(\overline{r} + 1)}}\\
        &\qquad A^{ur\ell^2(1-\overline{r}) + 4u(\overline{r}-1)\varphi(r+1) - 2u\ell - us - 2ui\ell + us\overline{r}(\ell - 1) + usr\overline{r} + usi}\\
        &= (A^{r - sr\ell - u\ell} B^{\alpha + r + u\ell(1 - \overline{r})} C^{- r - s\overline{r} + sr\ell(\overline{r}+1) - u\overline{r}\ell})^{A^{sr\ell - u(\overline{r} + 1)}}\\
        &\qquad A^{ur\ell^2(1-\overline{r}) + 4u(\overline{r}-1)\varphi(r+1) - 2u\ell - us - 2ui\ell + us\overline{r}(\ell - 1) + usr(\overline{r} + 1) + usi}\\
        &= A^{r - sr\ell - u\ell} B^{\alpha + r + u\ell(1 - \overline{r})} C^{- r - s\overline{r} + sr\overline{r}\ell + u(1 - \overline{r}\ell)}\\
        &\qquad A^{- ur\overline{r}\ell^2 + 4u(\overline{r}-1)\varphi(r+1) - 2u\ell - 2ui\ell + us\overline{r}\ell + us(\ell - 1) + usi + usr(\overline{r} + 1)}.
    \end{align*}

    On the other hand, applying Theorem \ref{thm.power.J} to
    \[
    (A^{r - sr\ell} B^{1 + r - u\overline{r}\ell} C^{- r - s\overline{r} + sr\ell(\overline{r} + 1)})^\alpha
    = A^{\textup{exp} A + \xi} B^{\textup{exp} B} C^{\textup{exp} C},
    \]
    we get
    $\textup{exp} A\equiv r - sr\ell + u\ell - ur\ell^2 - us + ur\varphi(\alpha) + usr\overline{r}\equiv r - sr\ell + u\ell - ur\ell^2 - us + usr(\overline{r}+1)\mod 4us$,
    $\textup{exp} B\equiv \alpha + r + u\ell(1 - \overline{r}) - us\overline{r} + usr(\overline{r} + 1) + 2us + 2u\varphi(\alpha)\equiv \alpha + r + u\ell(1 - \overline{r}) - us\overline{r} + usr(\overline{r} + 1)\mod 4us$,
    $\textup{exp}\, C\equiv - r - s\overline{r} + sr\overline{r}\ell - u\ell(\overline{r} + 1) + ur\overline{r}\mod 4u$, and
    $\xi\equiv usr(\overline{r} + 1)\mod 4us$, so
    \begin{align*}
        &(A^{r - sr\ell} B^{1 + r - u\overline{r}\ell} C^{- r - s\overline{r} + sr\ell(\overline{r} + 1)})^\alpha\\
        &= A^{r - sr\ell + u\ell - ur\ell^2 - us + usr(\overline{r}+1)} B^{\alpha + r + u\ell(1 - \overline{r}) - us\overline{r} + usr(\overline{r} + 1)} C^{- r - s\overline{r} + sr\overline{r}\ell - u\ell(\overline{r} + 1) + ur\overline{r}} A^{usr(\overline{r} + 1)}\\
        &= A^{r - sr\ell + u\ell} B^{\alpha + r + u\ell(1 - \overline{r})} C^{- r - s\overline{r} + sr\overline{r}\ell - u\ell(\overline{r} + 1)} A^{- ur\ell^2 + us(\overline{r} - 1) + usr},
    \end{align*}
    and from this,
    \begin{align*}
        &(B\Sigma)^\alpha\\
        &= (A^{r - sr\ell} B^{1 + r - u\overline{r}\ell} C^{- r - s\overline{r} + sr\ell(\overline{r} + 1)})^\alpha A^{ur\ell^2(1-\overline{r}) + 4u(\overline{r}-1)\varphi(r+1) + us\overline{r}(\ell - 1) + usr\overline{r}}\\
        &= A^{r - sr\ell + u\ell} B^{\alpha + r + u\ell(1 - \overline{r})} C^{- r - s\overline{r} + sr\overline{r}\ell - u\ell(\overline{r} + 1)}
        A^{- ur\overline{r}\ell^2 + 4u(\overline{r} - 1)\varphi(r + 1) + us(\overline{r} - 1) + us\overline{r}(\ell - 1) + usr(\overline{r} + 1)}.
    \end{align*}

    Then $(B\Sigma)^{[B\Sigma,A\Sigma]} = (B\Sigma)^\alpha$ if and only if $A^{4u\ell + 2ui\ell - us + 2us + usi} = 1$. Replacing $i = sk + r\ell - 2$, the last expression is equivalent to $A^{2usk + usr} = 1$, which happens to be true whenever $k + \overline{r}$ is even. The conditions on $k$ confirm that this is the case. Thus the equality $(B\Sigma)^{[B\Sigma,A\Sigma]} = (B\Sigma)^\alpha$ holds and the given assignment extends to an endomorphism of~$J$.

    Let us see that $\Sigma$ is an automorphism. We have
    \[
    (A^{2u})\Sigma
    = (A\Sigma)^{2u}
    = (A^{1 + r + si} B^r C^{- sri})^{2u}.
    \]

    Applying Theorem \ref{thm.power.J} to $(A^{1 + r + si} B^r C^{- sri})^{2u} = A^{\textup{exp} A + \xi} B^{\textup{exp} B} C^{\textup{exp} C}$ gives
    $\textup{exp} A\equiv 2u + us\mod 4us$,
    $\textup{exp} B\equiv us + usr\mod 4us$,
    $\textup{exp}\, C\equiv ur\mod 4u$,
    $\xi\equiv 0\mod 4us$, so
    \[
    (A^{2u})\Sigma
    = A^{2u + us} B^{us + usr} C^{ur}
    = A^{2u}.
    \]

    Thus $\Sigma$ fixes $Z_1(J)$ pointwise and it is an automorphism.
\end{proof}

\begin{theorem}\label{teoremafinal}
    The factors of the series \eqref{seriedejota} are as follows:
    $$
    \mathrm{Aut}_1(J)\cong (\Z/{2^m}\Z)^2,\;\mathrm{Aut}_2(J)/\mathrm{Aut}_1(J)\cong (\Z/{2^m}\Z)^2,
    $$
    $$
    \mathrm{Aut}_3(J)/\mathrm{Aut}_2(J)\cong  (\Z/{2^{m-1}}\Z)^2\times \Z/2\Z,\;
    \mathrm{Inn}(J)\mathrm{Aut}_3(J)/\mathrm{Aut}_3(J)\cong  (\Z/{2^{m-1}}\Z)^2,
    $$
    $$
    \mathrm{Aut}_4(J)/\mathrm{Inn}(J)\mathrm{Aut}_3(J)\cong \Z/2\Z,\;
    \mathrm{Aut}(J)/\mathrm{Aut}_4(J)\cong\begin{cases} (\Z/2\Z)^2\text{ if }m>2,\\ \Z/2\Z\text{ if }m=2.\end{cases}
    $$
    Thus $|\mathrm{Aut}(J)|=2^{8m}$ if $m>2$ and $|\mathrm{Aut}(J)|=2^{15}$ if $m=2$. Moreover, $\mathrm{Aut}(J)$
    is generated by $\Omega_{(1,A^{2^{2m-1}})},\Psi_{(1,C^{2^{m-1}})},
    \Delta_1,\Delta_2,A\delta,\Delta_3,\theta$, and, if $m>2$, also~$\Sigma$,
    where: $\Omega_{(1,A^{2^{2m-1}})}$ is defined in Proposition \ref{autj1};
    $\Psi_{(1,C^{2^{m-1}})}$ is given in Proposition \ref{autjo};
    $\Delta_1$ and $\Delta_2$ are defined in Proposition \ref{autj3}; $A\delta$ is conjugation by $A$;
    $\Delta_3$ is given in Proposition \ref{delta3isauto}; $\theta$ is the automorphism $A\leftrightarrow B$; and $\Sigma$
    is defined in Proposition \ref{sigmaisauto}.
\end{theorem}

\begin{proof}
    The descriptions of the first 4 factors follow from Propositions \ref{autj1} and \ref{autjo}, and Theorem~\ref{autj3}.
		
		As indicated in Sections \ref{autgmod2} and \ref{autgmodh}, we have presentations
    $$
    K = \groupPresentation{a,b}{a^{[a,b]}=a^\alpha,\; b^{[b,a]}=b^\alpha,\; a^{2^{2m-1}}=1,\; b^{2^{2m-1}}=1,\; [a,b]^{2^{m-1}}=1},
    $$
    $$
    H =\groupPresentation{A_0,B_0}{A_0^{[A_0,B_0]}=A_0^\alpha,\; B_0^{[B_0,A_0]}=B_0^\alpha,\; A_0^{2^{2m-1}}=1,\; B_0^{2^{2m-1}}=1}.
    $$
    Let $\pi_1:J\to H$ and $\pi_2:H\to K$ be the projection maps, given by $A\mapsto A_0,B\mapsto B_0$ and
    $A_0\mapsto a,B_0\mapsto b$, respectively.

    Consider $\mathrm{Aut}(J)/\mathrm{Aut}_4(J)$. By Proposition \ref{zi2}, we have an imbedding
    $$
    f:\mathrm{Aut}(J)/\mathrm{Aut}_4(J)\hookrightarrow
    \mathrm{Aut}(H)/\mathrm{Aut}_3(H).
    $$
    If $m=2$, then $\langle \overline{\nu}\rangle=\mathrm{Aut}(H)/\mathrm{Aut}_3(H)\cong \Z/2\Z$ by Theorem \ref{main3},
		where $\nu$ is the automorphism $A_0\leftrightarrow B_0$; As $\langle \overline{\theta}\rangle^f=\langle \overline{\nu}\rangle$,
		it follows that $\mathrm{Aut}(J)/\mathrm{Aut}_4(J)\cong \Z/2\Z$. If $m>2$, Theorem \ref{main3} gives 
		$\langle \overline{\nu}, \overline{\Sigma_0}\rangle=\mathrm{Aut}(H)/\mathrm{Aut}_3(H)\cong (\Z/2\Z)^2$,
		where $\Sigma_0$ is as defined in Corollary \ref{ulti}; since $\overline{\Sigma}^f=\overline{\Sigma_0}$,
		we deduce that  $\mathrm{Aut}(J)/\mathrm{Aut}_4(J)\cong (\Z/2\Z)^2$.
		
	 Next consider $\mathrm{Aut}_4(J)/\mathrm{Inn}(J)\mathrm{Aut}_3(J)$. By Proposition \ref{zi2}, we likewise have an imbedding
    \begin{equation}
    \label{em2}
    \eta_1:\mathrm{Aut}_4(J)/\mathrm{Inn}(J)\mathrm{Aut}_3(J)\hookrightarrow
    \mathrm{Aut}_3(H)/\mathrm{Inn}(H)\mathrm{Aut}_2(H).
    \end{equation}
    By (\ref{put4}), $\Delta_3\in\mathrm{Aut}_4(J)$, and we show below that the image of $\eta_1$ has order 2 and is generated by~$\overline{\Delta_3}$.
    The order of $\mathrm{Aut}(J)$ follows immediately from the structure of the above factors. Moreover,
    by Propositions \ref{autj1} and \ref{autjo}, and Theorem \ref{autj3}, we have
    \begin{equation}
    \label{em3}
    \mathrm{Inn}(J)\mathrm{Aut}_3(J)=\langle\Omega_{(1,A^{2^{2m-1}})},\Omega_{(A^{2^{2m-1}},1)}, \Psi_{(1,C^{2^{m-1}})}, \Psi_{(C^{2^{m-1}},1)}, \Delta_1,\Delta_2, A\delta, B\delta\rangle.
    \end{equation}
    It follows from \eqref{em3} and the above considerations that $\mathrm{Aut}(J)$ is generated by the right-hand side of \eqref{em3}
    together with $\Delta_3,\theta$ as well as $\Sigma$ if $m>2$. The use of $\theta$ allows us to reduce these generators to the stated list.

    We proceed to compute the required image of $\eta_1$.  By Theorem \ref{tamanio}, the image of $\eta_1$ is contained in $\overline{\langle \Gamma_0,\Gamma_0^\nu\rangle}$,
    where $\Gamma_0$ is as defined in Proposition \ref{casomdos2}, namely
    by $A_0\mapsto A_0^{1+s},B_0\mapsto B_0A_0^s$, and $\overline{\langle \Gamma_0,\Gamma_0^\nu\rangle}$ is the Klein 4-group.

    By Propositions \ref{outow} and \ref{zi2}, we have maps
    $$
    \eta_2:\mathrm{Aut}_3(H)/\mathrm{Inn}(H)\mathrm{Aut}_2(H)\to\mathrm{Aut}_2(K)/\mathrm{Inn}(K)\mathrm{Aut}_1(K),
    $$
    $$
    \eta_3:\mathrm{Aut}_2(K)/\mathrm{Inn}(K)\mathrm{Aut}_1(K)\to Z_2(K)^2/(Z_1(K)\times\langle c\rangle)^2,
    $$
    where $\eta_2$ is injective and $\eta_3$ is a bijection between copies of $(\Z/2\Z)^4$. In this notation, we have
    $$
    \overline{\Delta_3}\eta_1\eta_2\eta_3=\overline{(a^s b^s,a^s b^{s(2\delta_{2,m} + s - 3)})}.
    $$
    As $Z_1(K)=\langle a^{2s},b^{2s}\rangle$, it follows that $\overline{\Delta_3}\eta_1\eta_2\eta_3$ is nontrivial, whence
    $\overline{\Delta_3}\eta_1$ is a nontrivial element of the Klein 4-group $\overline{\langle \Gamma_0,\Gamma_0^\nu\rangle}$.
    Thus, either $\mathrm{Im}(\eta_1)=\langle\overline{\Delta_3}\eta_1\rangle$ has order 2, or  else
    $\mathrm{Im}(\eta_1)=\langle \overline{\Gamma_0},\overline{\Gamma_0^\nu}\rangle$ is the Klein 4-group. Suppose, if
    possible, that the latter case occurs. Then necessarily $\overline{\Gamma_0}\in \mathrm{Im}(\eta_1)$. Since $\mathrm{Inn}(J)$ maps onto
    $\mathrm{Inn}(H)$, we infer the existence of $\gamma\in\mathrm{Aut}_2(H)$ such that $\Gamma_0\gamma$ lifts to an automorphism, say $\beta$,
    of $J$. Here $\Gamma_0\gamma$ is given by $A_0\mapsto A_0^{1+s}w_1, B_0\mapsto B_0A_0^s w_2$, with $w_1,w_2$ in $Z_2(H)$, so
    $\beta$ is given by $A\mapsto A^{1+s}u_1, B\mapsto BA^s u_2$, with $u_1,u_2\in Z_3(J)$. As $A^s\in Z_4(J)$, we can write
		this in the form $A\mapsto A^{1+s}u_1, B\mapsto A^s B v_2$, with $u_1,v_2\in Z_3(J)$, against Proposition \ref{gammasnosonauto}.
\end{proof}

\section{Appendix}

We will prove a more general version of Theorem \ref{thm.commutators.J}. For the purpose
of this proof only, we expand our general notation and let $p\in\N$ be a prime factor of $\alpha-1$, so that $\alpha=1+p^m\ell$,
where $m\geq 1$ and $\ell$ is an arbitrary integer that is not a multiple of $p$.
We assume only that if $(p,m)=(3,1)$, then $(\alpha-1)/3\equiv 1\mod 3$. 
Thus we are in Case 1 or 2, as defined in \cite[Section 4]{MS}. According to \cite[Theorem 5.3]{MS},
the Sylow $p$-subgroup $J$ of the Macdonald group $G$ has presentation
$$
J=\langle A,B\,|\, A^{[A,B]}=A^\alpha, B^{[B,A]}=B^\alpha, A^{2^{3m-1}}=1=B^{2^{3m-1}} \rangle\text{ if }p=2,
$$
$$
J=\langle A,B\,|\, A^{[A,B]}=A^\alpha, B^{[B,A]}=B^\alpha, A^{p^{3m}}=1=B^{p^{3m}} \rangle\text{ if }p\neq 2.
$$
In either case, we set $C=[A,B]$. If $p$ is odd, then $A^{p^{2m}}=B^{-p^{2m}}\in Z(J)$ by \cite[Proposition 6.4]{MS},
while \cite[Theorem 8.1 and Proposition 9.1]{MS} ensure that $\langle A^{p^m}, B^{p^m}, C^{p^m}\rangle$ is abelian.

Recalling the definitions of $\phi$ and $\varphi$ given in Section \ref{section.formulas.J},
we readily see that given $n,k,t\in\Z$, the functions $\phi$ and $\varphi$ satisfy the following properties:
\begin{align}
    &\phi(n+2p^{3m}k)\equiv\phi(n)\mod p^{3m},\label{eq.comj.1}\\
    &\varphi(n+6p^{3m}k)\equiv\varphi(n)\mod p^{3m},\label{eq.comj.10}\\
    &\phi(-n) = \phi(n+1),\label{eq.comj.22}\\
    &\phi(n+1) + \phi(n) = n^2,\label{eq.comj.9}\\
    &\phi(nt) = \phi(n)\phi(t+1) + \phi(n+1)\phi(t),\label{eq.comj.12}\\
    &\phi(n+k) = \phi(n) + \phi(k) + nk,\label{eq.comj.23}\\
    &\varphi(n+k) = \varphi(n) + \varphi(k) + n\phi(k) + \phi(n)k.\label{eq.comj.24}
\end{align}

\begin{theorem}\label{nuevofer}
    For all $n,t\in\Z$ the following identities hold in $J$:
    \begin{align}
        &[C^n,A^t]
        = A^{-p^m\ell nt - p^{2m}\ell^2\phi(n)t},\label{eq.fer}\\
        &[C^n,B^t]
        = B^{p^m\ell nt - p^{2m}\ell^2\phi(n+1)t},\label{eq.fer2}\\
        &[A^n,B^t]
        =A^{-p^m\ell\phi(n)t} B^{p^m\ell n\phi(t)} C^{nt - p^m\ell\phi(n)\phi(t)} A^{\xi(n,t)},\label{eq.fer3}
				\end{align}
    where 
    \[
    \xi(n,t)= p^{2m}\ell^2\chi_p\{2\varphi(n+1)t + (2n-7)\phi(n)\phi(t) - 2n\phi(t) - (3n+1)n\varphi(t)-2\delta_{p,3}3^{m-1}
		\ell\phi(n)\phi(t)\},
    \]
		and $\delta_{p,3}=1$ if $p=3$, $\delta_{p,3}=0$ if $p\neq 3$, $\chi_p=1/2$ if $p=2$, and  $\chi_p=(1+p^m)/2$ if $p\neq 2$.
    \end{theorem}

\begin{proof}
    Suppose first that $n\in\N$ and $t\in\Z$. Then
		$$\alpha^n\equiv (1+p^m\ell)^n\equiv 1 + p^m\ell n + p^{2m}\ell^2\phi(n)\mod p^{3m},$$
		so
    \[
    [C^n,A^t]
    = C^{-n} A^{-t} C^n A^t
    = (A^{-t})^{C^n}A^t= A^{t(1 - \alpha^n)}
    = A^{-p^m\ell nt - p^{2m}\ell^2\phi(n)t}.
    \]

    If $n\in\Z$, let $k\in\N$ be such that $n+2p^{3m}k>0$. Then, by the above and \eqref{eq.comj.1},
    \[
    [C^n,A^t]
    = [C^{n+p^{3m}k},A^t]
    = A^{-p^m\ell (n+2p^{3m}k)t - p^{2m}\ell^2\phi(n+2p^{3m}k)t}=A^{-p^m\ell n t- p^{2m}\ell^2\phi(n)t}.
    \]

    This proves that \eqref{eq.fer} is true for all $n,t\in\Z$.

    To prove \eqref{eq.fer2} just apply the automorphism $A\leftrightarrow B$, $C\leftrightarrow C^{-1}$ of $J$
		to $$[C^{-n},A^t] = A^{p^m\ell nt - p^{2m}\ell^2\phi(-n)t}$$ and use the identity \eqref{eq.comj.22} to get 
		$[C^n,B^t] = B^{p^m\ell nt - p^{2m}\ell^2\phi(n+1)t}$ for every $n,t\in\Z$.
		
		To prove \eqref{eq.fer3} let us first see that for every $n\in\Z$,
    \begin{equation}
        \label{eq.comj.5}
        [A^n,B]
        = A^{-p^m\ell\phi(n) + p^{2m}\ell^2\varphi(n+1)} C^n.
    \end{equation}

    It is clear that \eqref{eq.comj.5} holds for $n=1$. Suppose it is true for some $n\geq 1$. Then by \eqref{comfor}
    \[
    [A^{n+1},B]
    = [A^n,B]^A\, [A,B]
    = (A^{-p^m\ell\phi(n) + p^{2m}\ell^2\varphi(n+1)} C^n)^A\, C
    = A^{-p^m\ell\phi(n) + p^{2m}\ell^2\varphi(n+1)} (C^n)^A\, C.
    \]

    As $[C^n,A] = A^{-p^m\ell n - p^{2m}\ell^2\phi(n)}$, then
    \[
    [A^{n+1},B]
    = A^{-p^m\ell\phi(n) + p^{2m}\ell^2\varphi(n+1)} C^n A^{-p^m\ell n - p^{2m}\ell^2\phi(n)} C.
    \]

    Since $[C,A^{p^{2m}}] = 1$ and $[C^n,A^{-p^m\ell n}] = A^{p^{2m}\ell^2n^2}$, using \eqref{eq.comj.9}, \eqref{eq.comj.23}, 
		and \eqref{eq.comj.24}, we get
    \[
    [A^{n+1},B]
    = A^{-p^m\ell(\phi(n)+n) + p^{2m}\ell^2(\varphi(n+1) - \phi(n) + n^2)} C^{n+1}
    = A^{-p^m\ell\phi(n+1) + p^{2m}\ell^2\varphi(n+2)} C^{n+1}.
    \]

    By induction, \eqref{eq.comj.5} holds for all $n\in\N$. If $n\in\Z$, let $k\in\N$ be such that $n+6p^{3m}k>0$. Then, by the above, \eqref{eq.comj.1} and \eqref{eq.comj.10}, $[A^n,B] = [A^{n+6p^{3m}k},B] = A^{-p^m\ell\phi(n) + p^{2m}\ell^2\varphi(n+1)} C^n$, so \eqref{eq.comj.5} holds for every $n\in\Z$.

    Next fix $n\in\Z$. Let us see that \eqref{eq.fer3} is true for every $t\in\N$ by induction on $t$. By \eqref{eq.comj.5}, we have that \eqref{eq.fer3} is true for $t=1$. Suppose it is true for some $t\geq 1$ and set $\eta_{(x,y)} = A^{-p^m\ell\phi(x)y} B^{p^m\ell x\phi(y)} C^{xy - p^m\ell\phi(x)\phi(y)}$. Then, by \eqref{comfor}, \eqref{eq.comj.5}, and the inductive hypothesis,
    \[
    [A^n,B^{t+1}]
    = [A^n,B] [A^n,B^t]^B
    = A^{-p^m\ell\phi(n) + p^{2m}\ell^2\varphi(n+1)} C^n (\eta_{(n,t)}A^{\xi(n,t)})^B.
    \]

    Since $A^{\xi(n,t)}\in Z(J)$, then
    \begin{equation}
        \label{eq.comj.13}
        [A^n,B^{t+1}]
        = A^{-p^m\ell\phi(n) + p^{2m}\ell^2\varphi(n+1)} C^n \eta_{(n,t)}^B A^{\xi(n,t)},
    \end{equation}
    where
    \[
    \eta_{(n,t)}^B
    = (A^{-p^m\ell\phi(n)t} B^{p^m\ell n\phi(t)} C^{nt - p^m\ell\phi(n)\phi(t)})^B
    = (A^{-p^m\ell\phi(n)t})^B B^{p^m\ell n\phi(t)} (C^{nt - p^m\ell\phi(n)\phi(t)})^B.
    \]

    Here $[C^{nt - p^m\ell\phi(n)\phi(t)},B] = B^{p^m\ell nt - p^{2m}\ell^2(\phi(n)\phi(t) + \phi(nt+1))}$, 
		$[C^{p^m},B^{p^m}] = 1=[C,B^{p^{2m}}]$, and
		$$
		[A^{-p^m\ell\phi(n)t},B] = A^{-p^{2m}\ell^2\phi(n)t(\chi_p+\delta_{p,3}3^{m-1}\ell)} 
		C^{-p^m\ell\phi(n)t}.
		$$
		Therefore,
		\begin{align*}
        \eta_{(n,t)}^B
        &= A^{-p^m\ell\phi(n)t-p^{2m}\ell^2\phi(n)t(\chi_p+\delta_{p,3}3^{m-1}\ell)} 
				C^{-p^m\ell\phi(n)t} B^{p^m\ell n\phi(t)} C^{nt - p^m\ell\phi(n)\phi(t)} \\
				&\quad\; B^{p^m\ell nt - p^{2m}\ell^2(\phi(n)\phi(t) + \phi(nt+1))}\\
        &= A^{-p^m\ell\phi(n)t-p^{2m}\ell^2\phi(n)t(\chi_p+\delta_{p,3}3^{m-1}\ell)} B^{p^m\ell n\phi(t) - p^{2m}\ell^2(\phi(n)\phi(t) + 
				\phi(nt+1))} C^{nt} B^{p^m\ell nt}\\
				&\quad\;
				C^{-p^m\ell\phi(n)(t + \phi(t))}.
    \end{align*}

    Using that $[C^{nt},B^{p^m\ell nt}] = B^{p^{2m}\ell^2n^2t^2}$, $[C,B^{p^{2m}}] = 1$, and by \eqref{eq.comj.23}, we have that
    \begin{align*}
    \eta_{(n,t)}^B
    &= A^{-p^m\ell\phi(n)t-p^{2m}\ell^2\phi(n)t(\chi_p+\delta_{p,3}3^{m-1}\ell)} 
		B^{p^m\ell n\phi(t+1) - p^{2m}\ell^2(\phi(n)\phi(t) + \phi(nt+1) - n^2t^2)}\\
		&\quad\;
		C^{nt - p^m\ell\phi(n)\phi(t+1)}.
    \end{align*}

    Now, as $[C^n,A^{-p^m\ell\phi(n)t}] = A^{p^{2m}\ell^2n\phi(n)t}$, $[C^n,B^{p^m\ell n\phi(t+1)}] = B^{p^{2m}\ell^2n^2\phi(t+1)}$, 
		$[C,B^{p^{2m}}] = 1$, and $A^{p^{2m}\chi_p}\in Z(J)$, then
    \begin{align}
        C^n\eta_{(n,t)}^B
        &= A^{-p^m\ell\phi(n)t + p^{2m}\ell^2\phi(n)t(n-\chi_p-\delta_{p,3}3^{m-1}\ell)}
        B^{p^m\ell n\phi(t+1) + p^{2m}\ell^2(n^2\phi(t+1) - \phi(n)\phi(t) - \phi(nt+1) + n^2t^2)}\nonumber\\
        &\quad\; C^{n(t+1) - p^m\ell\phi(n)\phi(t+1)}.\label{eq.comj.14}
    \end{align}

    Replacing \eqref{eq.comj.14} in \eqref{eq.comj.13} yields
    \begin{align*}
        [A^n,B^{t+1}]
        &= A^{-p^m\ell\phi(n)(t+1) + p^{2m}\ell^2(\varphi(n+1) + \phi(n)t(n-\chi_p-\delta_{p,3}3^{m-1}\ell))}\\
				&\quad\;
        B^{p^m\ell n\phi(t+1) + p^{2m}\ell^2(n^2\phi(t+1) - \phi(n)\phi(t) - \phi(nt+1) + n^2t^2)}
        C^{n(t+1) - p^m\ell\phi(n)\phi(t+1)} A^{\xi(n,t)}\\
        &= \eta_{(n,t+1)} A^{p^{2m}\ell^2(\varphi(n+1) + \phi(n)t(n-\chi_p-\delta_{p,3}3^{m-1}\ell))}
        B^{p^{2m}\ell^2(n^2\phi(t+1) - \phi(n)\phi(t) - \phi(nt+1) + n^2t^2)} A^{\xi(n,t)}.
    \end{align*}

    Thus by \eqref{put},
    \begin{equation}
        \label{eq.comj.6}
        [A^n,B^{t+1}]
        = \eta_{(n,t+1)} A^{p^{2m}\ell^2\chi_p\{2\varphi(n+1) + \phi(n)t(2n-1-2\delta_{p,3}3^{m-1}\ell) - 2n^2\phi(t+1) + 
				2\phi(n)\phi(t) + 2\phi(nt+1) - 2n^2t^2\}} A^{\xi(n,t)}.
    \end{equation}

    On the other hand, as a consequence of \eqref{eq.comj.9}, \eqref{eq.comj.12}, and \eqref{eq.comj.23}, we obtain
    \begin{align}
        2n^2\phi(t+1)
        &= 4\phi(n)\phi(t) + 4\phi(n)t + 2n\phi(t) + 2nt,\label{eq.comj.15}\\
        2\phi(nt+1)
        &= 4\phi(n)\phi(t) + 2\phi(n)t + 2n\phi(t) + 2nt,\label{eq.comj.16}\\
        2n^2t^2
        &= 8\phi(n)\phi(t) + 4\phi(n)t + 4n\phi(t) + 2nt.\label{eq.comj.17}
    \end{align}

    Replacing \eqref{eq.comj.15}, \eqref{eq.comj.16}, and \eqref{eq.comj.17} in \eqref{eq.comj.6} yields
    \[
    [A^n,B^{t+1}]
    = \eta_{(n,t+1)} A^{p^{2m}\ell^2\chi_p\{2\varphi(n+1) + (2n-7)\phi(n)t - 2nt - (3n + 1)n\phi(t)-2\delta_{p,3}3^{m-1}\ell\phi(n)t\}} 
		A^{\xi(n,t)},
    \]
    where, by \eqref{eq.comj.23} and \eqref{eq.comj.24},
    \[
    A^{p^{2m}\ell^2\chi_p\{2\varphi(n+1) + (2n-7)\phi(n)t - 2nt - (3n + 1)n\phi(t)-2\delta_{p,3}3^{m-1}\ell\phi(n)t\}} A^{\xi(n,t)}
    = A^{\xi(n,t+1)}.
    \]

    Thus $[A^n,B^{t+1}] = \eta_{(n,t+1)}A^{\xi(n,t+1)}$. This shows that \eqref{eq.fer3} holds for every $n\in\Z$ and $t\in\N$.

    Finally, if $t\in\Z$, let $k\in\N$ be such that $t+6p^{3m}k>0$. Then, by the above, 
		\eqref{eq.comj.1}, and \eqref{eq.comj.10}, we get $[A^n,B^t] = [A^n,B^{t+6p^{3m}k}] = \eta_{(n,t+6p^{3m}k)}A^{\xi(n,t+6p^{3m}k)} = 
		\eta_{(n,t)}A^{\xi(n,t)}$, and the result follows.
		\end{proof}

We now resume the general notation of the paper, assuming for the remainder of this section that $m\geq 1$. We will
use the relations below instead of \eqref{eq.comj.1} and \eqref{eq.comj.10}:
\begin{align}
    &\phi(n+8usk)\equiv\phi(n)\mod 4us,\quad n,k\in\Z,\label{eq.comj.1new}\\
    &\varphi(n+24usk)\equiv\varphi(n)\mod 4us,\quad n,k\in\Z.\label{eq.comj.10new}
\end{align}

\begin{proof}[Proof of Theorem \ref{thm.product.J}] Let $\Delta = (A^iB^jC^k)(A^aB^bC^c)$ and set $\eta_{(x,y)} = A^{-2s\ell\phi(x)y} B^{2s\ell x\phi(y)} C^{xy - 2s\ell\phi(x)\phi(y)}$ as in the proof of Theorem \ref{nuevofer}. Let us first work with $\Delta$ modulo $Z_1(J)$.

    From Theorem \ref{thm.commutators.J}, we obtain
    $[C^k,A^a] = A^{-2s\ell ka} A^{-4u\ell^2\phi(k)a}$,
    $[B^j,A^a] = \eta^{-1}_{(a,j)}A^{-\xi_1(a,j)}$, and
    $\eta^{-1}_{(a,j)} = C^{2s\ell\phi(j)\phi(a) - ja} B^{-2s\ell\phi(j)a} A^{2s\ell j\phi(a)}$. Then
    \[
    \Delta
    \equiv A^{i+a} B^j C^{2s\ell\phi(j)\phi(a) - ja} B^{-2s\ell\phi(j)a} A^{2s\ell j\phi(a)} C^k A^{-2s\ell ka} B^b C^c
    \mod Z_1(J).
    \]

    Since $[C^{2s},B^{2s}] = 1$, $[C^{-ja},B^{-2s\ell\phi(j)a}] = B^{4u\ell^2j\phi(j)a^2}$, $[C^k,A^{-2s\ell ka}] = A^{4u\ell^2k^2a}$, then
    \[
    \Delta
    \equiv A^{i+a} B^{j - 2s\ell\phi(j)a} C^{2s\ell\phi(j)\phi(a) - ja} A^{2s\ell (j\phi(a) - ka)} C^k B^b C^c
    \mod Z_1(J).
    \]

    As $[C^{2s},A^{2s}] = 1$, $[C^{-ja},A^{2s\ell (j\phi(a) - ka)}] = A^{4u\ell^2ja(j\phi(a) - ka)}$, $[C^k,B^b] = B^{2s\ell kb} B^{-4u\ell^2\phi(k+1)b}$, then
    \[
    \Delta
    \equiv A^{i+a} B^{j - 2s\ell\phi(j)a} A^{2s\ell (j\phi(a) - ka)} C^{2s\ell\phi(j)\phi(a) - ja} B^b C^k B^{2s\ell kb} C^c
    \mod Z_1(J).
    \]

    Seeing that
    $[B^{j - 2s\ell\phi(j)a},A^{2s\ell (j\phi(a) - ka)}] = \eta^{-1}_{(2s\ell (j\phi(a) - ka),\; j - 2s\ell\phi(j)a)} A^{-\xi_1(2s\ell (j\phi(a) - ka),\; j - 2s\ell\phi(j)a)}$, where
    \begin{align*}
        \eta^{-1}_{(2s\ell (j\phi(a) - ka),\; j - 2s\ell\phi(j)a)} =\;
        &C^{2us\phi(j)a(ka - j\phi(a))} C^{2u\ell^2\phi(j)(ka - j\phi(a))} C^{2s\ell j(ka - j\phi(a))}\\
        &B^{4u\ell^2\phi(j)(ka - j\phi(a))} A^{2u\ell^2j(ka - j\phi(a))},
    \end{align*}
    $[C^{2s\ell\phi(j)\phi(a) - ja},B^b] = B^{-2s\ell jab} B^{4u\ell^2(\phi(j)\phi(a)b - \phi(1-ja)b)}$, and $[C^k,B^{2s\ell kb}] = B^{4u\ell^2k^2b}$, then
    \[
    \Delta
    \equiv A^{i + a + 2s\ell (j\phi(a) - ka)} B^{j - 2s\ell\phi(j)a} C^{2s\ell j(ka - j\phi(a))} B^b C^{2s\ell\phi(j)\phi(a) - ja}  B^{2s\ell(kb - jab)} C^{k+c}
    \mod Z_1(J).
    \]

    Since $[C^{2s\ell j(ka - j\phi(a))},B^b] = B^{4u\ell^2jb(ka - j\phi(a))}$, $[C^{2s},B^{2s}] = 1$, $[C^{-ja},B^{2s\ell(kb - jab)}] = B^{4u\ell^2jab(ja - k)}$, then
    \[
    \Delta
    \equiv A^{i + a + 2s\ell (j\phi(a) - ka)} B^{j + b + 2s\ell(kb - jab - \phi(j)a)} C^{k + c - ja + 2s\ell (\phi(j)\phi(a) + j(ka - j\phi(a)))}
    \mod Z_1(J).
    \]

    Using \eqref{eq.comj.9} and \eqref{eq.comj.23} in the exponent of $C$ and collecting the above central elements, we obtain
    \[
    \Delta
    = A^{i + a + 2s\ell (j\phi(a) - ka)} B^{j + b + 2s\ell(kb - jab - \phi(j)a)} C^{k + c - ja + 2s\ell (jka - \phi(j+1)\phi(a))} \Delta^\prime,
    \]
    where
    \begin{align*}
        \Delta^\prime =\;
        &A^{4u\ell^2(-\phi(k)a + k^2a + j^2a\phi(a) - jka^2)
            + 2u\ell^2j(ka - j\phi(a))}\\
        &B^{4u\ell^2(j\phi(j)a^2 - \phi(k+1)b - j\phi(j)\phi(a)
            + \phi(j)ka + \phi(j)\phi(a)b - \phi(1-ja)b
            + k^2b - j^2\phi(a)b + j^2a^2b)}\\
        &C^{2us\phi(j)a(ka - j\phi(a)) + 2u\ell^2\phi(j)(ka
            - j\phi(a))} A^{-\xi_1(a,j)} A^{-\xi_1(2s\ell (j\phi(a)
            - ka),\; j - 2s\ell\phi(j)a)}.
    \end{align*}
		
		It remains to show that $\Delta^\prime = A^{\xi_2(j,k,a,b)}$. For this purpose, note that
\begin{align*}
    &\xi_1(2s\ell(j\phi(a) - ka), j - 2s\ell\phi(j)a)\\
    &= 2u\ell^2\{2\varphi(2s\ell(j\phi(a) - ka)+1)(j - 2s\ell\phi(j)a)\\
    &+ (4s\ell(j\phi(a) - ka) - 7)\phi(2s\ell(j\phi(a) - ka))\phi(j - 2s\ell\phi(j)a)\\
    &- 4s\ell(j\phi(a) - ka)\phi(j - 2s\ell\phi(j)a)\\
    &- (6s\ell(j\phi(a) - ka)+1)(2s\ell(j\phi(a) - ka))\varphi(j - 2s\ell\phi(j)a)\}.
\end{align*}
Since $s\mid\varphi(2s\ell(j\phi(a) - ka) + 1)$ then
\begin{align*}
    &\xi_1(2s\ell(j\phi(a) - ka), j - 2s\ell\phi(j)a)\\
    &\equiv 2u\ell^2(- 7)\phi(2s\ell(j\phi(a) - ka))\phi(j - 2s\ell\phi(j)a)\\
    &\equiv 2us(j\phi(a) - ka)\phi(j - 2s\ell\phi(j)a)\mod 4us.
\end{align*}
By \eqref{eq.comj.22} and \eqref{eq.comj.23}, $\phi(j - 2s\ell\phi(j)a) = \phi(j) + \phi(2s\ell\phi(j)a+1) - 2s\ell j\phi(j)a$, so
\begin{align*}
    &\xi_1(2s\ell(j\phi(a) - ka), j - 2s\ell\phi(j)a)\\
    &\equiv 2us(j\phi(a) - ka)(\phi(j) + \phi(2s\ell\phi(j)a+1))\\
    &\equiv 2us(j\phi(a) - ka)\phi(j) + 2us(j\phi(a) - ka)\phi(2s\ell\phi(j)a+1))\\
    &\equiv 2us(j\phi(a) - ka)\phi(j) + 2u^2(j\phi(a) - ka)\phi(j)a\mod 4us.
\end{align*}
Since $C^{2u} = A^{2us}$, then $C^{2us\phi(j)a(j\phi(a) - ka) + 2u\phi(j)(j\phi(a) - ka)} A^{-\xi_1(2s\ell(j\phi(a) - ka), j - 2s\ell\phi(j)a)} = 1$.

Now, if $n\in\Z$, then the use of \eqref{eq.comj.9} and \eqref{eq.comj.23} yields $n^2 = 2\phi(n) + n$.

Also by \eqref{eq.comj.22}, \eqref{eq.comj.12}, and \eqref{eq.comj.23}, $\phi(1-ja) = \phi(-1+ja+1) = \phi(ja) = \phi(j)\phi(a+1) + \phi(j+1)\phi(a) = 2\phi(j)\phi(a) + \phi(j)a + j\phi(a)$. By (\ref{put}) $\Delta^\prime$ is a power of $A$ whose exponent is 
\begin{align*}
    &= 2u\ell^2\{-2\varphi(a+1)j - (2a-7)\phi(a)\phi(j) + 2a\phi(j) + (3a+1)a\varphi(j) - 2\phi(k)a + 2k^2a + 2j^2a\phi(a)\\
    &\qquad - 2jka^2 + jka - j^2\phi(a) - 2j\phi(j)a^2 + 2\phi(k+1)b + 2j\phi(j)\phi(a) - 2\phi(j)ka - 2\phi(j)\phi(a)b\\
    &\qquad + 2\phi(1-ja)b - 2k^2b + 2j^2\phi(a)b - 2j^2a^2b\}\\
    &= 2u\ell^2\{-2\varphi(a+1)j - 2a\phi(a)\phi(j) + 7\phi(a)\phi(j) + 2a\phi(j) + (3a+1)a\varphi(j) - 2\phi(k)a + 4\phi(k)a\\
    &\qquad + 2ka + 4\phi(j)a\phi(a) + 2ja\phi(a) - 4jk\phi(a) - 2jka + jka - 2\phi(j)\phi(a) - j\phi(a) - 4j\phi(j)\phi(a)\\
    &\qquad - 2j\phi(j)a + 2\phi(k)b + 2kb + 2j\phi(j)\phi(a) - 2\phi(j)ka - 2\phi(j)\phi(a)b + 4\phi(j)\phi(a)b + 2\phi(j)ab\\
    &\qquad + 2j\phi(a)b - 4\phi(k)b - 2kb + 4\phi(j)\phi(a)b + 2j\phi(a)b - 8\phi(j)\phi(a)b - 4\phi(j)ab - 4j\phi(a)b - 2jab\}\\
    &= 2u\ell^2\{\phi(a)\phi(j)(-2a + 7 + 4a - 2 - 4j + 2j - 2b + 4b + 4b - 8b) + \phi(j)a(2 - 2j - 2k + 2b - 4b)\\
    &\qquad + j\phi(a)(2a - 4k + 2b + 2b - 4b - 1) + \phi(k)(-2a + 4a + 2b - 4b) + ka(-2j + j + 2)\\
    &\qquad - 2j(ab + \varphi(a+1)) + (3a+1)a\varphi(j)\}\\
    &= 2u\ell^2\{\phi(a)\phi(j)(-2j + 2a - 2b + 5) - 2\phi(j)a(j + k + b - 1) - j\phi(a)(-2a + 4k + 1)\\
    &\qquad + 2\phi(k)(a - b) - ka(j - 2) - 2j(ab + \varphi(a+1)) + (3a+1)a\varphi(j)\}.
\end{align*}
Thus $\Delta^\prime = A^{\xi_2(j,k,a,b)}$ as stated.
\end{proof}

\begin{proof}[Proof of Corollary \ref{cor.inverse.J}] 
A direct application of Theorem \ref{thm.product.J} and use of \eqref{eq.comj.22} and \eqref{eq.comj.23} give
    \[
    B^{-b} A^{-a} = (A^0B^{-b}C^0)(A^{-a}B^0C^0)
    = A^{-a - 2s\ell\phi(a+1)b} B^{-b + 2s\ell a\phi(b+1)} C^{- ab - 2s\ell\phi(a+1)\phi(b)} A^{\xi_2(-b,0,-a,0)},
    \]
    so
    \begin{align*}
        (A^aB^bC^c)^{-1}
        &= C^{-c}B^{-b}A^{-a}\\
        &= (A^0B^0C^{-c})(A^{-a - 2s\ell\phi(a+1)b}
        B^{-b + 2s\ell a\phi(b+1)}
        C^{- ab - 2s\ell\phi(a+1)\phi(b)}) A^{\xi_2(-b,0,-a,0)}\\
        &= A^{-a - 2s\ell(\phi(a+1)b + ac) - 4u\ell^2\phi(a+1)bc}
        B^{-b + 2s\ell(a\phi(b+1) + bc) - 4u\ell^2a\phi(b+1)c}
        C^{-c - ab - 2s\ell\phi(a+1)\phi(b)}\\
        &\quad\; A^{\xi_2(-b,0,-a,0)
            + \xi^2(0,-c,-a-2s\ell\phi(a+1)b,-b+2s\ell a\phi(b+1))}.\qedhere
    \end{align*}
		\end{proof}

\begin{proof}[Proof of Proposition \ref{prop.bigCommutator.J}]
    By Corollary \ref{cor.inverse.J},
    \begin{align*}
        (A^iB^jC^k)^{-1}
        &\equiv A^{-i - 2s\ell(\phi(i+1)j + ik)} B^{-j + 2s\ell(i\phi(j+1) + jk)}
        C^{-k - ij - 2s\ell\phi(i+1)\phi(j)}\mod Z_1(J),\\
        (A^aB^bC^c)^{-1}
        &\equiv A^{-a - 2s\ell(\phi(a+1)b + ac)} B^{-b + 2s\ell(a\phi(b+1) + bc)}
        C^{-c - ab - 2s\ell\phi(a+1)\phi(b)}\mod Z_1(J).
    \end{align*}

    It is easy to see that $\phi(n+2sk)\equiv\phi(n)\mod s$ for all $n,k\in\Z$, so Theorem \ref{thm.product.J} and the use of \eqref{eq.comj.22}, \eqref{eq.comj.9}, and \eqref{eq.comj.23}, if necessary, implies,
    \[
    (A^iB^jC^k)^{-1}(A^aB^bC^c)^{-1}
    \equiv A^{\delta_1} B^{\delta_2} C^{\delta_3}\mod Z_1(J),
    \]
    where
    \begin{align*}
        \delta_1
        &\equiv - i - a - 2s\ell\{\phi(i+1)j + \phi(a+1)b + j\phi(a+1) + ik + ac + (k + ij)a\}\mod 2u,\\
        \delta_2
        &\equiv - j - b + 2s\ell\{i\phi(j+1) + a\phi(b+1) + \phi(j+1)a + jk + bc + (k + ij)b + jab\}\mod 2u,\\
        \delta_3
        &\equiv - k - c - ij - ab - ja - 2s\ell\{j\phi(a+1)b + jac - i\phi(j+1)a + \phi(i+1)\phi(j)\\
        &\qquad + \phi(a+1)\phi(b) + ij^2a + \phi(j)\phi(a+1)\}\\
        &\equiv - k - c - ij - ab -ja - 2s\ell\{j\phi(a+1)b + i\phi(j)a + \phi(i+1)\phi(j) + \phi(a+1)\phi(b)\\
        &\qquad + \phi(j)\phi(a+1) + jac\}\mod 2u,
    \end{align*}
    and
    \begin{align*}
        &(A^iB^jC^k)(A^aB^bC^c)\\
        &\qquad\equiv A^{i + a + 2s\ell (j\phi(a) - ka)}
        B^{j + b + 2s\ell (kb - jab - \phi(j)a)}
        C^{k + c - ja + 2s\ell (jka - \phi(j+1)\phi(a))}\mod Z_1(J).
    \end{align*}

    Then, applying Theorem \ref{thm.product.J}, and using \eqref{eq.comj.22}, \eqref{eq.comj.9}, and \eqref{eq.comj.23}, shows that
    \begin{align*}
        [A^iB^jC^k,A^aB^bC^c]
        &= (A^iB^jC^k)^{-1}(A^aB^bC^c)^{-1}(A^iB^jC^k)(A^aB^bC^c)\\
        &\equiv A^{\textup{exp} A}B^{\textup{exp} B}C^{\textup{exp} C}\mod Z_1(J),
    \end{align*}
    where
    \begin{align*}
        \textup{exp} A
        &\equiv 2s\ell\{j\phi(a) - 2ka - \phi(i+1)j - \phi(a+1)b - j\phi(a+1) - ik - ac - ija\\
        &\qquad - (j+b)\phi(i+a) + (k+c+ij+ab+ja)(i+a)\}\\
        &\equiv 2s\ell\{j\phi(a) - \phi(i)b + ic - ka\}\mod 2u,\\
        \textup{exp} B
        &\equiv 2s\ell\{kb - jab - \phi(j)a + i\phi(j+1) + a\phi(b+1) + \phi(j+1)a + jk + bc + (k + ij)b\\
        &\qquad + jab - (k+c+ij+ab+ja)(j+b) + (j+b)^2(i+a) - \phi(j+b+1)(i+a)\}\\
        &\equiv 2s\ell\{i\phi(b) - \phi(j)a + kb + j(ib - ab - c)\}\mod 2u,\\
        \textup{exp}\, C
        &\equiv \delta_3 + k + c - ja + 2s\ell (jka - \phi(j+1)\phi(a)) - \delta_2(i + a + 2s\ell (j\phi(a) - ka))\\
        &\qquad + 2s\ell\{(j+b)(i+a)(k+c+ij+ab+ja) - \phi(j+b)\phi(i+a)\}\\
        &\equiv ib - ja + 2s\ell\{\phi(a)(\phi(j) + jb) - \phi(i)(\phi(b) + jb) + ijc - kab\}\mod 2u.\qedhere
    \end{align*}
\end{proof}

\begin{lemma}\label{lemma.sums.J}
    Let $a\in\Z$ and $t\in\N$. Then\vspace{5pt}\\
    $
    \sum_{k=0}^{t} k = \frac{t(t+1)}{2} = \phi(t+1),\\
    \sum_{k=0}^{t} k^2 = \frac{t(t+1)(2t+1)}{6} = 2\varphi(t+1) + \phi(t+1),\\
    \sum_{k=0}^{t} k^3 = (\frac{t(t+1)}{2})^2 = \phi(t+1)^2,\\
    \sum_{k=0}^{t}\phi(ak) = \sum_{k=0}^{t}\frac{ak(ak-1)}{2} = a^2\varphi(t+1) + \phi(a)\phi(t+1),\\
    \sum_{k=0}^{t}\phi(ak)k = \sum_{k=0}^{t}\frac{ak^2(ak-1)}{2} = \frac{a\phi(t+1)(a\phi(t+1)-1)}{2} - a\varphi(t+1),\\
    \sum_{k=0}^{t}\varphi(ak+1) = \sum_{k=0}^{t}\frac{(ak+1)(ak)(ak-1)}{6} = \frac{a\phi(t+1)(a^2\phi(t+1)-1)}{6},\\
    \sum_{k=0}^{t}\varphi(k) = \sum_{k=0}^{t}\frac{k(k-1)(k-2)}{6} = \frac{\phi(t+1)(\phi(t+1)-1)}{6} - \varphi(t+1),\\
    \sum_{k=0}^{t}\varphi(k)k = \sum_{k=0}^{t}\frac{k^2(k-1)(k-2)}{6} = \frac{1}{6}\left(\frac{t(t+1)(6t^3+9t^2+t-1)}{30} - 3\phi(t+1)^2 + 4\varphi(t+1) + 2\phi(t+1)\right).
    $
\end{lemma}

\begin{proof}
    In each case the use of the distributive law in the sum's general term, properties of finite sums, and the definitions of $\phi$ and $\varphi$ give the desired results.
\end{proof}

\begin{proof}[Proof of Theorem \ref{thm.power.J}] Fix $a,b,c\in\Z$. Suppose that $n\geq 2$, let $k,t\in\N$, and set $\Delta = (A^aB^bC^c)^n$. Let us first work with $\Delta$ modulo $Z_1(J)$. By definition,
    \begin{align*}
        \Delta
        &= (A^aB^bC^c)(A^aB^bC^c)\cdots (A^aB^bC^c)(A^aB^bC^c)(A^aB^bC^c)\\
        &= A^{na} (B^bC^c)^{A^{(n-1)a}} (B^bC^c)^{A^{(n-2)a}}\cdots (B^bC^c)^{A^{2a}} (B^bC^c)^{A^{a}} (B^bC^c).
    \end{align*}
    As $[B^b,A^{ka}]
    = [A^{ka},B^b]^{-1}
    = (A^{-2s\ell b\phi(ak)} B^{2s\ell a\phi(b)k} C^{abk - 2s\ell\phi(b)\phi(ak)} A^{\xi_1(ak,b)})^{-1}
    = C^{2s\ell\phi(b)\phi(ak) - abk}\\
    B^{-2s\ell a\phi(b)k} A^{2s\ell b\phi(ak)} A^{-\xi_1(ak,b)}$ and $[C^c,A^{ka}] = A^{-2s\ell ack - 4u\ell^2a\phi(c)k}$, we infer
    \begin{align*}
        \Delta
        &\equiv A^{na} (B^b C^{2s\ell\phi(b)\phi((n-1)a) - ab(n-1)} B^{-2s\ell a\phi(b)(n-1)} A^{2s\ell b\phi((n-1)a)} C^c A^{-2s\ell ac(n-1)})\cdots\\
        &\qquad (B^b C^{2s\ell\phi(b)\phi((1)a) - ab(1)} B^{-2s\ell a\phi(b)(1)} A^{2s\ell b\phi((1)a)} C^c A^{-2s\ell ac(1)}) (B^bC^c).
    \end{align*}
    Since $[C^{2s\ell\phi(b)\phi(ka) - abk},B^{-2s\ell a\phi(b)k}] = B^{4u\ell^2a^2\phi(b)bk^2}$ and $[A^{2s\ell b\phi(ka)},C^c] = A^{4u\ell^2bc\phi(ka)}$, then
    \begin{align*}
        \Delta
        &\equiv A^{na} (B^{b - 2s\ell a\phi(b)(n-1)} C^{2s\ell\phi(b)\phi((n-1)a) - ab(n-1) + c} A^{2s\ell(b\phi((n-1)a) - ac(n-1))})\cdots\\
        &\qquad (B^{b - 2s\ell a\phi(b)(1)} C^{2s\ell\phi(b)\phi((1)a) - ab(1) + c} A^{2s\ell (b\phi((1)a) - ac(1))}) (B^bC^c).
    \end{align*}
    Set $g_1(t) = \sum_{k=0}^{t} (b\phi(ka) - ack)$. By Lemma \ref{lemma.sums.J}, $g_1(t) = a^2b\varphi(t+1) + (\phi(a)b - ac)\phi(t+1)$. Then
    \begin{align*}
        \Delta
        &\equiv A^{na} A^{2s\ell g_1(n-1)} (B^{b - 2s\ell a\phi(b)(n-1)} C^{2s\ell\phi(b)\phi((n-1)a) - ab(n-1) + c})^{A^{2s\ell g_1(n-1)}}\cdots\\
        &\qquad (B^{b - 2s\ell a\phi(b)(1)} C^{2s\ell\phi(b)\phi((1)a) - ab(1) + c})^{A^{2s\ell g_1(1)}} (B^bC^c).
    \end{align*}
    As
    \begin{align*}
        &[B^{b - 2s\ell a\phi(b)k},A^{2s\ell g_1(k)}]\\
        &\qquad = [A^{2s\ell g_1(k)},B^{b - 2s\ell a\phi(b)k}]^{-1}\\
        &\qquad = (A^{2u\ell^2bg_1(k)} B^{4u\ell^2\phi(b)g_1(k)} C^{2s\ell bg_1(k) + 2u\ell^2\phi(b)g_1(k) + 2us\ell^3a\phi(b)g_1(k)k} A^{\xi_1(2s\ell g_1(k),b - 2s\ell a\phi(b)k)})^{-1}\\
        &\qquad = C^{-2s\ell bg_1(k)} A^{-2u\ell^2bg_1(k)} B^{-4u\ell^2\phi(b)g_1(k)} C^{- 2u\ell^2\phi(b)g_1(k) - 2us\ell^3a\phi(b)g_1(k)k} A^{-\xi_1(2s\ell g_1(k),b - 2s\ell a\phi(b)k)}
    \end{align*}
    and $[C^{2s\ell\phi(b)\phi(ka) - abk + c},A^{2s\ell g_1(k)}] = A^{-4u\ell^2g_1(k)(-abk + c)}$, we deduce
    \begin{align*}
        \Delta
        &\equiv A^{na + 2s\ell g_1(n-1)} (B^{b - 2s\ell a\phi(b)(n-1)} C^{-2s\ell bg_1(n-1) + 2s\ell\phi(b)\phi((n-1)a) - ab(n-1) + c})\cdots\\
        &\qquad (B^{b - 2s\ell a\phi(b)(2)} C^{-2s\ell bg_1(2) + 2s\ell\phi(b)\phi((2)a) - ab(2) + c}) (B^{b - 2s\ell a\phi(b)(1)} C^{-2s\ell bg_1(1) + 2s\ell\phi(b)\phi((1)a) - ab(1) + c})\\
        &\qquad (B^bC^c).
    \end{align*}
    Set $g_2(t) = \sum_{k=0}^{t} (b - 2s\ell a\phi(b)k)$ and $g_3(t) = 2s\ell\{\phi(b)\phi(ta) - bg_1(t)\} - abt + c$. By Lemma \ref{lemma.sums.J}, $g_2(t) = b(t+1) - 2s\ell a\phi(b)\phi(t+1)$. Therefore
    \[
    \Delta
    \equiv A^{na + 2s\ell g_1(n-1)} B^{g_2(n-1)} (C^{g_3(n-1)})^{B^{g_2(n-2)}}\cdots (C^{g_3(2)})^{B^{g_2(1)}} (C^{g_3(1)})^{B^{g_2(0)}} (C^{g_3(0)}).
    \]
    As $[C^{g_3(k)},B^{g_2(k-1)}] = B^{2s\ell(bck - ab^2k^2)} B^{4u\ell^2(\phi(b)b\phi(ka)k - b^2g_1(k)k + a^2\phi(b)b\phi(k)k - a\phi(b)c\phi(k) - b\phi(-abk+c+1)k)}$, we deduce
    \begin{align*}
        \Delta
        &\equiv A^{na + 2s\ell g_1(n-1)} B^{g_2(n-1)} (C^{g_3(n-1)} B^{2s\ell(bc(n-1) - ab^2(n-1)^2)})\cdots (C^{g_3(2)} B^{2s\ell(bc(2) - ab^22^2)})\\
        &\quad (C^{g_3(1)} B^{2s\ell(bc(1) - ab^21^2)}) (C^{g_3(0)}).
    \end{align*}
    Set $g_4(t) = \sum_{k=0}^{t} (bck - ab^2k^2)$. By Lemma \ref{lemma.sums.J}, $g_4(t) = b(c - ab)\phi(t+1) - 2ab^2\varphi(t+1)$. As $[C^{g_3(k)},B^{2s\ell g_4(k)}] = B^{4u\ell^2g_4(k)(-abk+c)}$, then
    \begin{align*}
        \Delta
        &\equiv A^{na + 2s\ell g_1(n-1)} B^{g_2(n-1) + 2s\ell g_4(n-1)} (C^{g_3(n-1)})^{B^{2s\ell g_4(n-1)}}\cdots (C^{g_3(2)})^{B^{2s\ell g_4(2)}} (C^{g_3(1)})^{B^{2s\ell g_4(1)}} (C^{g_3(0)})\\
        &\equiv A^{na + 2s\ell g_1(n-1)} B^{g_2(n-1) + 2s\ell g_4(n-1)} C^{g_5(n-1)}\mod Z_1(J),
    \end{align*}
    where $g_5(n-1) = \sum_{k=0}^{n-1} g_3(k)$. As $g_3(k) = 2s\ell\{\phi(b)\phi(ak) - a^2b^2\varphi(k+1) - (\phi(a)b-ac)b\phi(k+1)\} - abk + c$, then Lemma \ref{lemma.sums.J} produces
    \[
    g_5(n-1) = 2s\ell\{a^2\phi(b)\varphi(n) + \phi(a)\phi(b)\phi(n) - a^2b^2\sigma_2(1,n) - (\phi(a)b-ac)b\varphi(n+1)\} - ab\phi(n) + cn.
    \]
    Thus
    \[
    \Delta
    = A^{na + 2s\ell g_1(n-1)} B^{g_2(n-1) + 2s\ell g_4(n-1)} C^{g_5(n-1)} A^{\xi_4},
    \]
    were $A^{\xi_4}$ is the central element obtained by gathering all the above central commutators, i.e,
    \begin{align*}
        &A^{\xi_4} =\\
        &A^{4u\ell^2\sum\limits_{k=0}^{n-1}\{bc\phi(ka) - a\phi(c)k - g_1(k)(-abk+c)\} - 2u\ell^2\sum\limits_{k=0}^{n-1}bg_1(k) - \sum\limits_{k=0}^{n-1}\{\xi_1(ka,b) + \xi_1(2s\ell g_1(k),b-2s\ell a\phi(b)k)\}}\\
        &B^{4u\ell^2\sum\limits_{k=0}^{n-1}\{a^2\phi(b)bk^2 - \phi(b)g_1(k) + \phi(b)b\phi(ka)k - b^2g_1(k)k + a^2\phi(b)b\phi(k)k - a\phi(b)c\phi(k) - b\phi(-abk+c+1)k + g_4(k)(-abk+c)\}}\\
        &C^{\sum\limits_{k=0}^{n-1}\{- 2u\ell^2\phi(b)g_1(k) - 2us\ell^3a\phi(b)g_1(k)k\}}.
    \end{align*}
    Using properties of $\phi$ and $\varphi$, and applying Lemma \ref{lemma.sums.J}, gives\vspace{3pt}\\
    $\xi_1(2s\ell g_1(k),b-2s\ell a\phi(b)k)\equiv 2us\phi(b)g_1(k) + 2u^2a\phi(b)g_1(k)k\mod 4us$,\\
    $\xi_1(ak,b) = 4u\ell^2\{b\varphi(ak+1) + a\phi(b)\phi(ak)k - a\phi(b)k\} - 2u\ell^2\{7\phi(b)\phi(ak) + 3a^2\varphi(b)k^2\ + a\varphi(b)k\}$.\vspace{5pt}
    Using the above and $A^{2u}B^{2u}=1$, $A^{2us}=C^{2u}$, produces
    \begin{align*}
        \xi_4
        &= 2u\ell^2\sum_{k=0}^{n-1}\{a(\varphi(b)+2\phi(b)-2\phi(c))k + a^2(3\varphi(b)-2b\phi(b))k^2 + 2a\phi(b)c\phi(k) - 2a^2b\phi(b)\phi(k)k\\
        &\qquad\qquad\quad + (2bc+7\phi(b))\phi(ak) - 2\phi(b)(a+b)\phi(ak)k - 2b\varphi(ak+1) + 2b\phi(-abk+c+1)k\\
        &\qquad\qquad\quad + (2\phi(b)-2c-b)g_1(k) + 2b(a+b)g_1(k)k - 2cg_4(k) + 2abg_4(k)k\}.
    \end{align*}
    Appealing to Lemma \ref{lemma.sums.J} again, it is easy to see that\vspace{3pt}\\
    $
    \sum_{k=0}^{t} g_1(k) = a^2b\sigma_2(1,t+1) + (\phi(a)b-ac)\varphi(t+2),\\
    \sum_{k=0}^{t} g_1(k)k = a^2b(\sigma_3(t+1) + \sigma_1(1,t+1) - \varphi(t+1)) + (\phi(a)b-ac)(\sigma_1(1,t+1) + \varphi(t+2)),\\
    \sum_{k=0}^{t} g_4(k) = b(c-ab)\varphi(t+2) - 2ab^2\sigma_2(1,t+1),\\
    \sum_{k=0}^{t} g_4(k)k = b(c-ab)(\sigma_1(1,t+1) + \varphi(t+2)) - 2ab^2(\sigma_3(t+1) + \sigma_1(1,t+1) - \varphi(t+1)),\\
    \sum_{k=0}^{t}\phi(-abk+c+1)k = \sigma_1(ab,t+1) + (\phi(c+1) - abc)\phi(t+1) - ab(2c+1)\varphi(t+1).
    $\vspace{5pt}\\
    Thus
    \begin{align*}
        \xi_4
        &= 2u\ell^2\{a(\varphi(b)+2\phi(b)-2\phi(c))\phi(n)
        + a^2(3\varphi(b)-2b\phi(b))(2\varphi(n)+\phi(n))
        + 2a\phi(b)c\varphi(n)\\
        &\qquad\qquad - 2a^2b\phi(b)(\sigma_1(1,n)-\varphi(n))
        + (2bc+7\phi(b))(a^2\varphi(n) + \phi(a)\phi(n))\\
        &\qquad\qquad - 2\phi(b)(a+b)(\sigma_1(a,n) - a\varphi(n)) - 2b\sigma_2(a,n)\\
        &\qquad\qquad + (2\phi(b)-2c-b)(a^2b\sigma_2(1,n) + (\phi(a)b-ac)\varphi(n+1))\\
        &\qquad\qquad + 2b(a(\phi(a)b-ac)+b^2(\phi(a)-a^2))(\sigma_1(1,n) + \varphi(n+1))\\
        &\qquad\qquad + 2a^2b^2(a-b)(\sigma_3(n) + \sigma_1(1,n) - \varphi(n))\\
        &\qquad\qquad - 2c((c-ab)b\varphi(n+1) - 2ab^2\sigma_2(1,n))\\
        &\qquad\qquad + 2b(\sigma_1(ab,n) + (\phi(c+1) - abc)\phi(n) - ab(2c+1)\varphi(n))\}.
    \end{align*}
    This completes the proof for the case $n\geq 2$. We readily see that the result is valid when $n\in\{0,1\}$. Now, if $n<0$, as the exponent of $J$ is $4us$ in the case $m>1$ and 8 in the case $m=1$, and letting $q\in\N$ be such that $n + 30\cdot 6\cdot 4usq > 0$, we have that $(A^aB^bC^c)^n = (A^aB^bC^c)^{n+30\cdot 6\cdot 4usq}$. Then from \eqref{eq.comj.1new}, \eqref{eq.comj.10new}, the fact that $\sigma_3(n+30\cdot 6\cdot 4usq)\equiv\sigma_3(n)\mod 4us$, and the above, it follows that the result is also valid for $n<0$.
\end{proof}

\noindent{\bf Acknowledgements.} We thank Eamonn O'Brien for his help with Magma calculations.
We are indebted to the referee for a very careful reading of the manuscript that greatly improved
the presentation of the paper.


\begin{thebibliography}{RBMW}

    \small

		\bibitem[A]{A} M. A. Albar \emph{On Mennicke groups of deficiency zero I},
Internati. J. Math. $\&$ Math. Sci. 8 (1985) 821--824.

\bibitem[AA]{AA} M. A. Albar and A.-A. A. Al-Shuaibi \emph{On Mennicke groups of deficiency zero II},
Can. Math. Bull. 34 (1991) 289--293.

\bibitem[Al]{Al} D. Allcock \emph{Triangles of Baumslag-Solitar groups},
Can. J. Math. 64 (2012) 241--253.


\bibitem[AS]{AS} H. Abdolzadeh and R. Sabzchi
\emph{An infinite family of finite 2-groups with deficiency zero},
Int. J. Group Theory 6 (2017) 45--49.

\bibitem[AS2]{AS2} H. Abdolzadeh and R. Sabzchi
\emph{An infinite family of finite 3-groups with deficiency zero},
J. Algebra Appl. 18, No. 7, Article ID 1950121, 10 p. (2019).



\bibitem[BC]{BC} J.N.S Bidwell and M.J. Curran \emph{The automorphism group of a split metacyclic $p$-group},
Arch. Math. 87 (2006) 488--497.

\bibitem[BCP]{BCP} W. Bosma, J. Cannon, and C. Playoust \emph{The Magma algebra system. I. The user language},
J. Symbolic Comput. 24 (1997) 235–-265.



\bibitem[C]{C} M.J. Curran \emph{The automorphism group of a split metacyclic $2$-group},
Arch. Math. 89 (2007) 10--23.

\bibitem[C2]{C2} M.J. Curran \emph{The automorphism group of a nonsplit metacyclic $p$-group},
Arch. Math. 90 (2008) 483-–489.

\bibitem[CR]{CR} C.M. Campbell and E.F. Robertson
\emph{Remarks on a class of 2-generator groups of deficiency zero},
J. Aust. Math. Soc., Ser. A, 19 (1975) 297--305.

\bibitem[CR2]{CR2} C.M. Campbell and E.F. Robertson
\emph{A deficiency zero presentation for $SL(2,p)$},
Bull. Lond. Math. Soc. 12 (1980) 17--20.

\bibitem[CRT]{CRT} C.M. Campbell, E.F. Robertson, and R. M. Thomas
\emph{Finite groups of deficiency zero involving the Lucas numbers},
Proc. Edinb. Math. Soc., II. Ser., 33  (1990) 1--10.





\bibitem[GAP]{GAP4}
  The GAP~Group, \emph{GAP -- Groups, Algorithms, and Programming}, www.gap-system.org.
  







\bibitem[J]{J} D.L. Johnson \emph{A new class of 3-generator finite groups of deficiency zero},
J. Lond. Math. Soc., II. Ser., 19 (1979) 59--61.

\bibitem[Ja]{Ja} E. Jabara \emph{Gruppi fattorizzati da sottogruppi ciclici},
Rend. Semin. Mat. Univ. Padova 122 (2009) 65--84.

\bibitem[Jam]{Jam} A.-R. Jamali \emph{A new class of finite groups with three generators and three relations},
Algebra Colloq. 5 (1998) 465--469.


\bibitem[Jam2]{Jam2} A.-R. Jamali \emph{A further class of 3-generators, 3-relations finite groups},
Commun. Algebra 29 (2001) 879--887.

\bibitem[JR]{JR} D.L. Johnson and E. F. Robertson \emph{Finite groups of deficiency zero},
in Homological Group Theory (ed. C.T.C. Wall), Cambridge University Press, 1979.


\bibitem[K]{K} P.E. Kenne \emph{Some new efficient soluble groups},
Commun. Algebra 18 (1990) 2747--2753.



 \bibitem[M]{M} I.D Macdonald \emph{On a class of finitely presented groups},
Canad. J. Math 14 (1962) 602--613.



\bibitem[Ma]{Ma} I. Malinowska \emph{The automorphism group of a split metacyclic $2$-group and some groups of crossed homomorphisms},
Arch. Math. 93 (2009) 99--109.

\bibitem[Ma2]{Ma2} I. Malinowska \emph{On the structure of the automorphism group of a minimal nonabelian $p$-group (metacyclic case)},
Glas. Mat. 47(67) (2012) 153--164.


\bibitem[Me]{Me} J. Mennicke \emph{Einige endliche Gruppen mit drei Erzeugenden und drei Relationen},
Arch. Math. (Basel) 10 (1959) 409--418.


\bibitem[MS]{MS} A. Montoya Ocampo and F. Szechtman \emph{Structure of the Macdonald groups in one parameter}, J. Group Theory,
to appear, arXiv:2302.07079.

\bibitem[MS2]{MS2} A. Montoya Ocampo and F. Szechtman \emph{The automorphism group of finite $p$-groups associated to the Macdonald group}, arXiv:2303.06385.

\bibitem[MS3]{MS3} A. Montoya Ocampo and F. Szechtman \emph{On the isomorphism problem for certain $p$-groups}, Commun. Algebra, to appear,
arXiv:2308.03513.



\bibitem[P]{P} M. J. Post \emph{Finite three-generator groups with zero deficiency},
Commun. Algebra 6 (1978) 1289--1296.


\bibitem[PS]{PS} A. Previtali and F. Szechtman \emph{A study of the Wamsley group and its Sylow subgroups},
arXiv:2401.06585.

\bibitem[R]{R} E. F. Robertson \emph{A comment on finite nilpotent groups of deficiency zero},
Can. Math. Bull. 23 (1980) 313--316.








\bibitem[S]{S} E. Schenkman \emph{A factorization theorem for groups and Lie algebras},
Proc. Am. Math. Soc. 68 (1978) 149--152.


\bibitem[W]{W} J.W. Wamsley \emph{The deficiency of finite groups},
Ph.D. thesis, Univ. of Queensland (1969).


\bibitem[W2]{W2} J.W. Wamsley \emph{A class of three-generator, three-relation, finite groups},
Canad. J. Math. 22 (1070) 36--40.

\bibitem[W3]{W3} J.W. Wamsley \emph{A class of two generator two relation finite groups},
J. Aust. Math. Soc. 14 (1972) 38--40

\bibitem[W4]{W4} J.W. Wamsley \emph{Some finite groups with zero deficiency},
J. Aust. Math. Soc. 18 (1974) 73--75.

\bibitem[W5]{W5} J.W. Wamsley \emph{A class of finite groups with zero deficiency},
Proc. Edinb. Math. Soc., II. Ser. 19 (1974) 25--29.







\end{thebibliography}
\end{document}